\documentclass[final, 1p, times]{elsarticle}

\usepackage{hyperref}
\usepackage{graphicx}
\usepackage{amsmath}% amstex ?
\usepackage{amsthm}
\usepackage{amssymb,bbm}%

\DeclareMathAlphabet{\mathcal}{OMS}{cmsy}{m}{n}
\newcommand{\E}{\mathbb E}
\newcommand{\R}{\mathbb{R}}
\newcommand{\N}{\mathbb{N}}

\newcommand{\Z}{\mathbb{Z}}

\renewcommand{\P}{\mathbb{P}}

\newcommand{\ZZZ}{\mathcal{Z}}

\newcommand{\MMM}{\mathbf{M}}
\newcommand{\UUU}{\mathbf{U}}

\newcommand{\TTT}{\mathbb{T}}
\newcommand{\bfT}{\mathbf{T}}
\newcommand{\JJJ}{\mathbb{J}}
\newcommand{\III}{\mathbb{I}}
\newcommand{\VVV}{\mathbb{V}}

\newcommand{\BBB}{\mathbb{B}}
\newcommand{\Var}{\mathop{\mathrm{Var}}\nolimits}

\newcommand{\calR}{\mathcal{R}}

\newcommand{\calL}{\mathcal{L}}
\newcommand{\calX}{\mathcal{X}}

\newcommand{\calJ}{\mathcal{J}}

\newcommand{\eps}{\varepsilon}

\newcommand{\todistr}{\overset{d}{\underset{n\to\infty}\longrightarrow}}

\newcommand{\bsl}{\backslash}
\newcommand{\ind}{\mathbbm{1}}

\theoremstyle{plain}
\newtheorem{theorem}{Theorem}[section]
\newtheorem{lemma}[theorem]{Lemma}

\newtheorem{proposition}[theorem]{Proposition}

\theoremstyle{definition}

\theoremstyle{remark}
\newtheorem{remark}[theorem]{Remark}

\newcommand{\xxx}{}
\journal{Stochastic Processes and their Applications}

\begin{document}

\begin{frontmatter}
\title{Limiting distribution for the maximal standardized increment of a random walk}
%\titlerunning {Limiting distribution for the maximal standardized increment of a random walk}

\author[au1]{Zakhar Kabluchko\corref{cor1}}
\ead{zakhar.kabluchko@uni-ulm.de}
\cortext[cor1]{Corresponding author}

\author[au2]{Yizao Wang}
\ead{yizao.wang@uc.edu}

\address[au1]{Institute of Stochastics,
Ulm University,
Helmholtzstr.\ 18,
89069 Ulm, Germany}

\address[au2]{
Department of Mathematical Sciences,
University of Cincinnati,
4302 French Hall,
2815 Commons Way,
Cincinnati, OH, 45221-0025}

%\institute{Zakhar Kabluchko \at Institute of Stochastics,
%Ulm University,
%Helmholtzstr.\ 18,
%89069 Ulm, Germany
%\email{zakhar.kabluchko@uni-ulm.de}
%Yizao Wang   \at
%Department of Mathematical Sciences,
%University of Cincinnati,
%4302 French Hall,
%2815 Commons Way,
%Cincinnati, OH, 45221-0025
%\email{yizao.wang@uc.edu}
%            }


%
%\subjclass[2010]{Primary, 60G50, 60G70; secondary, 	60F10, 	60F05}

\begin{abstract}
Let $X_1,X_2,\ldots$ be independent identically distributed (i.i.d.)\ random variables with $\E X_k=0$, $\Var X_k=1$. Suppose that $\varphi(t):=\log \E e^{t X_k}<\infty$ for all $t>-\sigma_0$ and some $\sigma_0>0$.  Let $S_k=X_1+\ldots+X_k$ and $S_0=0$. We are interested in the limiting distribution of the \textit{multiscale scan statistic}
$$
\MMM_n=\max_{0\leq i <j\leq n} \frac{S_j-S_i}{\sqrt{j-i}}.
$$
We prove that for an appropriate normalizing sequence $a_n$, the random variable $\MMM_n^2-a_n$ converges to the Gumbel extreme-value law  $\exp\{-e^{-c x}\}$. The behavior of $\MMM_n$ depends strongly on the distribution of the $X_k$'s. We distinguish between four cases.  In the \textit{superlogarithmic} case we assume that $\varphi(t)<t^2/2$ for every $t>0$. In this case, we show that the main contribution to $\MMM_n$ comes from the intervals $(i,j)$ having  length $l:=j-i$ of order $a(\log n)^{p}$, $a>0$, where $p=q/(q-2)$ and $q\in\{3,4,\ldots\}$ is the order of the first non-vanishing cumulant of $X_1$ (not counting the variance). In the \textit{logarithmic} case we assume that the function $\psi(t):=2\varphi(t)/t^2$ attains its maximum $m_*>1$ at some unique point $t=t_*\in (0,\infty)$.
In this case, we show that the main contribution to $\MMM_n$ comes from the intervals $(i,j)$ of length $d_*\log n+a\sqrt{\log n}$, $a\in\R$, where $d_*=1/\varphi(t_*)>0$. In the \textit{sublogarithmic} case we assume that the tail of $X_k$ is heavier than $\exp\{-x^{2-\eps}\}$, for some $\eps>0$.  In this case, the main contribution to $\MMM_n$ comes from the intervals of length $o(\log n)$ and in fact, under regularity assumptions, from the intervals of length $1$. In the remaining, fourth case, the $X_k$'s are \textit{Gaussian}. This case has been studied earlier in the literature. The main contribution comes from intervals of length $a\log n$, $a>0$.  We argue that our results cover most interesting distributions with light tails.
The proofs are based on the precise asymptotic estimates for  large and moderate deviation probabilities for sums of i.i.d.\ random variables due to Cram\'er, Bahadur, Ranga Rao, Petrov and others, and a careful extreme value analysis of the random field of standardized increments by the double sum method.
\end{abstract}


\end{frontmatter}

\section{Introduction and statement of results}\label{sec:intro}
\subsection{Introduction}\label{subsec:intro}
Suppose we are given a long sequence of observations.  The observations are assumed to be independent identically distributed (i.i.d.)\ random variables with zero mean and unit variance, except, possibly, for a short interval, where the observations have positive mean.  This interval may be interpreted as a signal in an i.i.d.\ noise. The question is how to decide whether a signal is present and if yes, how to locate it.  A natural approach is to  build a \textit{multiscale scan statistic}. For every interval we compute the sum of the observations in this interval divided by the square root of the length of the interval. Large values of this normalized sum indicate the presence of a signal. Since no a priori knowledge about the location and length of the interval containing the signal is available, we take the maximum of such normalized sums over all possible intervals of all possible lengths.  Scan statistics with windows of \textit{fixed size} have been much studied; see, e.g.,\ \cite{glaz_book1,glaz_book2}. A large class of limit theorems dealing with fixed window size are the Erd\"os--R\'enyi--Shepp laws; see, e.g.,~\cite{csoergoe_book,csoergoe,deheuvels,deheuvels1,deheuvels2,book}.  The scan statistic we are interested in is built using windows of \textit{all possible sizes}.  In order to use this statistic for testing purposes we need to know its asymptotic distribution under the null hypothesis.

%In more precise terms, the problem can be stated as follows.
We arrive at the following problem. Let $X_1,X_2,\ldots$ be i.i.d.\ non-degenerate random variables with $\E X_k=0$, $\Var X_k=1$. Consider a random walk given by $S_k=X_1+\ldots+X_k$, $k\in\N$, and $S_0=0$.   For $n\in\N$ define the multiscale scan statistic $\MMM_n$ by
\begin{equation}\label{eq:def_Ln}
\MMM_n=\max_{0\leq i <j\leq n} \frac{S_j-S_i}{\sqrt{j-i}}.
\end{equation}
Following results on the asymptotic behavior of $\MMM_n$ as $n\to\infty$ are known. For random variables with finite exponential moments, \citet{shao}, confirming and extending a conjecture of R\'ev\'esz~\cite{revesz_book},  proved that
\begin{equation}\label{eq:shao}
\lim_{n\to\infty} \frac{\MMM_n}{\sqrt{2\log n}} = \sqrt{m_*} \;\;\; \text{a.s.}
\end{equation}
Here, $m_*\in [1,\infty]$ is a constant determined explicitly in terms of the distribution of $X_1$.
Shao's proof has been considerably simplified by~\citet{steinebach}; see also~\cite{kabluchko_munk1} for a multidimensional generalization. %For an exact convergence rate in Shao's law, see~\cite{kabluchko_munk2}.
%Concerning the limiting distribution of $\MMM_n$.
This describes the a.s.\ rate of growth of $\MMM_n$. But what about the limiting distribution?
In the case when $X_1,X_2,\ldots$ are i.i.d.\ standard Gaussian, \citet{siegmund_venkatraman} showed that for all $\tau\in\R$,
\begin{equation}\label{eq:siegmund_venkatraman}
\lim_{n\to\infty}\P\left[\MMM_n\leq \sqrt{2\log n}+ \frac{\frac 12 \log \log n +\log \frac{H}{2\sqrt {\pi}}+\tau}{\sqrt {2\log n}} \right]=\exp\{-e^{-\tau}\}.
\end{equation}
Here, $H>0$ is some explicit constant.  The distribution on the right-hand side is the Gumbel extreme-value law.  An independent proof of the same result was given in~\cite{kabluchko_unpub_07}. It was shown in~\cite{kabluchko_diss,kabluchko_unpub_07} that a result similar to~\eqref{eq:siegmund_venkatraman}, but with a different normalization, holds if we replace the Gaussian random walk by a Brownian motion. Generalizations of both results to the multidimensional setting with intervals replaced by cubes or rectangles, have been obtained in~\cite{kabluchko_spa_10}. Similar problem for a totally skewed $\alpha$-stable L\'evy process has been considered in~\cite{kabluchko_diss}. In the case when $X_1$ has regularly varying right tail, limit Fr\'echet distribution for $\MMM_n$ has been obtained by~\citet{mikosch1}; see also~\citet{mikosch2}.

Apart from these  special cases nothing has been known about the limiting distribution of $\MMM_n$.  Our aim is to settle this problem  for a broad class of random variables with light tails. It turns out that the behavior of $\MMM_n$ depends heavily on some fine properties of the distribution of $X_1$.
%In the first two cases, called the \textit{superlogarithmic} and the \textit{logarithmic} case,
We assume that for some $\sigma_0>0$,
\begin{equation}\label{eq:varphi_def}
\varphi(t) := \log \E e^{t X_1}<\infty \text{ for all } t\geq -\sigma_0.
\end{equation}
The function $\varphi$ (called the cumulant generating function of $X_1$) is strictly increasing on $[0,\infty)$, strictly convex, infinitely differentiable, and vanishes at $0$.

We will consider four cases depending on where the supremum of the function
\begin{equation}\label{eq:def_psi}
\psi(t):=\frac{\varphi(t)}{t^2/2}, \;\;\; t > 0,
\end{equation}
is attained. The constant $m_*$ in Shao's result~\eqref{eq:shao} is determined by $m_*=\sup_{t>0}\psi(t)$. Note that $\lim_{t\downarrow 0} \psi(t)=1$ since $\varphi(t)\sim t^2/2$ as $t\downarrow 0$. Hence, $m_*\geq 1$.  If $X_1$ is standard Gaussian, we even have $\psi(t) = 1$ identically, for all $t\in\R$.   Our four cases can be roughly described as follows, see Figure~\ref{fig:four_cases}:
\begin{enumerate}
\item \textit{Gaussian case:} $\psi(t)=1$ for all $t\in\R$.
\item \textit{Superlogarithmic case:}
%the supremum of $\psi$, equal to $1$, is attained as $t\downarrow 0$.
the supremum $m_*=1$  is attained as $t\downarrow 0$.
\item \textit{Logarithmic case:}
%the supremum of $\psi$, which is strictly larger than $1$, is attained at some finite $t=t_*\in (0,\infty)$.
the supremum $m_*>1$  is attained at some $t=t_*\in (0,\infty)$.
\item \textit{Sublogarithmic case:}
$m_*=+\infty$.
%the supremum of $\psi$  is infinite.  %is attained as $t\to\infty$.
\end{enumerate}
Since the Gaussian case has been fully analyzed in~\cite{siegmund_venkatraman,kabluchko_unpub_07,kabluchko_spa_10}, we concentrate on the remaining three cases.
%We will consider three cases depending on whether the supremum of the function $\frac{\varphi(t)}{t^2}$, $t\geq 0$, is attained at $0$, some finite positive number $t_*$ or at $+\infty$.
%The main difference between the three cases is the set length of the intervals which make the main contribution to $\MMM_n$.
Let us explain the difference between the cases. The definition of $\MMM_n$ involves a maximum taken over intervals $(i,j)$ of different lengths $l:=j-i$. It turns out that different lengths make different contributions to $\MMM_n$. In all three cases we will single out some family of lengths which are \textit{optimal} in the sense that the contribution of all other lengths to $\MMM_n$ is asymptotically negligible. We will show that the optimal lengths are given as follows:
\begin{enumerate}
\item \textit{Gaussian case:} $l=a\log n$, $a>0$.
\item \textit{Superlogarithmic case:} $l=a\log^p n$, $a>0$, where $p>1$.
\item \textit{Logarithmic case:} $l=d_*\log n + a\sqrt {\log n}$, $a\in\R$, where $d_*>0$.
\item \textit{Sublogarithmic case:} $l=o(\log n)$, and, under more assumptions, $l=1$.
\end{enumerate}
Here, $p>1$ and $d_*>0$ are parameters depending on the distribution of $X_1$.
To give  exact meaning to these statements we will analyze the random variable $\MMM_n(h_1,h_2)$ obtained by restricting the lengths over which the maximum is taken to some range $[h_1,h_2]$. Namely, for  $0\leq h_1<h_2$,  define
\begin{equation}\label{eq:def_Ln_A1A2}
\MMM_n(h_1,h_2)=\max_{\substack{0\leq i <j\leq n\\h_1\leq j-i\leq h_2}} \frac{S_j-S_i}{\sqrt{j-i}}.
\end{equation}
Then, the statement that in the superlogarithmic case the lengths $a\log^p n$, $a>0$, are optimal, means that
$$
\lim_{A\to +\infty} \limsup_{n\to\infty} \P[\MMM_n = \MMM_n(A^{-1} \log^p n, A\log^p n)] = 1.
$$
As we will show below, analogous statements hold in all four cases.

We are now ready to state our results on the limiting distribution of the multiscale scan statistic $\MMM_n$. In fact, the results will be stated in terms of $\MMM_n^2$ because this greatly simplifies the notation. It is easy to switch back to $\MMM_n$; see Section~\ref{subsubsec:Mn_Mn_square}.

%%%%%%%%%%%%%%%%%FIGURE 1: The graph of varphi %%%%%%%%%%%%%%%%%%%%%%
\begin{figure}[t]
\begin{center}
\includegraphics[height=0.4\textwidth, width=0.8\textwidth]{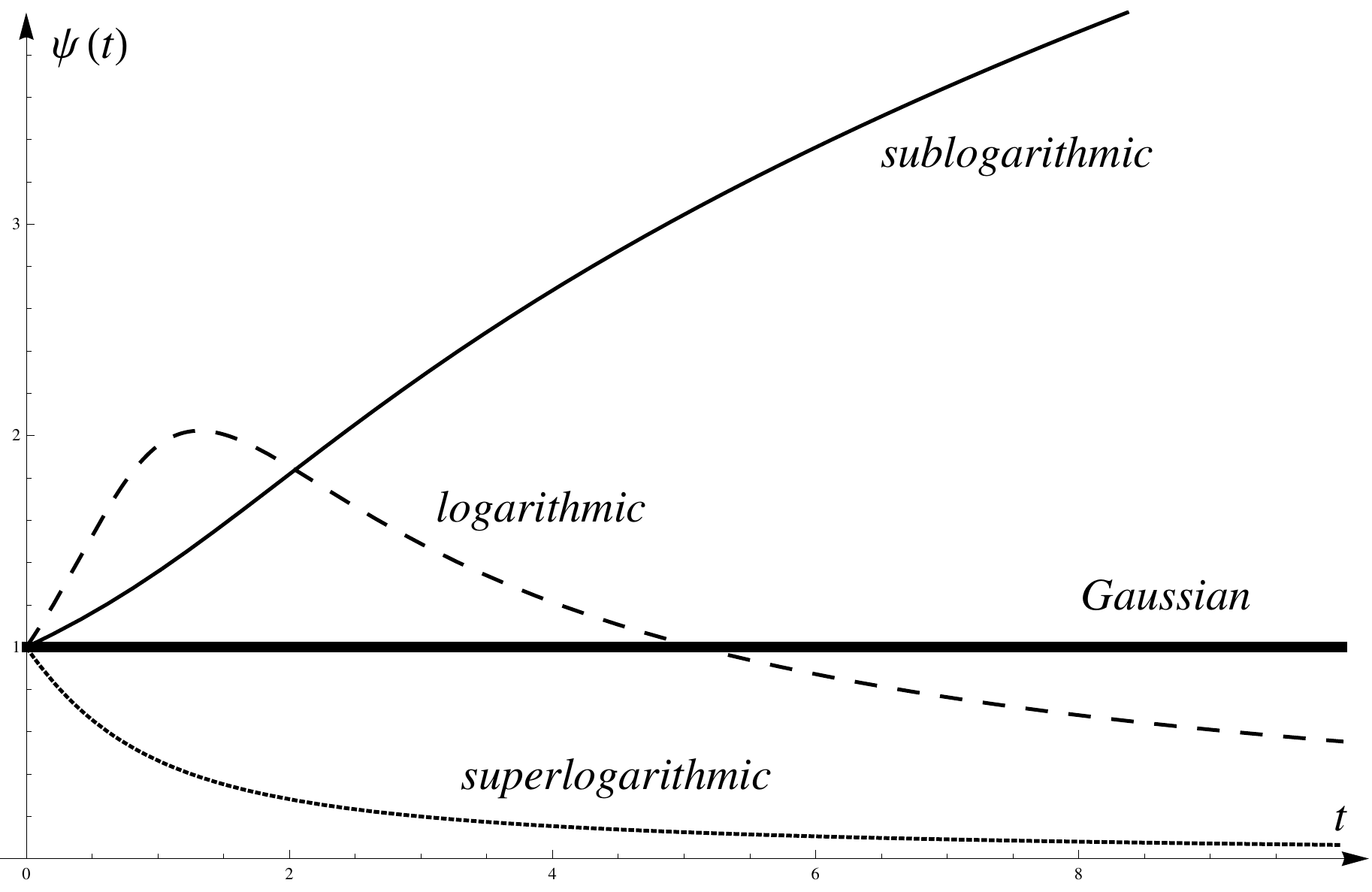}
\end{center}
\caption
{
The graph of $\psi(t)=2\varphi(t)/t^2$ in the Gaussian case (bold), superlogarithmic case (dotted), logarithmic case (dashed), sublogarithmic case (solid).
} \label{fig:four_cases}
\end{figure}
%%%%%%%%%%%%%%%%%%%%%%%%%%%%%%%%%%%%%%%%%%%%%%%%%%%%%%%%

%Let us stress that the normal distribution has $\psi(t)\equiv 1$ and hence does \textit{not} belong to any of these cases.

%Let us stress that the normal distribution having $\varphi(t)=\frac{t^2}{2}$ belongs to none of the three cases (since the supremum is attained for every $t>0$).

%We agree that $t_{\infty}$ is chosen to be maximal with the above property.
%Note that $\varphi$ is strictly  convex, infinitely differentiable function.
%The trivial case in which $X_1=0$ a.s.\ is always excluded from consideration.
%With this convention, $\varphi$ is even strictly convex.

\subsection{The superlogarithmic case} \label{subsec:subgauss}
Here we consider random variables which are in some sense dominated by the Gaussian distribution.
We assume that for all $\eps>0$,
%Recall that a random variable $X_1$ is called subgaussian if $\varphi(t)\leq \frac{t^2}{2}$ for all $t\in\R$; see~\cite{buldygin_kozachenko_book}.  %Here, we make a slightly stronger assumption.
%We make a closely related assumption, namely  that for all $\eps>0$,
\begin{equation}\label{eq:varphi_subgauss}
\sup_{t\geq \eps} \frac{\varphi(t)}{t^2/2}<1.
\end{equation}
Equivalently, $\psi(t)<1$ for every $t>0$, and $\limsup_{t\to\infty}\psi(t)<1$. A closely related notion is the subgaussianity; see~\cite{buldygin_kozachenko_book}.
The first non-zero term in the Taylor expansion of $\varphi(t)$ at $0$ is $t^2/2$ since we assume that $\E X_k=0$, $\Var X_k=1$. Of crucial importance will be the second non-zero term in the Taylor expansion of $\varphi(t)$. We have, for some $q\in\{3,4,\ldots\}$ and $\kappa>0$,
\begin{equation}\label{eq:varphi_taylor}
\varphi(t)=\frac {t^2}{2}-\kappa t^{q}+o(t^{q}), \;\;\; t\downarrow 0.
\end{equation}
Thus, $q$ is the order of the first non-zero cumulant of $X_1$, not counting the variance. The most common value of $q$ is $3$, however, for symmetric distributions the third cumulant vanishes and we typically have $q=4$.  Note that the coefficient $\kappa$ cannot be negative, since otherwise~\eqref{eq:varphi_subgauss} would be violated for sufficiently small $\eps>0$.
%Recall that the coefficients in the Taylor series of $\varphi(t)$ are called cumulants. The third cumulant is called the skewness, whereas the fourth one is called curtosis.  Usually, we have $q=3$. However, for symmetric distributions (or more generally, for distributions with non-zero skewness), but non-zero curtosis, we have $q=4$. Note also that the normal distribution does not satisfy~\eqref{eq:varphi_subgauss}, but would correspond formally to the case $q=\infty$ since all its cumulants starting with the third one vanish.
\begin{theorem}\label{theo:main_subgauss}
Let $X_1,X_2,\ldots$ be i.i.d.\ random variables with $\E X_k=0$, $\Var X_k=1$. Suppose that conditions~\eqref{eq:varphi_def}, \eqref{eq:varphi_subgauss}, \eqref{eq:varphi_taylor} are satisfied. Then, for all $\tau\in\R$,
\begin{equation}
\lim_{n\to\infty}\P\left[\frac 12 \MMM_n^2\leq  \log (n\log^{\frac 12 \cdot \frac  {q-6}{q-2}}n)+\tau\right]
=
\exp\big\{-\Lambda_{q,\kappa} e^{-\tau}\big\},
\end{equation}
where $\Lambda_{q,\kappa}=\frac{1}{\sqrt {\pi}} \Gamma(\frac{q}{q-2}) (2\kappa)^{\frac 2 {q-2}} $.
\end{theorem}
It turns out that in the superlogarithmic case, the main contribution to $\MMM_n$ is done by intervals with length $a \log^p n$, where $a\in (0,\infty)$, and
\begin{equation}\label{eq:def_p}
p=\frac{q}{q-2}\in \left\{3,2, \frac 53, \frac 32, \frac 75, \ldots\right\}.
\end{equation}
%The most common values are $p=3$ and $p=2$.
Moreover, we will even prove that the ``intensity'' with which the length $a\log^p n$ contributes to $\MMM_n$ is given by some explicit function $\Lambda_{q,\kappa}(a)$. Namely, we have the following result.
\begin{theorem}\label{theo:main_subgauss_scales}
%Assume that~\eqref{eq:varphi_subgauss} and~\eqref{eq:varphi_taylor} hold.
%Let $d_n(a)=a\log^{\frac{q}{q-2}} n$, $a>0$, and fix arbitrary $0<A_1<A_2$.
Fix arbitrary $0<A_1<A_2$. Define $l_n^-=A_1\log^{p} n$ and $l_n^+=A_2\log^{p} n$.
Under the same assumptions as in Theorem~\ref{theo:main_subgauss}, for every $\tau\in\R$,
\begin{equation}
\lim_{n\to\infty}\P\left[\frac 12 \MMM_{n}^2(l_n^-,l_n^+) \leq  \log (n\log^{\frac 12 \cdot \frac {q-6}{q-2}}n)+\tau\right]
=
\exp\left\{ - e^{-\tau} \int_{A_1}^{A_2} \Lambda_{q,\kappa}(a) da \right\},
%e^{ - e^{-\tau} \int_{A_1}^{A_2} \Lambda_{q,\kappa}(a) da },
\end{equation}
where $\Lambda_{q,\kappa}:(0,\infty)\to (0,\infty)$ is a function given by
\begin{equation}\label{eq:def_Lambda_q_kappa}
\Lambda_{q,\kappa}(a)=\frac 1 {2\sqrt {\pi} a^{2}} \exp\{-\kappa 2^{\frac q2}a^{-\frac{q-2}{2}}\}. % \;\;\;a>0.
\end{equation}
\end{theorem}
Note that $\Lambda_{q,\kappa}=\int_{0}^{\infty}\Lambda_{q,\kappa}(a)da$, so that formally we can  obtain Theorem~\ref{theo:main_subgauss} from Theorem~\ref{theo:main_subgauss_scales} by taking $A_1=0$, $A_2=\infty$.
%We can view $\Lambda_{q,\kappa}(a)$ as an intensity measuring the contribution of intervals having length  $(a+o(1))\log^{p} n$.
Note that $\Lambda_{q,\kappa}(a)\to 0$ as $a\downarrow 0$ or $a\uparrow \infty$. This means that too small and too large intervals make small contributions to $\MMM_n$. The unique maximum of the function $a\mapsto \Lambda_{q,\kappa}(a)$  is attained at
$$
a_*= 2^{\frac{q-4}{q-2}}\kappa^{\frac 2 {q-2}}(q-2)^{\frac 2 {q-2}}.
$$
Thus, the largest contribution to $\MMM_n$ comes from the intervals of length $(a_*+o(1))\log^p n$.

\subsection{The logarithmic case}\label{subsec:supergauss}
We assume that~\eqref{eq:varphi_def} holds and there is $t_*>0$ such that
\begin{equation}\label{eq:varphi_supergauss}
m_*:=\frac{\varphi(t_*)}{t_*^2/2} > 1.
\end{equation}
Moreover, we assume that $t_*$ is the unique point of maximum of $\psi(t)=2\varphi(t)/t^2$ in the following uniform sense: for every $\eps>0$,
\begin{equation}\label{eq:t*_unique}
\sup_{0<t<t_*-\eps}  \frac{\varphi(t)}{t^2/2} < m_*
\;\;\;\text{and}\;\;\;
\sup_{t>t_*+\eps}  \frac{\varphi(t)}{t^2/2} < m_*.
\end{equation}
Equivalently, $\psi(t)<m_*$ for all $t>0$, $t\neq t_*$, and $\limsup_{t\to\infty}\psi(t)<m_*$. Recall that a random variable $X_1$ is lattice if there are $a>0$, $b\in\R$ such that $X_1\in a\Z+b$ with probability $1$. Otherwise, $X_1$ is called non-lattice.
\begin{theorem}\label{theo:main_supergauss}
Let $X_1,X_2,\ldots$ be i.i.d.\ random variables with $\E X_k=0$, $\Var X_k=1$. Suppose that conditions~\eqref{eq:varphi_def}, \eqref{eq:varphi_supergauss}, \eqref{eq:t*_unique} are satisfied and the distribution of $X_1$ is non-lattice.
Then, for every $\tau\in\R$,
\begin{equation}\label{eq:main_supergauss}
\lim_{n\to\infty}\P\left[\frac 12 \MMM_n^2\leq m_*(\log n+\tau)\right]
=
\exp\left\{- \Theta_*  e^{-\tau} \right\},
\end{equation}
where $\Theta_*= \frac{\sqrt{m_*} H_*^2}{\sqrt 2\beta_* \sigma_*}$, $\sigma_*^2=\varphi''(t_*)$  and
%$H_*$ is given by~\eqref{eq:H*}.
\begin{align}
\beta_*^2 &=\frac {(\varphi'(t_*))^4}{8 m_*} \left(\frac 1 {\sigma_*^{2}}-\frac 1 {m_*}\right)>0,\label{eq:beta*_H*1}\\
H_* &= \lim_{B\to\infty} \frac 1B \E \left[ \max _{k=0,\ldots,B} e^{t_* S_k-k\varphi(t_*)}\right]\in (0,1). \label{eq:beta*_H*2}
\end{align}
%Here $s_*>0$ is the unique solution of $\frac{I(s_*)}{s_*^2/2}= \frac 1 {m_*}$.
\end{theorem}
The  positivity of $\beta_*^2$  will be established in Lemma~\ref{lem:one_point_super_gauss}. More explicit expression for $H_*$ will be given in Section~\ref{subsec:pickands_type_const} below. We believe that in the lattice case the convergence in Theorem~\ref{theo:main_supergauss} breaks down. (A similar phenomenon was observed in~\cite{komlos} for scan statistics with fixed window size). However, we still have tightness.
\begin{theorem}\label{theo:main_supergauss_lattice}
Let $X_1,X_2,\ldots$ be i.i.d.\ random variables with $\E X_k=0$, $\Var X_k=1$. Assume that conditions~\eqref{eq:varphi_def}, \eqref{eq:varphi_supergauss}, \eqref{eq:t*_unique} hold.
%hold and the distribution of $X_1$ is lattice.
Then,  the sequence of random variables $\MMM_n^2- 2m_*\log n$, $n\in\N$, is tight.
\end{theorem}
In the next theorem we compute the contribution of different lengths to $\MMM_n$. Let $d_*=1/\varphi(t_*)$. We will show that only intervals whose length differs from $d_*\log n$ by a quantity of order $\sqrt {\log n}$ are relevant. Recall that $\MMM_n(h_1,h_2)$ was defined in~\eqref{eq:def_Ln_A1A2}.
\begin{theorem}\label{theo:main_supergauss_scales}
%Suppose that~\eqref{eq:varphi_supergauss} and~\eqref{eq:t*_unique} hold. Assume that the distribution of $X_1$ is non-lattice.
Fix arbitrary $A_1<A_2$. Define $l_n^-=d_*\log n + A_1\sqrt{\log n}$ and $l_n^+=d_*\log n + A_2\sqrt{\log n}$.
Under the same assumptions as in Theorem~\ref{theo:main_supergauss}, for every $\tau\in\R$,
\begin{equation}
\lim_{n\to\infty}\P\left[\frac 12 \MMM_n^2(l_n^-, l_n^+) \leq  m_*(\log n+\tau)\right]=\exp\left\{ - e^{-\tau} \int_{A_1}^{A_2} \Theta(a) da\right\}.
\end{equation}
Here, $\Theta:\R\to (0,\infty)$ is a function given by $\Theta(a)=\frac{\sqrt{m_*}H_*^2}{2\sqrt {\pi} \sigma_*} e^{-\frac{\beta_*^2 a^2}2}$, where $\beta_*$ and $H_*$ are as in~\eqref{eq:beta*_H*1} and~\eqref{eq:beta*_H*2}.
\end{theorem}
Note that $\Theta_*=\int_{-\infty}^{\infty}\Theta(a)da$, so that we can obtain (at least formally) Theorem~\ref{theo:main_supergauss} from Theorem~\ref{theo:main_supergauss_scales} by taking $A_1=-\infty$, $A_2=\infty$. Since $\Theta(a)$ attains its maximum at $a=0$, the maximal contribution to $\MMM_n$ comes from  the intervals with length $d_*\log n+ o(\sqrt {\log n})$.

\subsection{The sublogarithmic case}\label{subsec:exponential}
In this case we consider random variables whose right tail is heavier than the standard Gaussian tail. In this case only intervals of length $o(\log n)$ make contribution to $\MMM_n$, as the next theorem shows.
\begin{theorem}\label{theo:main_exp}
Let $X_1,X_2,\ldots$ be i.i.d.\ random variables with $\E X_k=0$, $\Var X_k=1$ and such that $\varphi(t)=\log \E e^{tX_1}$ is finite on $[0,t_0)$, for some $t_0>0$. Assume that for some $\alpha<2$ we have $\P[X_1>x]>e^{-x^{\alpha}}$, for sufficiently large $x$.   Then, for every $a>0$,
\begin{equation}
\lim_{n\to\infty}\P[\MMM_n=\MMM_n(1, a\log n)] =1.
\end{equation}
\end{theorem}
Under some regularity assumptions on the tail of $X_1$ it is possible to show that only intervals of length $1$ (that is, only individual observations) contribute to $\MMM_n$. In this case, the study of $\MMM_n$ is equivalent to the study of the maximum $\UUU_n=\max\{X_1,\ldots,X_n\}$.  We assume that for some $\alpha\in [1,2)$ and $D>0$,
\begin{equation}\label{eq:exponenntial_tail}
\lim_{x\to +\infty} \frac{1}{x^{\alpha}} \log \P [X_1>x] = -D. %\text{ as } x\to +\infty.
\end{equation}
\begin{theorem}\label{theo:main_exp_regular}
Let $X_1,X_2,\ldots$ be i.i.d.\ random variables with $\E X_k=0$, $\Var X_k=1$ and such that~\eqref{eq:exponenntial_tail} holds. Then,
$
\lim_{n\to\infty}\P[\MMM_n=\UUU_n] =1.
$
\end{theorem}
%The sublogarothmic case is in some sense trivial: $\MMM_n$ behaves in the same way as the maximum of the individual observations $X_1,\ldots,X_n$.
The above results show that the  square-root normalization we used in~\eqref{eq:def_Ln} is not very natural in the sublogarithmic case.  See~\cite{steinebach,lanzinger_stadtmueller,siegmund_yakir,mikosch1,mikosch2} for other types of normalization which can be used in the case of sublogarithmic (or even heavier) tails.

\subsection{Examples}\label{subsec:examples}
In this section we show that most classical families of distributions considered in the probability theory belong to one of the four cases considered above. %The only exception is the normal distribution which belongs to neither of the assumptions. The case of the normal distribution has been completely analyzed in~\cite{siegmund_venkatraman}, \cite{kabluchko_unpub_07}, \cite{kabluchko_spa_10}.
%In all examples considered below

%\smallskip
%\noindent\textit{Symmetric Bernoulli.}
\subsubsection{Symmetric Bernoulli}
Let $X_1,X_2,\ldots$ be i.i.d.\ with symmetric Bernoulli distribution, that is $\P[X_k=\pm 1]=1/2$. We have
$$
\varphi(t)=\log \cosh (t)=\frac{t^2}{2}-\frac{t^4}{12}+o(t^4), \;\;\; t\to 0.
$$
We are in the superlogarithmic case. Indeed, all coefficients of the Taylor series
$$
\cosh t-e^{t^2/2}=\frac 12 \sum_{k=2}^{\infty} t^{2k}\left(\frac{1}{(2k)!}-\frac{1}{2^kk!}\right)
$$
are negative. This shows that $\varphi(t)<t^2/2$ for every $t>0$. Since we also have $\lim_{t\to\infty} \varphi(t)/t^2=0$, it follows that condition~\eqref{eq:varphi_subgauss} is fulfilled. We are in the superlogarithmic case with $q=4$. The optimal lengths are $a\log^2 n$, $a>0$.

%%%%%%%%%%%%%%%%%FIGURE 1: The graph of varphi %%%%%%%%%%%%%%%%%%%%%%
\begin{figure}[t]
\begin{center}
\includegraphics[height=0.4\textwidth, width=0.8\textwidth]{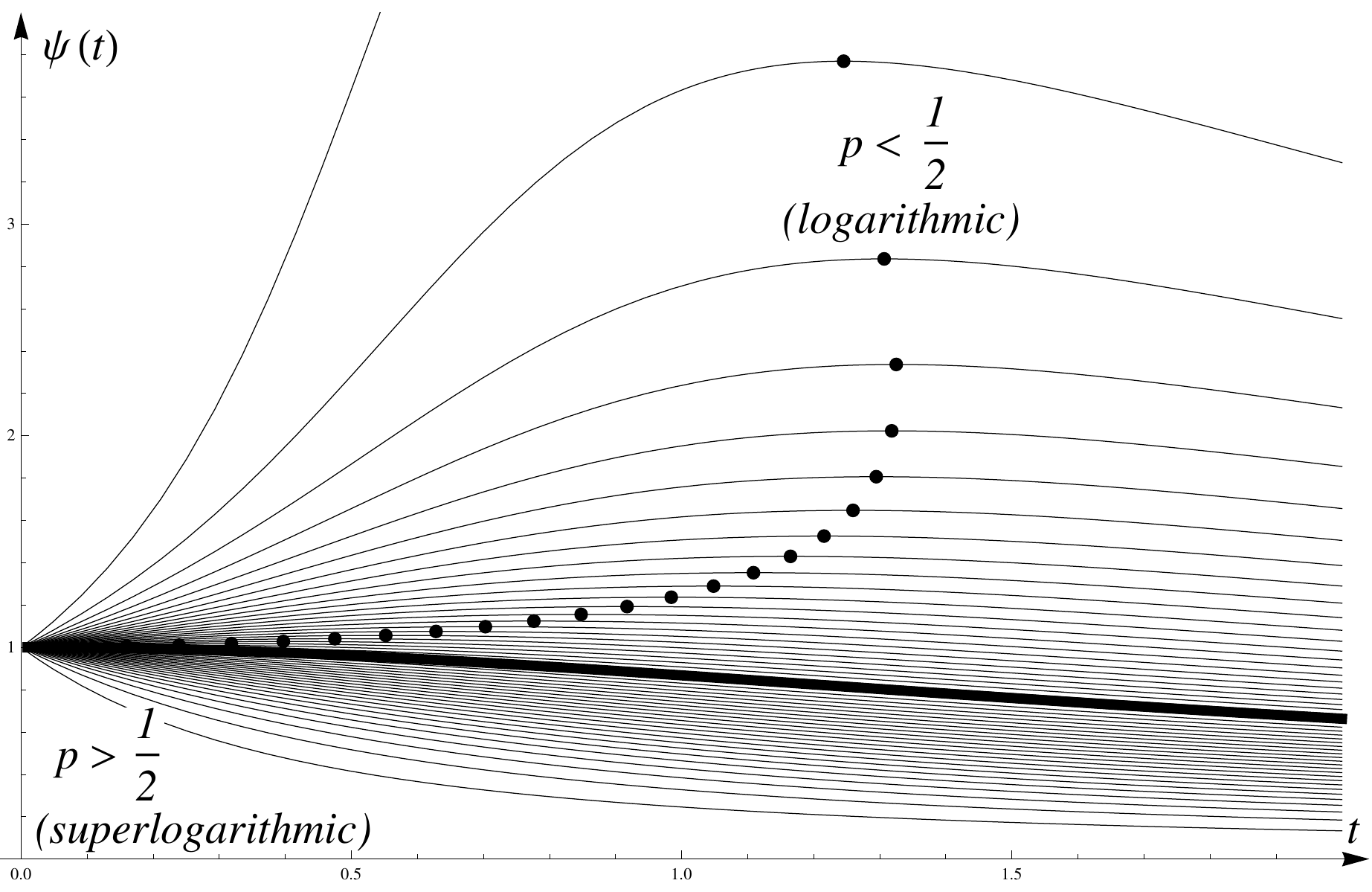}
\end{center}
\caption
{
The graph of $\psi(t)=2\varphi(t)/t^2$ for the (generally, non-symmetric) Bernoulli distributions in dependence on the parameter $p=1/50,\ldots,49/50$. The bold black boundary corresponds to $p=1/2$ (superlogarithmic case). The curves below the boundary correspond to $p>1/2$ (superlogarithmic case), the curves above the boundary correspond to $p<1/2$ (logarithmic case). On each curve above the boundary the point of maximum $(t_*,m_*)$ is shown.
} \label{fig:psi}
\end{figure}
%%%%%%%%%%%%%%%%%%%%%%%%%%%%%%%%%%%%%%%%%%%%%%%%%%%%%%%
%\smallskip
%\noindent \textit{

\subsubsection{Non-symmetric Bernoulli}\label{subsubsec:non_symm_bern}
Fix $p\in (0,1)$ and consider i.i.d.\ Bernoulli random variables $Y_1,Y_2,\ldots$ with $\P[Y_k=+1]=p$ and $\P[Y_k=-1]=1-p$. Consider also the normalized random variables
$$
X_k=\frac 1{\sigma} (Y_k-(2p-1)),\;\;\; \sigma=\sqrt{4p(1-p)}.
$$
Then, $\E X_k=0$ and $\Var X_k=1$. The cumulant generating function of $X_k$ is given by
\begin{equation}\label{eq:non_symm_bern_varphi}
\varphi(t)=-\frac{2p t}{\sigma}+\log (1+pe^{\frac{2t}{\sigma}}-p)=\frac{t^2}{2}+ \frac{1-2p}{3\sigma}t^3+o(t^3), \;\;\; t\to 0.
\end{equation}
The graph of the function $\psi(t)=2\varphi(t)/t^2$ in dependence on the parameter $p$ is shown in Figure~\ref{fig:psi}. As already shown above, for $p=1/2$ we are in the superlogarithmic case with $q=4$, and the optimal lengths are $a\log^2 n$, $a>0$.
\begin{proposition}\label{prop:non_symm_bern_subgauss}
If $p\in(1/2,1)$, then $\varphi(t)<t^2/2$ for all $t>0$. We are in the superlogarithmic case with $q=3$. The optimal lengths are $a\log^3 n$, $a\geq 0$.
\end{proposition}
For $p<1/2$ the coefficient of $t^3$ in the Taylor series~\eqref{eq:non_symm_bern_varphi} is positive. This implies that $m_*=\sup_{t>0}\psi(t)>1$. Also, it follows from~\eqref{eq:non_symm_bern_varphi} that $\lim_{t\to\infty}\psi(t)=0$, hence the maximum $m_*$ is finite.  From  Figure~\ref{fig:psi} we see that the maximum is attained at a unique point. (We were not able to prove this fact rigorously. It seems that the proof requires tedious computations with transcendental functions). Thus, we should be in the logarithmic case.  The optimal lengths are $d_* \log n + a\sqrt {\log n}$, $a\in\R$. Here, $d_*=1/\varphi(t_*)$, where $t_*$ is the solution of the transcendental equation $2t_*\varphi'(t_*)=\varphi(t_*)$.

%\smallskip
%\noindent \textit{
\subsubsection{Binomial}
If some distribution satisfies the superlogarithmic or the logarithmic assumptions, then the same is true for its convolution powers. More precisely, this means the following. Let $X_1,X_2,\ldots$ be i.i.d.\ random variables with distribution function $F(z)$. Let $\tilde X_1,\tilde X_2,\ldots$ be i.i.d.\ random variables with distribution function $F^{*m}(z\sqrt m)$, where $m\in\N$ is fixed, and $F^{*m}$ denotes the $m$-th convolution power of $F$. Then, the cumulant generating function of $\tilde X_k$ is given by $\tilde \varphi(t)=m \varphi(t/\sqrt m)$. The equality
$$
\tilde \psi(\sqrt m t):=\frac{\tilde \varphi(\sqrt m t)}{(\sqrt m t)^2/2}=\frac{\varphi(t)}{t^2/2}=\psi(t)
$$
entails that if $X_1$ satisfies the superlogarithmic or the  logarithmic conditions, then the same holds for the variable $\tilde X_1$. For example, our results on the Bernoulli distributions imply that the binomial distribution $\text{Bin}(k,p)$ (after standardization) belongs to the superlogarithmic case for $p\geq 1/2$ and to the logarithmic case for $p<1/2$.
%If the distribution function $F$ is infinitely divisible, we can even take $m>0$ to be non-integer.  It follows that the Gamma distribution $\text{G}(\alpha,\lambda)$ and the negative binomial distribution $\text{NB}(\alpha, p)$ belong, after standardization, to the logarithmic class regardless of the values of the parameters.

%\smallskip
%\noindent\textit{
\subsubsection{Uniform} Let $X_1,X_2,\ldots$ be i.i.d.\ random variables with uniform distribution on the interval $[-\sqrt 3, \sqrt 3]$, so that $\E X_k=0$ and $\Var X_k=1$. We have
$$
\varphi(t)=\log \frac{\sinh (\sqrt 3 t)}{\sqrt 3 t} = \frac{t^2}{2}-\frac {t^4}{20}+ o(t^4), \;\;\; t\to 0.
$$
We are in the superlogarithmic case with $q=4$.  To see this note that
$$
\frac{\sinh (\sqrt 3 t)}{\sqrt 3 t}- e^{t^2/2} =
\sum_{k=2}^{\infty} t^{2k} \left(\frac{3^k}{(2k+1)!}-\frac{1}{2^k k!}\right).
$$
All coefficients are negative, as one easily verifies by induction. The optimal lengths are $a\log^2n$, $a>0$.  %since $(2n+1)!\leq 6^n n!$, $n\in\N$.

%\smallskip
%\noindent \textit{Two-sided exponential.} Let $X_1,\ldots$ be i.i.d.\ with density function $\frac{1}{\sqrt 2} e^{-\sqrt 2 |x|}$, %$x\in\R$. Then,
%$$
%varphi(t)=-\log (1-\frac {t^2}2)
%$
%and we are in the logarithmic case.

%\smallskip
%\noindent \textit{Poisson.} Let $Y_1,\ldots$ be i.i.d.\ having Poisson distribution with parameter %$\lambda>0$. Consider the normalized variables $X_k=\frac1{\sqrt{\lambda}}(Y_k-\lambda)$. We have
%$$
%\varphi(t)=\lambda(e^{\frac{t}{\sqrt {\lambda}}}-1)=\frac{t^2}{2}+\frac{t^3}{6\sqrt {\lambda}}+o(t^3).
%$$
%We are in the logarithmic case. We have tightness. We believe that there is no convergence.

%\smallskip
%\noindent \textit{
\subsubsection{Gamma, Negative Binomial, Poisson} The former two distributions (including \textit{exponential} and \textit{geometric} as special cases) are covered by Theorems~\ref{theo:main_exp} and~\ref{theo:main_exp_regular}.   The Poisson distribution  is covered by Theorem~\ref{theo:main_exp}. Although it does not satisfy the assumptions of Theorem~\ref{theo:main_exp_regular}, it is easy to check that the conclusion of this theorem remains valid in the Poisson case. The square root normalization in the definition of $\MMM_n$, see~\eqref{eq:def_Ln}, is thus not natural for these distributions. See~\cite{steinebach,siegmund_yakir} for alternative normalizations.

%The exponential and more generally, Gamma distribution, belongs to the exponential class and are covered by Theorem~\ref{theo:main_exp}.

%\smallskip
%\noindent \textit{Geometric and Negative Binomial.} The geometric  distribution $\text{Geo}(p)$ fails to satisfy~\eqref{eq:exponenntial_tail}. However, the same method of proof as in Section~\ref{sec:proof_exp} shows that the behavior of $\MMM_n$ is dominated by the maximum of individual observations. Now, the geometric distribution does not belong to a max-domain of attraction. Therefore, there is no non-trivial limiting distribution for $\MMM_n$. Instead, the sequence $\MMM_n + \log_p n$ is tight.
%Similar considerations apply more generally to the negative binomial distribution.

\subsection{Remarks}\label{subsec:remarks}
We sketch some possible extensions and modifications of our results.

\subsubsection{$\MMM_n$ versus $\MMM_n^2$}\label{subsubsec:Mn_Mn_square} In order to simplify the formulas, we stated our results for $\MMM_n^2$ instead of $\MMM_n$. It is easy to translate everything to $\MMM_n$: if  $\MMM_n^2-a_n$ converges weakly to some distribution $G$ for some sequence $a_n\to +\infty$, then
\begin{equation}\label{eq:limit_Ln_no_square}
%\lim_{n\to\infty} \P\left[\MMM_n\leq \sqrt {a_n}+ \frac{\tau}{2\sqrt{a_n}}\right] = G(\tau).
2\sqrt a_n (\MMM_n-\sqrt {a_n})\todistr G.
\end{equation}
Here is the proof of this implication. Note that $\P[\MMM_n\leq 0]$ goes to $0$ as $n\to\infty$ since it can be estimated above by $\P[\UUU_n\leq 0]$, where $\UUU_n=\max\{X_1,\ldots,X_n\}$. Hence, for every $\tau\in\R$,
$$
\P\left[\MMM_n\leq \sqrt {a_n}+ \frac{\tau}{2\sqrt{a_n}}\right] = \P\left[\MMM_n^2 \leq a_n+\tau + \frac{\tau^2}{4a_n} \right]+o(1), \;\;\; n\to\infty.
$$
By our assumption, the right-hand side goes to $G(\tau)$, for every $\tau\in\R$ where $G$ is continuous. This yields~\eqref{eq:limit_Ln_no_square}. Similar argumentation applies to Theorems~\ref{theo:main_subgauss_scales} and~\ref{theo:main_supergauss_scales}.

\subsubsection{Hitting times}
It is possible to state our main results, Theorems~\ref{theo:main_subgauss} and~\ref{theo:main_supergauss} in terms of the hitting time
\begin{equation}\label{eq:hitting_time}
\bfT(u)=\min\{n\in\N: \MMM_n>u\}, \;\;\;u>0,
\end{equation}
rather than in terms of $\MMM_n$. This approach was used in~\cite{siegmund_venkatraman}. There, $\bfT(u)$ was introduced as a stopping rule for a sequential change-point detection.  It turns out that $\bfT(u)$ has limiting exponential distribution, as $u\to\infty$. The Gaussian case was analyzed in~\cite{siegmund_venkatraman}. In the non-Gaussian case we have the following two results.
\begin{proposition}\label{prop:main_subgauss_hitting}
Let the assumptions of Theorem~\ref{theo:main_subgauss} be fulfilled. Let $\alpha=\frac 12 \cdot \frac{q-6}{q-2}$. Then, for every $y>0$,
\begin{equation}
\lim_{u\to\infty} \P\left[2^{-\alpha}u^{2\alpha}e^{-\frac{u^2}2}\bfT(u)>y \right]=\exp\{-\Lambda_{q,\kappa} y\}.
\end{equation}
\end{proposition}
\begin{proposition}\label{prop:main_supergauss_hitting}
Let the assumptions of Theorem~\ref{theo:main_supergauss} be fulfilled. For every $y>0$,
\begin{equation}
\lim_{u\to\infty} \P\left[e^{-\frac{u^2}{2m_*}} \bfT(u)>y \right]=\exp\{-\Theta_* y\}.
\end{equation}
\end{proposition}
\begin{proof}[Proof of Propositions~\ref{prop:main_subgauss_hitting}, \ref{prop:main_supergauss_hitting}]
Fix $y>0$. Let $n=2^{\alpha}u^{-2\alpha}e^{u^2/2}y$ (for Proposition~\ref{prop:main_subgauss_hitting}) or $n=e^{u^2/(2m_*)} y$ (for Proposition~\ref{prop:main_supergauss_hitting}). Note that $n$ need not be integer. With $\tau=-\log y$, we have, as $u,n\to\infty$,
\begin{equation}\label{eq:tech_hitting1}
\frac{u^2}2= \log ([n]\log^{\alpha}[n])+\tau+o(1),
\;\;\;
\text{resp.}
\;\;\;
\frac {u^2}2= m_*(\log [n]+\tau).
\end{equation}
Recall from Section~\ref{subsubsec:Mn_Mn_square} that $\lim_{n\to\infty}\P[\MMM_n<0]=0$. In the case of Proposition~\ref{prop:main_subgauss_hitting},
\begin{equation}\label{eq:tech_hitting2}
\P\left[2^{-\alpha}u^{2\alpha}e^{-\frac{u^2}2}\bfT(u)>y \right]
%=
%\P[\bfT(u)>n]
=
\P[\MMM_{[n]}\leq u]
=
\P\left[\frac 12 \MMM_{[n]}^2\leq \frac{u^2}{2}\right]+o(1).
\end{equation}
Taking into account~\eqref{eq:tech_hitting1} and applying to the right hand side of~\eqref{eq:tech_hitting2} Theorem~\ref{theo:main_subgauss} we obtain that the limit of the right-hand side of~\eqref{eq:tech_hitting2} is $e^{-\Lambda_{q,k} y}$. Proposition~\ref{prop:main_supergauss_hitting} is proven analogously.
\end{proof}

\subsubsection{Two-sided version of $\MMM_n$}
If in the signal detection problem mentioned at the beginning of the paper we do not know whether the signal  has positive or negative mean, it is natural to consider $|\MMM_n|=\max\{\MMM_n^+, \MMM_n^-\}$ as a test statistic, where
$$
\MMM_{n}^+=\MMM_n=\max_{0\leq i <j\leq n} \frac{S_j-S_i}{\sqrt{j-i}},
\;\;\;
\MMM_{n}^-=-\min_{0\leq i <j\leq n} \frac{S_j-S_i}{\sqrt{j-i}}.
$$
%$\tilde \MMM_n$ is defined in the same way as $\MMM_n$, but with $\tilde X_k=-X_k$ in place of $X_k$.
Large values of $\MMM_n^+$ (resp.,\ $\MMM_n^-$) indicate the presence of a signal with positive (resp.,\ negative) mean.
%Assume that $\E e^{tX_1}<\infty$ for all $t\in\R$.
Since $\MMM_n^-$ is obtained from $\MMM_n$ by the substitution $X_k\mapsto -X_k$, our results (under appropriate assumptions)  yield  limiting distributions for both  $\MMM_n^+$ and  $\MMM_n^-$. Moreover, $\MMM_n^+$ and $\MMM_n^-$ become asymptotically independent as $n\to\infty$. We leave this fact without a proof, but note that for i.i.d.\ random variables it is well-known that the maximum and the minimum become asymptotically independent as the sample size goes to $\infty$. If the $X_k$'s have symmetric distribution and if $\MMM_n$ has a limiting distribution of the form $\exp\{- b e^{-c\tau}\}$, for some constants $b,c>0$, the asymptotic independence implies that $|\MMM_n|$ has limiting distribution of the form $\exp\{- 2 b e^{-c\tau}\}$. For non-symmetric $X_k$, it is possible that $\MMM_n^+$ and $\MMM_n^-$ belong to different cases. If this happens, the case with the larger normalizing sequence determines the behavior of $|\MMM_n|$.

\subsubsection{Non-unique maximum}
Among the distributions satisfying~\eqref{eq:varphi_def} there are some exotic examples which are not covered by  our results. For example, it is possible  that the supremum of $\psi$ (which is strictly larger than $1$) is attained at several points $t_{1}, \ldots, t_{m}>0$ simultaneously. In this case, Theorem~\ref{theo:main_supergauss} still holds, but the constant $\Theta_{*}$ in~\eqref{eq:main_supergauss} has to be replaced by $\Theta_{1}+\ldots+\Theta_{m}$, where the summands $\Theta_{i}$ correspond to the contributions of the different $t_{i}$'s. It is however not possible that the maximum of $\psi$ is attained at some interval (or some set having a limit point in $[0,\infty)$). This follows from the uniqueness theorem for analytic functions. (Note that $\psi$ can be extended analytically to the right half-plane).  It is also possible that the maximum of $\psi$ is equal to $1$, but is attained at $t_{1}= 0$ and some other point $t_{2}>0$.  The first point is described by Theorem~\ref{theo:main_subgauss} with normalization sequence $a_{1,n}=\log (n\log^{\frac 12 \cdot \frac{q-6}{q-2}} n)$, the second point is described by Theorem~\ref{theo:main_supergauss} with normalization sequence $a_{2,n}=\log n$.  If $q<6$ (which is usually the case), then $a_{2,n}-a_{1,n}\to+\infty$ and the contribution of $t_1=0$ is asymptotically negligible. Our results do not cover the situation in which $\psi(t)<1$ for all $t>0$, but $\lim_{t\to\infty}\psi(t)=1$. It is, however, difficult to find a distribution with these properties.

\subsubsection{Strong approximation} The first na\"{\i}ve attempt to obtain the limiting distribution for $\MMM_n$ is to approximate the random walk $S_n$ by a Gaussian random walk $W_n$  using  the strong invariance principle  of Koml\'os--Major--Tusn\'ady~\cite{csoergoe_book}.
We will now explain why this approach fails.
%This approach must fail, since for no distribution of $X_k$ the sequence needed to normalize $\MMM_n$ is the same as for the normal distribution.
If we exclude the case in which $S_n$ is itself Gaussian, the best possible rate of strong approximation is $|W_n-S_n|=O(\log n)$ a.s.; see~\cite{csoergoe_book}. Given $0\leq i <j\leq n$ with $l=j-i$ we obtain for the difference
$$
\frac{W_j-W_i}{\sqrt{j-i}} - \frac{S_j-S_i}{\sqrt{j-i}}
$$
the estimate $O(\log n/\sqrt l)$. If we want to apply this to show that the weak limit theorem satisfied by $\MMM_n$ has the same form in the Gaussian and in the non-Gaussian case, the approximation error should be of smaller order than the fluctuations of $\MMM_n$, which are of order $1/\sqrt {\log n}$; see, e.g.,~\eqref{eq:siegmund_venkatraman}.  Thus, we obtain a sufficiently accurate strong approximation if $l$ is of larger order than $\log^3 n$.
%If $l$ is of order $\log^3 n$ or smaller, the strong approximation is too inaccurate.
However, our results show that the behavior of $\MMM_n$ is determined by the intervals of length at most $O(\log^3 n)$. In this domain the strong approximation is too inaccurate.
%That's why the strong approximation argument fails.
The best one can prove using a direct strong approximation argument is the following result:  for any sequence $l_n$ such that $l_n/\log^3 n\to \infty$ but $l_n/n\to 0$, there is a sequence $b_n$ not depending on the distribution of $X_1$ such that $\MMM_n^2(l_n,n)-b_n$ converges to the Gumbel distribution. In the Gaussian case this can be proved by the methods of~\cite{kabluchko_unpub_07}. Then, the strong approximation implies that the same limit theorem holds for any $X_1$ satisfying~\eqref{eq:varphi_def}.

% \subsubsection{The optimal lengths}

\section{Notation and strategy of the proof}\label{subsec:method}
First we fix some notation which will be used throughout the paper. Let $\{X_k, k\in\Z\}$ be non-degenerate i.i.d.\ random variables with $\E X_k=0$, $\Var X_k=1$. We always assume that~\eqref{eq:varphi_def} holds. Define the two-sided random walk $\{S_k, k\in\Z\}$ by
$$
S_k=X_1+\ldots+X_k,\;\;\; S_0=0,\;\;\; S_{-k}=X_{-1}+\ldots+X_{-k}, \;\;\; k\in\N.
$$

Consider the set $\III=\{(i,j)\in\Z^2: i<j\}$.  Our main object of study is the \textit{standardized increments random field}  $\ZZZ=\{\ZZZ_{i,j}, (i,j)\in\III\}$ defined by
\begin{equation}\label{eq:def_ZZZ}
\ZZZ_{i,j}=\frac{S_j-S_i}{\sqrt{j-i}}, \;\;\; (i,j)\in \III.
\end{equation}
\xxx
See Figure~\ref{fig:random_field} for a realization of this random field. Elements $(i,j)\in\III$ will be called \textit{intervals}. Any interval $(i,j)\in\III$ will be identified  with $\ZZZ_{i,j}$, the corresponding standardized increment, as well as with the set $\{i+1,\ldots,j\}$. We call $l:=j-i\in\N$ the \textit{length} of $(i,j)$.

%%%%%%%%%%%%%%%%%FIGURE: Realization of the Random Field \ZZZ %%%%%%%%%%%%%%%%%%%%%%
\begin{figure}[t]
\begin{center}
\includegraphics[height=0.33\textwidth, width=0.99\textwidth]{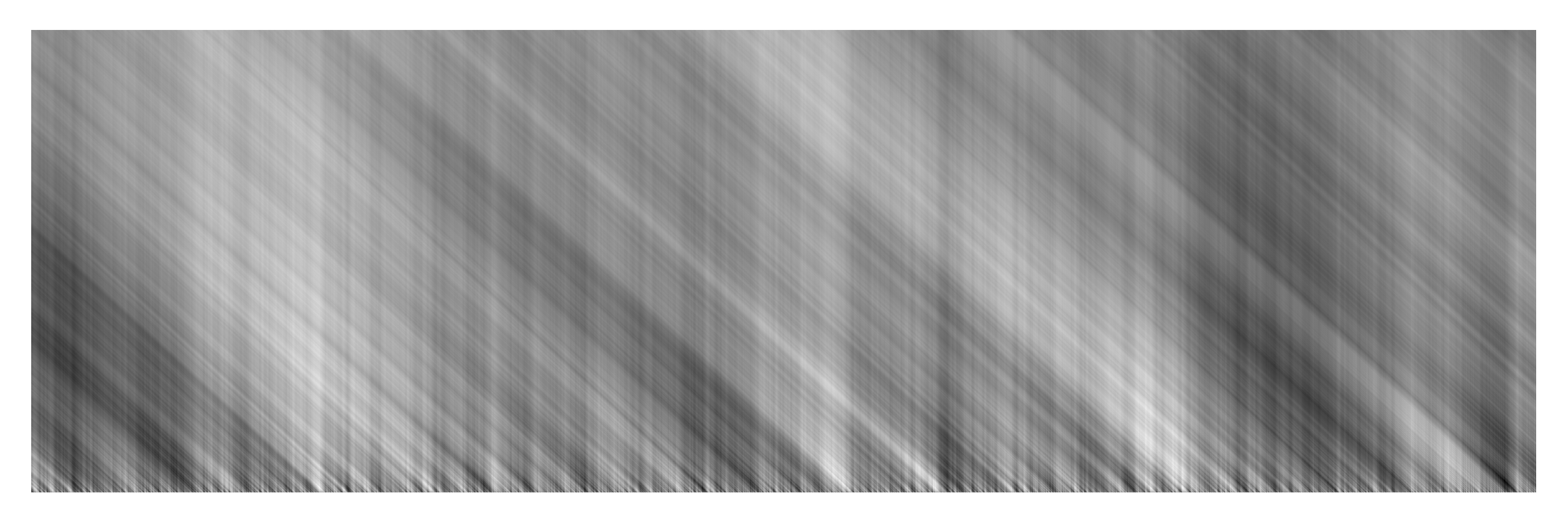}
\end{center}
\caption
{
A realization of the random field $\ZZZ_{i, i+l}$ for $1\leq i\leq 3000$ (on the horizontal axis) and $1\leq l\leq 1000$ (on the vertical axis). Dark points correspond to large values of the field.
} \label{fig:random_field}
\end{figure}
%%%%%%%%%%%%%%%%%%%%%%%%%%%%%%%%%%%%%%%%%%%%%%%%%%%%%%%

%%%%%%%%%%%%%%%%%FIGURE %%%%%%%%%%%%%%%%%%%%%%
\begin{figure}[t]
\begin{center}
\includegraphics[height=0.35\textwidth, width=0.49\textwidth]{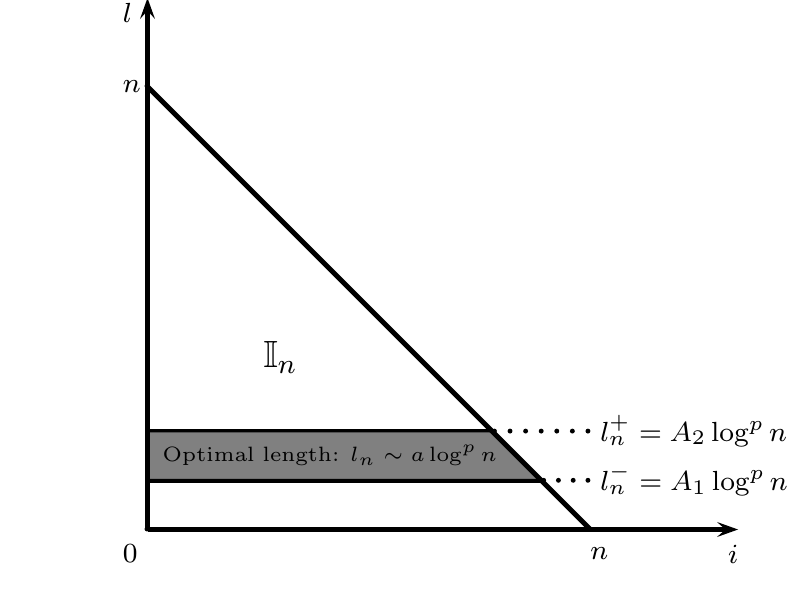}
\end{center}
\caption
{
The set $\III_n$. On the horizontal axis, we put the initial point $i$ of the interval $(i,j)$. On the vertical axis, we put its length $l=j-i$.  The strip shows the intervals whose length $l$ is between $l_n^-=A_1\log^p n$ and $l_n^+=A_2\log^p n$ (the optimal lengths in the superlogarithmic case).
} \label{fig:draw_I_n}
\end{figure}
%%%%%%%%%%%%%%%%%%%%%%%%%%%%%%%%%%%%%%%%%%%%%%%%%%%%%%%

We are interested in the random variable
$\MMM_n=\max_{(i,j)\in\III_n}\ZZZ_{i,j}$,
where $\III_n\subset \III$ is a set of intervals given by
\begin{equation}
\III_n=\{(i,j)\in\Z^2: 0\leq i < j \leq n\}, \;\;\; n\in\N;
\end{equation}
see Figure~\ref{fig:draw_I_n}. Clearly, $\MMM_n$ is a maximum of random variables, but these are neither independent, nor identically distributed.  We prove our results by a careful extreme-value analysis of the field $\ZZZ$.
%Let us explain the method of the proof.
We will take an appropriate threshold $u_n\to\infty$ and compute the limit of the \textit{exceedance probability} $\P[\MMM_n>u_n]$. Our method can be subdivided into $3$ steps.

\vspace*{2mm}
\noindent
\textsc{Step 1.} We start by  computing the \textit{individual probability} $\P[\ZZZ_{i,j}>u_n]$ for intervals $(i,j)$ having ``optimal length'' $l_n$.   The optimal length is chosen as in Section~\ref{subsec:intro}. For example, in the superlogarithmic case the optimal lengths are of the form $a\log^p n$, $a>0$, meaning that (as we will prove in Section~\ref{subsec:non_opt_subgauss}) the contribution of all other lengths is negligible:
$$
\lim_{A_1\to 0} \limsup_{n\to\infty} \P[\MMM_n(1, A_1\log^p n) >u_n]
=
\lim_{A_2\to +\infty} \limsup_{n\to\infty} \P[\MMM_n(A_2\log^p n, n) >u_n]
=
0.
$$
To compute the individual probability, we need classical limit theorems on large and moderate deviations which will be recalled in Section~\ref{sec:large_dev}. The key results of Step~1 are Lemma~\ref{lem:one_point_sub_gauss} (in the superlogarithmic case) and Lemma~\ref{lem:one_point_super_gauss} (in the logarithmic case).

\vspace*{2mm}
\noindent
\textsc{Step 2.}
In the second step we compute the \textit{local probability}  $\P[\max_{(i,j)\in \TTT_n}\ZZZ_{i,j}]$, where $\TTT_n\subset \III$ is a discrete square with side length of order $q_n$. Here, $q_n$ is chosen to be the ``extremal decorrelation length'' of $\ZZZ$. This means that the exceedances of $\ZZZ$ at two points of $\III$ at distance of order $q_n$ retain non-trivial asymptotic dependence in the large $n$ limit (that is, they become neither completely dependent, nor completely independent). There is a way to characterize $q_n$ using  the language of the Poisson clumping heuristic; see~\cite{aldous_book}. The intervals $(i,j)$ for which  $\ZZZ_{i,j}>u_n$ form small clumps distributed randomly in $\III$. Then, the linear size of these clumps is of order $q_n$. From this part of the proof it will be clear why it is not possible to choose the ``optimal length'' as large as possible (say, of order $n$).  Namely, long intervals are strongly dependent (meaning that the extremal decorrelation length is large there); see also Figure~\ref{fig:random_field}.   Therefore, long intervals make only a small contribution to $\MMM_n$. The key results of Step~2 are Lemma~\ref{lem:local_sub_gauss} (in the superlogarithmic case) and Lemma~\ref{lem:local_super_gauss} (in the logarithmic case).

\vspace*{2mm}
\noindent
\textsc{Step 3.}
The final step is to compute the exceedance probability over a domain of size much larger than $q_n$. Such domain can be decomposed into many small domains of size $q_n$, see Figure~\ref{fig:draw_block}, and the exceedance events over these small domains are asymptotically independent due to the extremal decorrelation. The asymptotic independence is shown by estimating the \textit{double sum} appearing in the Bonferroni inequality.  Thus, we can apply the Poisson limit theorem for weakly dependent events. The key steps of the third step are Lemma~\ref{lem:exc_JJJ_subgauss} (in the superlogarithmic case) and Lemma~\ref{lem:exc_JJJ_supergauss} (in the logarithmic case).

\vspace*{2mm}
On a rigorous level, there are several (closely related) powerful methods to analyze extremes of random fields; see~\cite{leadbetter_etal_book,piterbarg_book,berman_book,aldous_book}. We use a modification of the double sum method of~\citet{pickands}; see also~\citet[Chapter~12]{leadbetter_etal_book}, \citet[Chapter~D]{piterbarg_book}.
Originally, the method was used to analyze extremes of Gaussian processes,  but it can be applied to non-Gaussian scan statistics as well; see~\cite{book,piterbarg_kozlov,piterbarg,chan}. These references deal with fixed window size, for an example with variable window size see~\cite{kabluchko_spodarev08}.  A related method was also used by~\cite{siegmund_yakir}.

Throughout the paper $C_1,C_2,\ldots$ and $c_1,c_2,\ldots$ are positive constants which may change from line to line. They may depend on the distribution of $X_1$ and parameters specified in the text. We write $a_n\sim b_n$ if $\lim_{n\to\infty}a_n/b_n=1$. Let $\N_0=\{0,1,2,\ldots\}$.

\section{Results on large and moderate deviations}\label{sec:large_dev}
In our proofs we will make a heavy use of the exact asymptotic results for the probabilities of large and moderate deviations of sums of i.i.d.\ random variables. Recall from~\eqref{eq:varphi_def} that $\varphi$ is the cumulant generating function of the $X_k$'s.
%Let $X_1,X_2,\ldots$ be i.i.d.\ non-degenerate random variables with $\E X_k=0$, $\Var X_k$, and partial sums $S_k=X_1+\ldots+X_k$. We always assume that the log-Laplace transform $\varphi(t)$ given by~\eqref{eq:varphi_def} exists for all $t\geq -\eps_0$, for some $\eps_0>0$.
Define the Legendre--Fenchel transform of $\varphi$:
\begin{equation}\label{eq:def_I}
I(s)=\sup_{t\geq 0} (st-\varphi(t)), \;\;\; s\geq 0.
\end{equation}
Let $s_{\infty}=\sup\{s\in\R:\P[X_1\leq s]<1\}>0$ be the right endpoint of $X_1$. Then,  $I$ is a finite, strictly convex, strictly increasing, infinitely differentiable function on $[0,s_{\infty})$, and $I(s)=+\infty$ for $s>s_{\infty}$. Also, we have $s_{\infty}=\limsup_{t\to\infty} \varphi(t)/t$.

The next theorem on the probability of ``moderate'' deviations proved originally by Cram\'er~\cite{cramer} has been subsequently strengthened by~\citet{feller}, \citet{petrov_cramer} and~\citet{hoeglund}; see also~\cite{petrov_book} and~\cite{ibragimov_linnik_book}.
\begin{theorem}%[Cram\'er--Feller--Petrov]
\label{theo:moderate_dev}
Let $x_k>0$ be a sequence such that $x_k\to\infty$ but $x_k=o(\sqrt k)$ as $k\to\infty$. Then, as $k\to\infty$,
\begin{equation}\label{eq:cramer}
\P\left[\frac {S_k}{\sqrt k} >x_k\right] \sim \frac{1}{\sqrt{2\pi} x_k} \exp\left\{-kI\left(\frac {x_k}{\sqrt k} \right) \right\}.
\end{equation}
\end{theorem}
 Often, it is more convenient to introduce the Cram\'er series $\lambda(y)=y^{-3} (\frac{y^2}{2}-I(y))$ and state~\eqref{eq:cramer} in the following equivalent form (which is valid without the requirement $x_k\to\infty$):
\begin{equation}\label{eq:cramer1}
\P\left[\frac{S_k}{\sqrt k}>x_k\right] \sim \bar \Phi(x_k)  \cdot  \exp \left\{\frac{x_k^3}{\sqrt k}\cdot \lambda \left(\frac{x_k}{\sqrt k}\right)\right\}.
\end{equation}
Here, $\bar \Phi(x)$ is the tail function of the standard normal distribution. Relations~\eqref{eq:cramer} and~\eqref{eq:cramer1} are equivalent since $\bar \Phi(x)\sim \frac{1}{\sqrt{2\pi}x}e^{-x^2/2}$ as $x\to\infty$. Note that for $x_k$ being constant the central limit theorem is recovered.

The next theorem due to~\citet{bahadur_rao} and~\citet{petrov_ld} deals with the probabilities of ``large'' deviations of $S_k$.
\begin{theorem}%[Bahadur--Ranga Rao--Petrov]
\label{theo:large_dev}
Assume that $X_1$ is non-lattice. Let $x_k>0$ be a sequence such that $x_k\sim \alpha \sqrt k$ for some $\alpha>0$, as $k\to\infty$. Then, as $k\to\infty$,
\begin{equation}
\P\left[\frac {S_k}{\sqrt k} >x_k\right]\sim \frac{1}{\sqrt{2\pi k}}  \frac{1}{I'(\alpha) \sigma(\alpha)} \exp\left\{-k I\left(\frac{x_k}{\sqrt k}\right)\right\}.
\end{equation}
Here, $\sigma^2(\alpha)=\varphi''(I'(\alpha))$.
\end{theorem}
For lattice variables the theorem should be modified; see~\cite{petrov_ld}. In fact, Theorems~\ref{theo:moderate_dev} and~\ref{theo:large_dev} can be included  as special cases in a general result; see~\cite{hoeglund}. We will need just the following inequality; see~\cite{petrov_cramer}, \cite{hoeglund}. It is valid both for ``moderate'' and ``large'' deviations,  both in the lattice and in the non-lattice case.
\begin{theorem}\label{theo:petrov_est}
For every $A\in (0, s_{\infty})$ there is a constant $C=C(A)$ such that for all $k\in\N$, $x\in (0, A\sqrt k)$,
\begin{equation}
\P\left[\frac{S_k}{\sqrt k}>x\right] \leq \frac C x  \exp \left\{-k I\left(\frac{x}{\sqrt k}\right)\right\}.
\end{equation}
\end{theorem}
The next lemma is elementary and well-known. It is weaker than Theorem~\ref{theo:petrov_est}, but valid  without restriction on $x$.  Instead of~\eqref{eq:varphi_def} we assume that $\varphi(t)=\log \E e^{tX_1}$ is finite on $[0,t_{\infty})$, for some $t_{\infty}>0$.
\begin{lemma}\label{lem:ld_est}
%Let $X_1,X_2,\ldots$ be i.i.d.\ random variables with $\E X_1=0$.
For every $k\in\N$ and $x>0$, we have
$$
%\P[S_k>ks]\leq \exp\{-kI(s)\}.
\P\left[\frac{S_k}{\sqrt k}>x\right]\leq \exp\left\{-kI\left(\frac{x}{\sqrt k}\right)\right\}.
$$
\end{lemma}
\begin{proof}
By Markov's inequality $\P[S_k>\sqrt k x]\leq e^{- t \sqrt kx+k\varphi(t)}$, for every $t\geq 0$. Take the minimum over $t\geq 0$.
\end{proof}

%********************************************************%
%************* PROOF SUPERLOGARITHMIC   *****************%
%********************************************************%

\section{Proof in the superlogarithmic case}\label{sec:proof_subgauss}
In this section we prove Theorem~\ref{theo:main_subgauss} and Theorem~\ref{theo:main_subgauss_scales}. It will be convenient to pass from conditions~\eqref{eq:varphi_subgauss} and~\eqref{eq:varphi_taylor} to their Legendre--Fenchel conjugates. We will assume that for every $\eps>0$,
\begin{equation}\label{eq:I_subgauss}
\inf_{s\geq \eps} \frac{I(s)}{s^2/2}>1.
\end{equation}
We also need the Taylor expansion of $I$ at $0$: with $\kappa>0$,
\begin{equation}\label{eq:I_taylor}
I(s)=\frac {s^2}{2}+\kappa s^{q} +o(s^{q}), \;\;\; s\downarrow 0.
\end{equation}
\begin{proposition}\label{prop:dual_conditions_subg}
Assume that~\eqref{eq:varphi_def} holds. Then, conditions~\eqref{eq:varphi_subgauss} and~\eqref{eq:I_subgauss} are equivalent. Also, conditions~\eqref{eq:varphi_taylor} and~\eqref{eq:I_taylor} are equivalent.
\end{proposition}
\begin{proof}
Assume that~\eqref{eq:varphi_subgauss} holds. Let $\eps>0$ be arbitrary. By~\eqref{eq:varphi_subgauss} we can find $c=c(\eps)<1$ such that $\varphi(t)\leq c t^2/2$ for all $t\geq \eps$. It follows that for every $s\geq \eps$,
$$
I(s)=\sup_{t\geq 0} (st-\varphi(t)) \geq \sup_{t\geq \eps} (st-\varphi(t)) \geq  \sup_{t\geq \eps} \left(st-\frac {ct^2}2\right)=\frac{s^2}{2c}.
$$
Note that the last equality holds since the supremum of $st-\frac {ct^2}2$ is attained at $t=\frac sc >\eps$. It follows that~\eqref{eq:I_subgauss} holds. The proof that~\eqref{eq:I_subgauss} implies~\eqref{eq:varphi_subgauss} is analogous.

Assume now that~\eqref{eq:varphi_taylor} holds. Note that $\varphi$ is analytic in a neighborhood of zero. Taking the derivative, we obtain $\varphi'(t)= t - q \kappa t^{q-1}+o(t^{q-1})$.  Now, by Legendre--Fenchel duality, $I'$ is the inverse function of $\varphi'$. Taking the inverse function of $\varphi'$ we obtain $I'(s)=s + q \kappa s^{q-1}+o(s^{q-1})$. Integrating, we obtain~\eqref{eq:I_taylor}. The proof of the converse implication is analogous.
\end{proof}
Recall the notation $p=\frac{q}{q-2}$. Fix some $\tau\in\R$ and define a normalizing sequence $u_n=u_n(\tau)>0$ by
\begin{equation}\label{eq:def_un_subgauss}
u_n^2= 2\log (n\log^{\frac 32-p}n) + 2\tau.
\end{equation}
Throughout the remainder of Section~\ref{sec:proof_subgauss} we assume that conditions~\eqref{eq:I_subgauss} and~\eqref{eq:I_taylor} are satisfied. Sections~\ref{subsec:ind_subgauss}--\ref{subsec:global_subg} are devoted to the proof of Theorem~\ref{theo:main_subgauss_scales}. In Section~\ref{subsec:non_opt_subgauss} we complete the proof of Theorem~\ref{theo:main_subgauss}.

\subsection{Individual probability}\label{subsec:ind_subgauss}
The first step is to compute the probability that the random variable $\ZZZ_{i,j}$ exceeds some large threshold at some individual point $(i,j)\in \III$. We will consider intervals $(i,j)$ whose length $l_n$  is optimal, that is $l_n\sim a\log ^p n$, for some $a>0$.
\begin{lemma}\label{lem:one_point_sub_gauss}
Let $l_n\in\N$ be a sequence such that $a:=\lim_{n\to\infty}l_n/\log^p n \in (0,\infty)$. Let $s\in\R$ be fixed. Then, as $n\to\infty$,
$$
P_{n}(s)
:=
\P\left[\frac{S_{l_n}}{\sqrt{l_n}}>u_n-\frac{s}{u_n}\right]
\sim
%\frac{e^s}{\sqrt {2\pi} u_n} \exp\left\{-\frac{u_n^2}{2} - \kappa 2^{\frac{q}2} a^{-\frac {q-2}{2}} \right\}.
\frac{1}{2\sqrt {\pi}} \cdot e^{s-\kappa 2^{\frac{q}2} a^{-\frac {q-2}{2}}} \cdot \frac{e^{-\tau}}{n\log^{2-p} n}.
$$
\end{lemma}
\begin{proof}
By~\eqref{eq:def_un_subgauss}, we have
$
u_n\sim\sqrt{2\log n}
$.
 Note that $u_n-\frac {s}{u_n}=o(\sqrt {l_n})$ since $p>1$. By Theorem~\ref{theo:moderate_dev},
\begin{equation}\label{eq:P_nc_subgauss}
P_{n}(s)\sim \frac{1}{\sqrt{2\pi} u_n} \exp\left\{-l_nI\left( \frac {u_n-\frac {s}{u_n}}{\sqrt {l_n}} \right) \right\}.
\end{equation}
We evaluate the term under the sign of the exponential. Since $l_n\sim a \log^p n$ and $p=\frac q{q-2}$, we have
$$
\left(\frac{u_n-\frac {s}{u_n}}{\sqrt {l_n}}\right)^{q} = 2^{\frac{q}2} a^{-\frac {q-2}{2}}l_n^{-1}+o(l_n^{-1}).
$$
Using the Taylor expansion in~\eqref{eq:I_taylor}, we obtain
$$
l_n I\left(\frac {u_n-\frac {s}{u_n}} {\sqrt l_n} \right)
=
\frac 12 \left(u_n-\frac {s}{u_n}\right)^2
+
\kappa 2^{\frac{q}2} a^{-\frac {q-2}{2}} +o(1).
%\left(\frac{u_n-\frac {c}{u_n}}{\sqrt {l_n}} \right)^{3+q}.
$$
Inserting this into~\eqref{eq:P_nc_subgauss} yields
\begin{equation}\label{eq:Pn_s_subgauss}
P_{n}(s)
\sim
\frac{e^s}{\sqrt {2\pi} u_n} \exp\left\{-\frac{u_n^2}{2} - \kappa 2^{\frac{q}2} a^{-\frac {q-2}{2}} \right\}.
\end{equation}
To complete the proof of Lemma~\ref{lem:one_point_sub_gauss} note that $e^{u_n^2/2} = e^{\tau} n \log^{\frac 32 -p} n$ by~\eqref{eq:def_un_subgauss}.
\end{proof}

\subsection{Local probability}\label{subsec:local_subgauss}
The next step is to compute the exceedance probability over a small discrete square in the space of intervals.
Given an interval  $(x,y)\in \III$ of length  $l:=y-x$ and a ``length fluctuation'' $r\in\N$, let $\TTT_{r}(x,y)$ be the set of all intervals $(i,j)\in \III$ satisfying
$$
x-r < i \leq x \text{ and } y \leq  j < y+r.
$$
Note that all intervals from the set $\TTT_r(x,y)$ are extensions of the ``base'' interval $(x,y)$; see Figure~\ref{fig:draw_cluster2}. We can view $\TTT_r(x,y)$ as a discrete square with side length $r$ in the grid $\III\subset \Z^2$.  The base interval $(x,y)$ corresponds to the right bottom  vertex of this square. The cardinality of $\TTT_r(x,y)$ is $r^2$. Fix any sequence $q_n\in\N$ satisfying $q_n\sim \log^{p-1} n$ as $n\to\infty$. Note that $q_n\to\infty$ since $p>1$.
%%%%%%%%%%%%%%%%%FIGURE %%%%%%%%%%%%%%%%%%%%%%
\begin{figure}[t]
\begin{center}
\includegraphics[height=0.4\textwidth, width=0.5\textwidth]{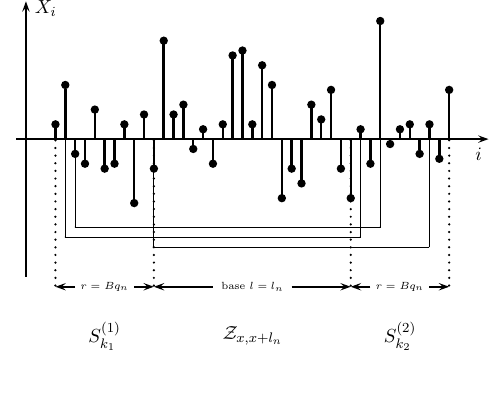}
\end{center}
\vspace{-0.8cm}
\caption
{
%Lemma~\ref{lem:local_sub_gauss} and its proof.
The structure of the set $\TTT_{n}=\TTT_{Bq_n}(x,x+l_n)$. All the intervals share the common base $(x,x+l_n)$.
} \label{fig:draw_cluster2}
\end{figure}
%%%%%%%%%%%%%%%%%%%%%%%%%%%%%%%%%%%%%%%%%%%%%%%%%%%%%%%
\begin{lemma}\label{lem:local_sub_gauss}
Let $l_n\in\N$ be a sequence such that $a:=\lim_{n\to\infty}l_n/\log^p n \in (0,\infty)$. Fix $B\in\N$ and $x\in\Z$. Write $\TTT_n=\TTT_{Bq_n}(x,x+l_n)$; see Figure~\ref{fig:draw_cluster2}. Then,  as $n\to \infty$, we have
\begin{equation}
Q_n:=\P\left[\max_{(i,j)\in \TTT_n} \ZZZ_{i,j} > u_n \right]
\sim
P_n(0) \cdot
\left\{
%1+  \left( \E \exp \sup _{t\in [0,B/a]} ( \sqrt 2  W(t) - t ) \right)^2
1+  H^2\left(\frac Ba\right)
\right\},
\end{equation}
where $P_n(0)$ is as in Lemma~\ref{lem:one_point_sub_gauss}, the function $H:(0,\infty)\to(0,\infty)$ is defined by
\begin{equation}\label{eq:HB}
H(B)=\E \left[\sup _{t\in [0,B]} e^{ \sqrt 2  W(t) - t }\right], \;\;\;B>0,
\end{equation}
and  $\{W(t), t\geq 0\}$ is a standard Brownian motion.
%\E \exp \sup _{t\in [0,B]} \left( \sqrt{\frac 2a} W(t) - \frac {t}{a}\right)
\end{lemma}
\begin{proof}
The idea of the proof is to represent the field $\ZZZ_{i,j}$ as its value at the base interval $(x,x+l_n)$ plus some incremental process. In order to exceed the level $u_n$ over $\TTT_n$ either the value $\ZZZ_{x,x+l_n}$ at the base interval should be larger than $u_n$, or this value should be of the form $u_n-u_n^{-1} s$, $s\geq 0$, and the supremum of the incremental process should be larger than $u_n^{-1}s$.  We will show that the incremental process converges to the sum of two independent Brownian motions.

Let $X_k^{(1)}$, $X_k^{(2)}$, $k\in\N$, be i.i.d.\ random variables with the same distribution as the $X_k$'s, and  which are independent of the $X_k$'s. Define two independent random walks
\begin{equation}\label{eq:Sk_i}
S_k^{(i)}=X_1^{(i)}+\ldots+X_k^{(i)},\;\;\; k\in\N_0,\;\;\; i=1,2.
%S_k^1=X_1^1+\ldots+X_k^1, \;\;\; S_k=X_1''+\ldots+X_k''.
\end{equation}
Let $V_n$ be a random variable defined by $\ZZZ_{x,x+l_n}=u_n-u_n^{-1}V_n$. Then, since any interval from $\TTT_n$ has the form $(x-k_1,x+l_n+k_2)$ with some integers $0\leq k_1,k_2 < Bq_n$, we have %a distributional equality
$$
%\max_{(i,j)\in\TTT_n} \ZZZ_{i,j} \eqdistr
Q_n=\P\left[\max_{0\leq k_1,k_2<Bq_n} \frac{(u_n-\frac{V_n}{u_n})\sqrt{l_n}+S_{k_1}^{(1)} + S_{k_2}^{(2)}} {\sqrt{l_n+k_1+k_2}}>u_n\right].
$$
By taking $k_1=k_2=0$ we see that the maximum  on the right-hand side exceeds $u_n$ if $V_n<0$. Note that $\P[V_n<0]=P_n(0)$, the probability which was evaluated in Lemma~\ref{lem:one_point_sub_gauss}.  The random variables $V_n$, $S_{k_1}^{(1)}$, $S_{k_2}^{(2)}$ are independent.  Conditioning on $V_n=s$  and integrating over  $s\geq 0$, we obtain
%Therefore,
%$$
%Q_n=\P\left[\max_{0\leq k_1,k_2<Bq_n}\left(S_{k_1}^{(1)} + %S_{k_2}^{(2)}-\frac{u_n(k_1+k_2)}{\sqrt{l_n+k_1+k_2}+\sqrt{l_n}}\right)>\frac{V_n\sqrt{l_n}}{u_n} \right]
%$$
%By taking $k_1=k_2=0$ we see that the event in the right-hand side occurs if $V_n<0$. Note that $\P[V_n<0]=P_n(0)$. Conditioning on $V_n=s$ and integrating over  $s\geq 0$, we obtain
\begin{equation}\label{eq:Q_n_dec_subgauss}
Q_n
=
%P_n(0)\left(1+\int_{0}^{\infty} G_n(s) d\mu_n(s) \right),
P_n(0) + \int_{0}^{\infty} G_n(s) d\mu_n(s),
\end{equation}
where $\mu_n$ is the probability distribution of $V_n$ and $G_n$ is a non-increasing function defined by
\begin{align*}
G_n(s)
=
\P\left[\max_{0\leq k_1,k_2 < Bq_n}
\left(S_{k_1}^{(1)} + S_{k_2}^{(2)}-u_n(\sqrt{l_n+k_1+k_2}-\sqrt{l_n}) -\frac{s\sqrt{l_n}}{u_n}\right)> 0 \right].
%\frac{u_n(k_1+k_2)}{\sqrt{l_n+k_1+k_2}+\sqrt{l_n}}
\end{align*}
%a measure on $[0,\infty)$ given  by
%$\mu_n([0,s])=\P[V_n\leq s]/P_n(0)$.
By Lemma~\ref{lem:one_point_sub_gauss}, for every $s\geq 0$,
\begin{equation}\label{eq:asympt_F_n}
\lim_{n\to\infty} \frac{\mu_n([0,s))}{P_n(0)}= \lim_{n\to\infty} \frac{P_n(s)}{P_n(0)}= e^s.
\end{equation}

We are going to compute $\lim_{n\to\infty} G_n(s)$ for $s\geq 0$. In fact, to be able to use Lemma~\ref{lem:dom_conv}, see below, we need a slightly stronger result.  Let $s_n$ be any sequence converging to $s$. We will compute $\lim_{n\to\infty} G_n(s_n)$. We have
\begin{align}\label{eq:def_Gnsn}
G_n(s_n)
=
\P\left[\sup_{t_1,t_2\in[0,B]} (W_n^{(1)}(t_1)+W_n^{(2)}(t_2)-f_n(t_1,t_2)) >0\right],
\end{align}
where $\{W_n^{(i)}(t), t\in [0,B]\}$, $i=1,2$, are stochastic processes and $f_n:[0,B]^2\to\R$ is a function given by
$$
W_n^{(i)}(t)= \frac{S_{tq_n}^{(i)}}{\sqrt {q_n}},\;\;\; f_n(t_1,t_2)=\frac{u_n(\sqrt{l_n+t_1 q_n+t_2 q_n}-\sqrt{l_n})+\frac{s_n\sqrt{l_n}}{u_n}}{\sqrt {q_n}}
$$
if $t,t_1,t_2\in q_n^{-1}\Z\cap [0,B)$, and by linear interpolation otherwise. On the last interval of length $q_n^{-1}$ we agree to use constant interpolation.
Recall that $u_n\sim \sqrt{2\log n}$, $l_n\sim a \log^{p} n$, $q_n\sim \log^{p-1} n$, as $n\to\infty$.
Elementary calculus shows that uniformly in $t_1,t_2\in [0,B]$,
$$
%f(t_1,t_2):=\lim_{n\to\infty} f_n(t_1,t_2)=\frac{t_1+t_2}{\sqrt{2a}}+\sqrt{\frac a 2} s.
f_n(t_1,t_2) \to f(t_1,t_2):=\frac{t_1+t_2}{\sqrt{2a}}+\sqrt{\frac a 2} s \;\; \text{ as } n\to\infty.
$$
Given a compact metric space $K$ let $C(K)$ be the space of continuous functions on $K$ endowed with the sup-metric.  By Donsker's invariance principle, as $n\to\infty$, the processes $\{W_n^{(i)}(t), t\in [0,B]\}$, $i=1,2$, converge weakly on the space $C[0,B]$ to two independent standard Brownian motions $\{W_i(t), t\in[0,B]\}$, $i=1,2$. Consider the map
$$
\Psi: C[0,B]\times C[0,B] \times C([0,B]^2) \to \R
$$
defined by
$$
\Psi(w_1,w_2,f)= \sup_{t_1,t_2\in [0,B]} (w_1(t_1)+w_2(t_2)-f(t_1,t_2)).
$$
The  map $\Psi$ is continuous in the product topology.
By the continuous mapping theorem, see Theorem 3.27 in~\cite{kallenberg_book}, it follows that the sequence of random variables
$$
\Psi(W_n^{(1)},W_n^{(2)}, f_n) = \sup_{t_1,t_2\in [0,B]}(W_n^{(1)}(t_1)+W_n^{(2)}(t_2)-f_n(t_1,t_2))
$$
converges in distribution to the random variable
$$
\Psi(W_1,W_2, f)
=
\sup_{t_1,t_2\in [0,B]} \left(W_1(t_1)+W_2(t_2)- \frac{t_1+t_2}{\sqrt {2a}}-\sqrt{\frac a 2} s\right).
$$
By the scaling property of the Brownian motion, the latter variable has the same distribution as
$
\sqrt{\frac a 2}(M_1+M_2-s)
$,
where $M_1,M_2$ are two independent copies of the random variable
$$
M:=
\sup_{t\in [0,a^{-1}B]} (\sqrt 2W(t)-t)\geq 0.
%\sup_{t\in [0,B]} \left(\sqrt {\frac 2a}W(t)-\frac{t}{a}\Big.
$$
Here, $\{W(t), t\geq 0\}$ is a standard Brownian motion.  Note that the random variable $M$ (and hence, $M_1+M_2$) has continuous distribution function. It follows that for every sequence $s_n$ converging to $s\geq 0$,
\begin{equation}\label{eq:asympt_G_n}
\lim_{n\to\infty} G_n(s_n) = \P[M_1+M_2>s].
\end{equation}
Taking~\eqref{eq:asympt_F_n} and~\eqref{eq:asympt_G_n} together, we obtain, formally,
\begin{equation}\label{eq:lem_local_complete_proof}
\lim_{n\to\infty} \int_{0}^{\infty} G_n(s) \frac {d\mu_n(s)} {P_n(0)}
=
\int_0^{\infty} \P[M_1+M_2>s] e^s ds
=
\E e^{M_1+M_2} = (\E e^{M})^2.
\end{equation}
Recalling~\eqref{eq:Q_n_dec_subgauss} we obtain the statement of the lemma. The first equality in~\eqref{eq:lem_local_complete_proof} will be justified in Remark~\ref{rem:dom_conv_just_supergauss} after some technical preparations have been done.
\end{proof}

The next lemma gives a somewhat weaker statement than Lemma~\ref{lem:local_sub_gauss}, but this statement is valid under more general assumptions. The lemma will be used later to estimate the exceedance probability over the non-optimal lengths. Essentially, it states that the local exceedance probability can be estimated by the individual exceedance probability times some constant.   %For $x\in\Z$ and $l,r\in\N$  let $T(x,l,r)$ be the set of intervals $(i,j)\in\III$ such that $x-r<i \leq x$ and $x+l\leq j<x+l+r$. Note that all intervals contain the base interval $(x,x+l)$.
\begin{lemma}\label{lem:local_est_sub_gauss}
Fix constants $B_1,B_2>0$. Then, for all $x\in\Z$, $l,r\in\N$ and all $u>0$ such that $B_1 l > u^2$ and $r<B_2lu^{-2}$, we have
$$
Q(l,r,u)
:=
\P\left[ \max_{(i,j)\in \TTT_r(x,x+l)} \frac{S_j-S_i}{\sqrt l}>u \right]
<
C_1 u^{-1}\exp\left\{-\frac{u^2}{2} -C_2 u^{q} l^{-\frac{q-2}{2}}\right\},
$$
where the constants $C_1$ and $C_2$ depend on $B_1$ and $B_2$ but don't depend on $x,l,r,u$.
\end{lemma}
\begin{proof}
%The idea the same as in the proof of Lemma~\ref{lem:local_sub_gauss}.
Define two independent random walks $S_k^{(i)}$, $i=1,2$ as in~\eqref{eq:Sk_i}. Let $V_{l,u}$ be a random variable defined by $\ZZZ_{x,x+l}=u-u^{-1}V_{l,u}$.  Then, since any interval from $\TTT_r(x,x+l)$ has the form $(x-k_1,x+l+k_2)$ with some integers $0\leq k_1,k_2<r$, we have
$$
Q(l,r,u)= \P\left[\max_{0\leq k_1,k_2<r} \frac{S_{k_1}^{(1)}+S_{k_2}^{(2)}}{\sqrt l} > \frac{V_{l,u}}{u}\right].
$$
Taking $k_1,k_2=0$ we see that the maximum on the right-hand side is non-negative. Conditioning on  $V_{l,u}=s$ and considering the cases $s<0$ and $s\geq 0$ separately, we obtain
\begin{equation}\label{eq:Qlru}
Q(l,r,u)=F_{l,u}(0)+\int_{0}^{\infty} G_{l,r,u}(s) dF_{l,u}(s),
\end{equation}
where $F_{l,u}$ is the distribution function of $V_{l,u}$ and
\begin{equation}\label{eq:Glru}
G_{l,r,u}(s)=\P\left[\max_{0\leq k_1,k_2<r} \left(S_{k_1}^{(1)} + S_{k_2}^{(2)}\right) >\frac{s\sqrt l}{u} \right]
\leq
2\P\left[\max_{0\leq k <r} S_{k} >\frac{s\sqrt l}{2u} \right].
\end{equation}
%\begin{equation}\label{eq:Qlru}
%Q(l,r,u)=\P[\ZZZ_{x,x+l}>u]+\P\left[\max_{(i,j)\in T(x,l,r)} \frac{S_j-S_i}{\sqrt l} >u, \ZZZ_{x,x+l}\leq %u\right].
%\end{equation}
%The first probability is equal to $\P[S_l/\sqrt l>u]$. Denote the second probability by $\tilde Q(l,r,u)$. We %have
%$$
%\tilde Q(l,r,u)=\int_{0}^{\infty} G_{l,r,u}(s) dF_{l,u}(s),
%$$
%where $F_{l,u}$ is the distribution function of $u(u-\ZZZ_{x,x+l})$ and
%$$
%G_{l,r,u}(s)=\P\left[\max_{(i,j)\in T(x,l,r)} \ZZZ_{i,j}>u \Big| \ZZZ_{x,x+l}=u-\frac{s}{u}\right].
%$$
We estimate $F_{l,u}(s)$ for $s\in [0, \frac 34 u^2]$. Write $v=u-\frac su$. By the assumption $v<\sqrt{B_1 l}$ we can apply Theorem~\ref{theo:petrov_est} and~\eqref{eq:I_subgauss}, \eqref{eq:I_taylor} to get
$$
F_{l,u}(s)
=
\P\left[\frac{S_l}{\sqrt l} \geq v\right]
\leq
c_1 v^{-1} \exp\left\{ -\frac{v^2}2-c_2 v^{q}l^{-\frac{q-2}{2}}\right\}.
$$
Here, the constants $c_1,c_2,\ldots$ depend on $B_1$, $B_2$ but don't depend on $x,l,r,u$. If $s\in [0, \frac 34 u^2]$, then $v\geq u/4$ and we obtain
\begin{equation}\label{eq:asympt_F_lu}
F_{l,u}(s)\leq c_3 u^{-1} e^s  \exp\left \{ -\frac{u^2}2-c_4 u^{q}l^{-\frac{q-2}{2}}\right\}.
\end{equation}
It is however easy to see that this inequality continues to hold for $s\geq 3u^2/4$. Indeed, if $c_4$ is sufficiently small, then the assumption $B_1 l > u^2$ implies that $c_4 u^{q}l^{-(q-2)/2}\leq u^2/8$. Hence, if $c_3$ is sufficiently large, the right-hand side of~\eqref{eq:asympt_F_lu} is greater than $1$ and~\eqref{eq:asympt_F_lu} holds.

We estimate $G_{l,r,u}(s)$ for $s\geq 0$.
%Let $X_k^{(i)}$, $k\in\N$, $i=1,2$, be i.i.d.\ random variables with the same distribution as the $X_k$'s which are independent of the $X_k$'s. Define the random walks $S_k^{(i)}=X_1^{(i)}+\ldots+X_k^{(i)}$, $i=1,2$.
%Then, since any pair from $T(x,l,r)$ has the form $(x-k_1,x+l+k_2)$ with some integer $0\leq k_1,k_2<r$, we have
%\begin{align*}
%G_{l,r,u}(s)
%%&=
%%\P\left[\max_{k_1,k_2=0,\ldots,r-1} \frac{\sqrt l \ZZZ_{x,x+l}+S_{k_1}^{(1)} + S_{k_2}^{(2)}} %{\sqrt{l+k_1+k_2}}>u \Big| %\ZZZ_{x,x+l}=u-\frac{s}{u}\right]\\
%&=
%\P\left[\max_{k_1,k_2=0,\ldots,r-1} \left(\ZZZ_{x,x+l}+ \frac{S_{k_1}^{(1)} + S_{k_2}^{(2)}} %{\sqrt{l}}\right) >u \Big| \ZZZ_{x,x+l}=u-\frac{s}{u}\right]\\
%&=
%\P\left[\max_{k_1,k_2=0,\ldots,r-1} \left(S_{k_1}^{(1)} + S_{k_2}^{(2)}\right) >\frac{s\sqrt l}{u} \right]\\
%&\leq
%2\P\left[\max_{k=0,\ldots,r-1} S_{k}^{(1)} >\frac{s\sqrt l}{2u} \right].
%\end{align*}
Applying to the right-hand side of~\eqref{eq:Glru} the inequality stated in Theorem~2.4 on p.~52 in~\cite{petrov_book}, we obtain
$$
G_{l,r,u}(s)
\leq
4\P\left[ S_{r} >\frac{s\sqrt l}{2u} -\sqrt{2r} \right]
\leq
4\P\left[ \frac{S_{r}}{\sqrt r} > c_5s -\sqrt{2} \right]
\leq
4 e^{-r I(\frac{c_5s-\sqrt 2}{\sqrt r})}
.
$$
In the second inequality, we used the assumption $r<B_2lu^{-2}$. In the third inequality we used Lemma~\ref{lem:ld_est}. From~\eqref{eq:I_subgauss} we obtain
\begin{equation}\label{eq:Glru_est}
G_{l,r,u}(s)
\leq
c_6 e^{-c_7s^2}.
\end{equation}
Strictly speaking, this is valid only as long as $c_5 s\geq \sqrt 2$, however, we can choose the constant $c_6$ so large that~\eqref{eq:Glru_est} continues to hold in the case $c_5 s\leq \sqrt 2$.
It follows from~\eqref{eq:Qlru}, \eqref{eq:asympt_F_lu}, \eqref{eq:Glru_est}  that
\begin{align*}
Q(l,r,u)
&\leq
F_{l,u}(0)+\sum_{k=0}^{\infty} G_{l,r,u}(k) F_{l,u}(k+1)\\
&\leq
c_8\left(1+\sum_{k=0}^{\infty}e^{-c_7 k^2}e^k\right) u^{-1} \exp\left \{ -\frac{u^2}2-c_4 u^{q}l^{-\frac{q-2}{2}}\right\}. %+ c_9 e^{-c_{10}u^4}\\
%&\leq
%c_{11} u^{-1} \exp\Big \{ -\frac{u^2}2-c_4 u^{q}l^{-\frac{q-2}{2}}\right\}
\end{align*}
The proof of Lemma~\ref{lem:local_est_sub_gauss} is complete.
%To complete the proof recall~\eqref{eq:Qlru} and estimate the first term on the right-hand side of~\eqref{eq:Qlru} by~\eqref{eq:asympt_F_lu} with $s=0$.
\end{proof}

\begin{lemma}\label{lem:dom_conv}
Let $\nu, \nu_n$, $n\in\N$, be measures on $[0,\infty)$ which are finite on compact intervals. Let $G, G_n$, $n\in\N$, be measurable functions on $[0,\infty)$ which are uniformly bounded on compact intervals. Assume that
\begin{enumerate}
\item $\nu_n$ converges to $\nu$ weakly on every interval $[0,t]$, $t\geq 0$;
\item for $\nu$-a.e.\ $s\geq 0$ and for every sequence $s_n\to s$ we have  $\lim_{n\to\infty} G_n(s_n)=G(s)$;
\item $\lim_{T\to+\infty} \int_{T}^{\infty} |G_n| d\nu_n=0$ uniformly over $n\in\N$.
\end{enumerate}
Then, $\lim_{n\to\infty}\int_{0}^{\infty} G_nd\nu_n=\int_0^{\infty} G d\nu$.
\end{lemma}
\begin{proof}
Write $F_n(s)=\nu_n([0,s])$ and $F(s)=\nu([0,s])$, $s\geq 0$.  By the first assumption,  $\lim_{n\to\infty} F_n(T)=F(T)$ for all $T\geq 0$, $T\notin \mathcal D$, where $\mathcal D$ is the set of (at most countable) discontinuities of  $F$.  By the third assumption it suffices to show that for every $T>0$, $T\notin \mathcal D$,
\begin{equation}\label{eq:lem_dom_conv_proof1}
\lim_{n\to\infty}\int_{0}^{T} G_nd\nu_n=\int_0^{T} G d\nu.
\end{equation}
Fix some $T\notin \mathcal D$.  By the first assumption and  by Skorokhod's representation theorem we can construct (generally, dependent) $[0,T]$-valued random variables $\xi_n$ and $\xi$ on a common probability space such that $\xi_n\to \xi$ a.s.,\ $\P[\xi_n\in B]=\nu_n(B)/F_n(T)$ and  $\P[\xi\in B]=\nu(B)/F(T)$, for every Borel set $B\subset [0,T]$. It follows from the second condition that $G_n(\xi_n)\to G(\xi)$ a.s. Since for uniformly bounded random variables the a.s.\ convergence implies the convergence of expectations, we obtain~\eqref{eq:lem_dom_conv_proof1}.
\end{proof}

\begin{remark}\label{rem:dom_conv}
We will verify the last condition of Lemma~\ref{lem:dom_conv} by showing that $\nu_n([0,s])\leq c_1 e^{c_2s}$ and $|G_n(s)|\leq c_3e^{-c_4 s^2}$ for all (large) $n\in\N$.
Then,
$$
\int_{T}^{\infty} |G_n|d\nu_n
\leq
\sum_{k=[T]}^{\infty}  \nu_n([k,k+1]) \sup_{s\in [k,k+1]} |G_n(s)|
\leq
\sum_{k=[T]}^{\infty} c_5 e^{c_2 k} e^{-c_4 k^2},
$$
which converges to $0$ uniformly in $n\in\N$,  as $T\to+\infty$.
\end{remark}

\begin{remark}\label{rem:dom_conv_just_supergauss}
We are in position to justify the first equality in~\eqref{eq:lem_local_complete_proof}. The first two assumptions of Lemma~\ref{lem:dom_conv}, with $\nu_n=\mu_n/P_n(0)$, are fulfilled by~\eqref{eq:asympt_F_n} and~\eqref{eq:asympt_G_n}. To verify the last assumption we use Remark~\ref{rem:dom_conv}.
For $s\geq 0$ we have, as established in~\eqref{eq:asympt_F_lu} and Lemma~\ref{lem:one_point_sub_gauss},
$$
\mu_n([0,s])\leq P_n^{-1}(0) u_n^{-1} e^s \exp\left\{-\frac{u_n^2}{2}- cu_n^q l_n^{-\frac{q-2}{2}}\right\}\leq c_1 e^s.
$$
For $G_n$ we obtain from~\eqref{eq:Glru_est}  the estimate $G_n(s)\leq G_{l_n, q_n, u_n}(s)\leq c_2 e^{-c_3 s^2}$. Now, Remark~\ref{rem:dom_conv} can be applied.
%the measure $\mu_n$ converges weakly, on every interval $[0,T]$, to the measure with density $e^s$. Using this fact together with~\eqref{eq:asympt_G_n} and the continuous mapping theorem (see Theorem~3.27 in~\cite{kallenberg_book}), we obtain that for every $T>0$,
%$$
%\lim_{n\to\infty} \int_{0}^{T} G_n(s) dF_n(s) = \int_0^{T} \P[M_1+M_2>s] e^s ds.
%$$
%On the other hand, by~\eqref{eq:Glru_est} below, we have $G_n(s)\leq e^{-c s^2}$ for every $s\geq 0$. Also,
\end{remark}

\subsection{Estimating the double sum}\label{subsec:double_sum_subgauss}
Given $0<A_1<A_2$ we define  $l_n^-=A_1 \log^p n$ and $l_n^+=A_2 \log^p n$. Recall that $q_n$ is any sequence such that $q_n\sim \log^{p-1} n$ as $n\to\infty$.
%%%%%%%%%%%%%%%%%FIGURE %%%%%%%%%%%%%%%%%%%%%%
\begin{figure}[t]
\begin{center}
\includegraphics[height=0.35\textwidth, width=0.49\textwidth]{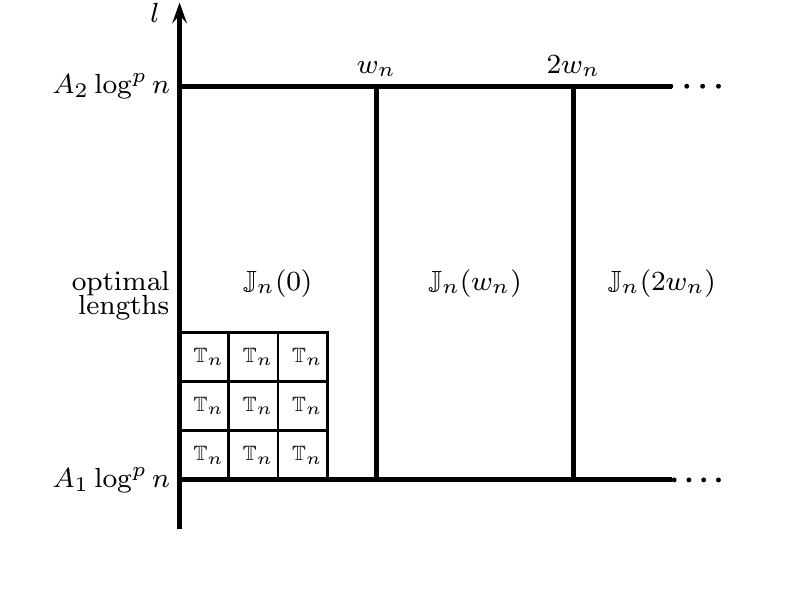}
\end{center}
\caption
{
%Method of proof of Theorem~\ref{theo:main_subgauss_scales}. The left figure shows the set $\III_n$. On the horizontal axis we put the initial point $i$ of the interval $(i,j)$, on the vertical axis we put its length $l=j-i$.
On the horizontal axis we put the initial point $i$ of the interval $(i,j)$, on the vertical axis we put its length $l=j-i$. The strip is the set of intervals whose length $l$ is between $l_n^-$ and $l_n^+$. The figure shows the decomposition of this strip into disjoint ``blocks'' of the form $\JJJ_n(z)$. Lemma~\ref{lem:exc_JJJ_subgauss} computes the exceedance probability over one block $\JJJ_n(z)$ by decomposing it into disjoint squares of the form $\TTT_n=\TTT_{Bq_n}(x,y)$. The exceedance probability over a single square $\TTT_{n}$ is computed in Lemma~\ref{lem:local_sub_gauss}. In Section~\ref{subsec:global_subg} it is shown that the exceedance events over different blocks of the form $\JJJ_n(z)$ are approximatively independent.
%The main tool to prove this is Lemma~\ref{lem:double_probab_subgauss} which allows to estimate the ``double sum''.
} \label{fig:draw_block}
\end{figure}
%%%%%%%%%%%%%%%%%%%%%%%%%%%%%%%%%%%%%%%%%%%%%%%%%%%%%%%

\begin{lemma}\label{lem:exc_JJJ_subgauss}
Let $w_n$  be any integer sequence such that $w_n/q_n\to\infty$ but  $w_n=O(\log^p n)$, as $n\to\infty$. For $z\in\Z$ let $\JJJ_n(z)$ be the set of all intervals $(i,j)\in\III$ such that $z\leq i< z+w_n$ and $j-i\in [l_n^-, l_n^+]$; see Figure~\ref{fig:draw_block}. Then, as $n\to\infty$,
\begin{equation}\label{eq:lem:exc_JJJ_subgauss}
\P\left[\max_{(i,j)\in \JJJ_n(z)} \ZZZ_{i,j}>u_n\right]\sim e^{-\tau} \frac{w_n}{n} \int_{A_1}^{A_2} \Lambda(a)da,
\end{equation}
where $\Lambda(a)=\Lambda_{q,\kappa}(a)=\frac{1}{2\sqrt {\pi} a^2} \exp\{-\kappa 2^{\frac q 2} a^{-\frac{q-2}{2}}\}$, $a>0$, is as in~\eqref{eq:def_Lambda_q_kappa}.
\end{lemma}
\begin{proof}
By translation invariance we may take $z=0$. Take some $B\in\N$ and recall that $q_n\in\N$ is a sequence satisfying $q_n\sim \log^{p-1}n$. To get rid of the boundary effects we introduce two sequences $\eps_n$ and $\delta_n$ such that $\eps_n/q_n\to\infty$ and $\delta_n/q_n\to\infty$, but $\eps_n=o(w_n)$ and $\delta_n=o(\log^p n)$, as $n\to\infty$.
%Given $q\in\N$ we
Introduce the following two-dimensional discrete grids with mesh size $q_n$:
\begin{align}
\calJ_n(B) &=\{(x,y)\in Bq_n\Z^2: x\in [-\eps_n, w_n+\eps_n],\; y-x \in [l_n^--\delta_n,  l_n^++\delta_n] \}, \label{eq:calXn}\\
\calJ_n'(B)&=\{(x,y)\in Bq_n\Z^2: x\in [\eps_n, w_n-\eps_n], \; y-x\in [l_n^-+\delta_n, l_n^+-\delta_n] \}. \label{eq:calXn_prime}
\end{align}
%Introduce the one-dimensional discrete grids with mesh size $q$:
%\begin{align}
%\calX_n(B) &= Bq_n\Z\cap [-\eps_n, w_n+\eps_n], &  \calL_n(Bq_n) &=Bq_n\Z\cap [l_n^--\delta_n, l_n^++\delta_n], \label{eq:calXn}\\
%\calX_n'(B)&= Bq_n\Z\cap [\eps_n, w_n-\eps_n], &  \calL_n'(Bq_n) &=Bq_n\Z\cap [l_n^-+\delta_n, l_n^+-\delta_n]. \label{eq:calXn_prime}
%\end{align}
Note that $\calJ_n'(B)\subset \calJ_n(B)$.  The discrete squares $\{\TTT_{Bq_n}(x,y)\}_{(x,y)\in \calJ_n(B)}$ (which were defined in Section~\ref{subsec:local_subgauss}) are disjoint and cover the set $\JJJ_n(0)$. Similarly, the discrete squares $\{\TTT_{Bq_n}(x,y)\}_{(x,y)\in \calJ_n'(B)}$  are disjoint and contained in $\JJJ_n(0)$. By the Bonferroni inequality, we have, for every $B\in\N$,
\begin{equation}\label{eq:bonferroni_subgauss}
S_n'(B)-S_n''(B)\leq \P\left[\max_{(i,j)\in \JJJ_n(0)} \ZZZ_{i,j}>u_n\right]\leq S_n(B),
\end{equation}
where
\begin{align}
S_n(B)&=\sum_{(x,y)\in\calJ_n(B)}   \P \left[\max_{(i,j)\in \TTT_{Bq_n}(x,y)} \ZZZ_{i,j}>u_n\right],
\label{eq:Sn_B_subg}\\
S_n'(B)&=\sum_{(x,y)\in\calJ_n'(B)} \P \left[\max_{(i,j)\in \TTT_{Bq_n}(x,y)} \ZZZ_{i,j}>u_n\right],
\label{eq:Sn_B_prime_subg}\\
S_n''(B)&=\sum \P \left[\max_{(i,j)\in \TTT_{Bq_n}(x_1,y_1)} \ZZZ_{i,j}>u_n , \max_{(i,j)\in \TTT_{Bq_n}(x_2,y_2)} \ZZZ_{i,j}>u_n \right]
\label{eq:Sn_B_2prime_subg}
\end{align}
and in~\eqref{eq:Sn_B_2prime_subg} the sum is taken over all pairs $(x_1,y_1)\in \calJ_n'(B)$ and $(x_2,y_2)\in \calJ_n'(B)$ such that $(x_1, y_1)\neq  (x_2,y_2)$.
The statement of Lemma~\ref{lem:exc_JJJ_subgauss} follows by letting $n\to\infty$ and then $B\to\infty$ in~\eqref{eq:bonferroni_subgauss} and applying  Lemmas~\ref{lem:Sn_B_subgauss}, \ref{lem:Sn_B_prime_subgauss}, \ref{lem:double_sum_subgauss} which we will prove below.
\end{proof}

%The next lemma can be found in~\cite{pickands}, \cite{leadbetter_etal_book}, \cite{piterbarg_book}.
%\begin{lemma}\label{lem:pickands_const_subgauss}
%For $B>0$ define $H(B)=\E \exp\sup_{t\in [0,B]} (\sqrt 2 W(t)-t)$, where $\{W(t), t\geq 0\}$ is a standard Brownian motion. Then, $\lim_{B\to\infty} %H(B)/B = 1$.
%\end{lemma}

%Recall that $w_n\to\infty$ is an integer sequence such that $w_n=O(\log^{p-1} n)$ as $n\to\infty$.
\begin{lemma}\label{lem:Sn_B_subgauss}
Let $S_n(B)$ be defined as in~\eqref{eq:Sn_B_subg}. We have
\begin{equation}\label{eq:lem:Sn_B_subgauss}
\lim_{B\to\infty} \limsup_{n\to\infty} nw_n^{-1} S_n(B) \leq  e^{-\tau} \int_{A_1}^{A_2} \Lambda(a)da.
\end{equation}
\end{lemma}
\begin{proof}
Since the probability in the right-hand side of~\eqref{eq:Sn_B_subg} depends only on $l:=y-x$ by translation invariance, we have
$$
S_n(B)\leq  \frac{w_n+o(w_n)}{Bq_n} \sum_{l\in\calL_n(B)} \P \left[\max_{(i,j)\in \TTT_{Bq_n}(0,l)} \ZZZ_{i,j}>u_n\right],
$$
where $\calL_n(B)=Bq_n\Z\cap [l_n^--\delta_n, l_n^++\delta_n]$.
The idea is now to apply to each probability Lemma~\ref{lem:local_sub_gauss} and replace Riemann sums by integrals. Introduce the function
$$
\lambda_{n,B}(a)=n \log^{2 - p} n \cdot \P \left[\max_{(i,j)\in \TTT_{Bq_n}(0,l_{n,B}(a))} \ZZZ_{i,j}>u_n\right], \;\;\; a>0,
$$
where $l_{n,B}(a)=\max\{l\in Bq_n\Z: l\leq a\log^p n\}$. The function $\lambda_{n,B}(a)$ is locally constant and its constancy intervals have length $Bq_n / \log^{p} n\sim B/\log n$.
It follows that
\begin{equation}\label{eq:tech0_subg}
S_n(B)\leq \frac{w_n+o(w_n)}{B^2n} \int_{A_1-\frac{2\delta_n}{\log^{p} n}}^{A_2+\frac{2\delta_n}{\log^{p} n}}
\lambda_{n,B}(a)da.
\end{equation}
For every fixed $a>0$, the sequence $l_n=l_{n,B}(a)$ satisfies the assumption of Lemma~\ref{lem:local_sub_gauss}. Hence, by Lemmas~\ref{lem:local_sub_gauss} and~\ref{lem:one_point_sub_gauss}, we have the pointwise convergence
\begin{equation}\label{eq:tech_pointwise1}
\lim_{n\to\infty} \lambda_{n,B}(a) = e^{-\tau}\Lambda_B(a), \;\;\; \Lambda_B(a)=\frac{1}{2\sqrt {\pi}}\cdot e^{-\kappa 2^{\frac q 2} a^{-\frac{q-2}{2}}} \left(1+H^2\left(\frac B a\right)\right).
\end{equation}
Also, by Lemma~\ref{lem:local_est_sub_gauss},  $\lambda_{n,B}(a)$ is bounded by a constant not depending on $a, n$, as long as $a$ stays bounded away from $0$ and $\infty$.
%Also, $|\calX_n(B)|\sim w_n/(Bq_n)$ and hence,
Applying the dominated convergence theorem to~\eqref{eq:tech0_subg} we obtain
\begin{equation}\label{eq:tech1_subg}
\limsup_{n\to\infty} nw_n^{-1}S_n(B) \leq e^{-\tau} \int_{A_1}^{A_2}  B^{-2}  \Lambda_B(a) da.
\end{equation}
Now we let $B\to\infty$. Recall that $H$ is a function defined by~\eqref{eq:HB}. It is known  that  $\lim_{B\to\infty} H(B)/B = 1$; see~\cite[p.~72]{pickands} or~\cite[p.~232]{leadbetter_etal_book}.
It follows from~\eqref{eq:tech_pointwise1} that  uniformly in $a\in [A_1,A_2]$, we have
$$
\lim_{B\to\infty} B^{-2} \Lambda_B(a) = \frac{1}{2\sqrt {\pi} a^2} e^{-\kappa 2^{\frac q 2} a^{-\frac{q-2}{2}}}=\Lambda(a).
$$
To complete the proof let $B\to\infty$ in~\eqref{eq:tech1_subg}.
\end{proof}

\begin{lemma}\label{lem:Sn_B_prime_subgauss}
Let $S_n'(B)$ be defined as in~\eqref{eq:Sn_B_prime_subg}. We have
\begin{equation}\label{eq:lem:Sn_B_prime_subgauss}
\lim_{B\to\infty} \liminf_{n\to\infty} nw_n^{-1} S_n'(B)\geq e^{-\tau} \int_{A_1}^{A_2} \Lambda(a)da.
\end{equation}
\end{lemma}
\begin{proof}
Analogous to the proof of Lemma~\ref{lem:Sn_B_subgauss}.
\end{proof}
\begin{remark}
It follows from~\eqref{eq:lem:Sn_B_subgauss} and~\eqref{eq:lem:Sn_B_prime_subgauss} that in both equations we can replace inequality by equality.
\end{remark}

The next lemma is needed to estimate the ``double sum'' $S_n''(B)$.
It states that the exceedance events over different intervals become asymptotically independent with exponential decorrelation speed as the symmetric difference of the intervals gets larger.
%The next lemma states that the probability of two simultaneous exceedances over two intervals is much smaller than the probability of exceedance over one of the intervals provided that the symmetric difference of the intervals has cardinality much larger than $q_n$.
Consider two intervals $K_1=(i_1,j_1)\in\III$ and $K_2=(i_2,j_2)\in\III$ with lengths $k_1:=j_1-i_1$ and $k_2:=j_2-i_2$ satisfying $k_1,k_2\in[l_n^-,l_n^+]$ and such that $i_1,i_2, j_1,j_2\in q_n\Z$. Let $k\in\N_0$ be the length of the intersection $K:=K_1\cap K_2$. (More precisely, $K$ is the intersection of the sets $\{i_1+1,\ldots,j_1\}$ and $\{i_2+1,\ldots,j_2\}$). Assume without restriction of generality that $k_1\leq k_2$ and write $\Delta=\Delta(K_1,K_2)=k_2-k$. In some sense, $\Delta$ measures the distance between the intervals $K_1$ and $K_2$.
\begin{lemma}\label{lem:double_probab_subgauss}
%Given interval $(i_0,j_0)\in \III$ introduce the random event
Given an interval $K_0=(i_0,j_0)\in\III$  define a random event
$$
F_{n}(i_0,j_0)=\left\{\max_{(i,j)\in \TTT_{q_n}(i_0,j_0)} \ZZZ_{i,j}>u_n\right\}.
%\;\;\;
%F_{2n}=\left\{\max_{(i,j)\in T_{q_n}(i_2,k_2)} \ZZZ_{i,j}>u_n\right\}.
$$
There exist constants $C_1,C_2>0$  (depending on $A_1,A_2$ but not depending on $K_1$, $K_2$, $n$) such that for every $K_1,K_2$ as above,
$$
P_n(K_1,K_2) := \P[F_n(K_1)\cap F_n(K_2)] \leq C_1 u_n^{-1} e^{-u_n^2/2} e^{-C_2 \Delta(K_1,K_2)/q_n}.
$$
\end{lemma}
\begin{remark}
Note that $\P[F_n(K_\eps)]\sim \text{const} \cdot  u_n^{-1} e^{-u_n^2/2}$, $\eps=1,2$, by Lemma~\ref{lem:local_sub_gauss} and~\eqref{eq:Pn_s_subgauss}. The factor $e^{-C_2 \Delta(K_1,K_2)/q_n}$ provides an estimate for the dependence between the events $F_n(K_1)$ and $F_n(K_2)$.
\end{remark}
\begin{proof}[Proof of Lemma~\ref{lem:double_probab_subgauss}]
Given a finite set $I\subset \Z$ let $S_I=\sum_{m\in I} X_m$. Any interval $(i_0,j_0)\in\III$ will be identified with the finite set $\{i_0+1,\ldots,j_0\}$. In particular, we need the random variables $S_{K_1}=S_{j_1}-S_{i_1}$ and $S_{K_2}=S_{j_2}-S_{i_2}$. Introduce the random variables
\begin{align*}
D_1^-&=\max_{0\leq k<q_n} (S_{i_1}-S_{i_1-k}), & D_1^+&=\max_{0\leq k<q_n} (S_{j_1+k}-S_{j_1}),\\
D_2^-&=\max_{0\leq k<q_n} (S_{i_2}-S_{i_2-k}),  & D_2^+&=\max_{0\leq k<q_n} (S_{j_2+k}-S_{j_2}).
\end{align*}
These random variables  are corrections appearing when we extend the base intervals $K_1$ and $K_2$ by small intervals of length at most $q_n$.
With this notation  we have an inclusion of events
$$
F_{n}(K_{\eps})\subset \{S_{K_{\eps}}+D_{\eps}^++D_{\eps}^->\sqrt{k_{\eps}} u_n\},\;\;\; \eps=1,2.
%\;\;\; F_{n}\subset \{S_{K_2}+D_2^++D_2^->\sqrt{k_2} u_n\}.
%F_{n}(K_{1})\subset \{S_{K_{1}}+D_{1}^++D_{1}^->\sqrt{k_{1}} u_n\},
%\;
%F_{n}(K_{2})\subset \{S_{K_{2}}+D_{2}^++D_{2}^->\sqrt{k_{2}} u_n\}.
$$
Denote by $\bar K_1=(i_1-q_n, j_1+q_n)$ the extended version of the interval $K_1$. Note that the interval $\bar K_1$ contains all intervals from $\TTT_{q_n}(i_1,j_1)$. Let $\bar K=K_2\cap \bar K_1$ be the extended intersection of $K_1$ and $K_2$, and denote its length by $\bar k=|\bar K| < k+2q_n$. Let $\bar \Delta=k_2-\bar k$ be the length of $K_2\bsl \bar K=K_2\bsl \bar K_1$. Fix $\eps>0$. Introduce the following random events
\begin{align*}
E_{1}&=\left\{S_{\bar K}+D_2^-+D_2^+>\sqrt{\bar k} u_n + 2\sqrt{2q_n} + \eps \frac{\bar \Delta \sqrt{k_2+2q_n}}{q_n u_n}\right\},\\
E_{2}&=\left\{S_{K_1}+D_1^-+D_1^+>\sqrt{k_1} u_n\right\},\\
E_{3}&=\left\{S_{K_2}-S_{\bar K}>(\sqrt{k_2}-\sqrt{\bar k}) u_n - 2\sqrt{2q_n} -\eps \frac{\bar \Delta \sqrt{k_2+2q_n}}{q_n u_n}\right\}.
\end{align*}
Note that $F_{n}(K_1) \subset E_{2}$. We have $F_{n}(K_1)\cap F_{2}(K_2)\subset E_1\cup (E_2\cap E_3)$. By construction, the events $E_{2}$ and $E_3$ are independent. We will estimate the probabilities of $E_1$, $E_2$, $E_3$. Bringing these estimates together will complete the proof of the lemma.
First we estimate the probability of $E_2$. By Lemma~\ref{lem:local_est_sub_gauss},
\begin{equation}\label{eq:est_E2_subg}
\P[E_2]
=
\P\left[\max_{(i,j)\in \TTT_{q_n}(i_1,j_1)} \frac{S_j-S_i}{\sqrt{k_1}}>u_n\right]
\leq
C u_n^{-1} e^{-u_n^2/2}.
\end{equation}
We now estimate $\P[E_3]$.
%In the case $\bar \Delta=0$ we have $K_2=\bar K$ and hence, $\P[E_3]=0$.
Consider the case $\bar \Delta\geq 10 \sqrt A q_n$. Since $\bar k\leq k_2 \leq A \log^p n$, $u_n\sim\sqrt{2\log n}$, $q_n\sim \log^{p-1} n$, we can choose $\eps>0$ so small that the following inequality is valid:
$$
(\sqrt{k_2}-\sqrt{\bar k}) u_n - \eps \frac{\Delta \sqrt{k_2+2q_n}}{q_n u_n}
=
\frac{u_n\bar {\Delta}}{\sqrt{k_2}+\sqrt{\bar k}} - 2\sqrt{2q_n} - \eps \frac{\bar \Delta \sqrt{k_2+2q_n}}{q_n u_n}
\geq
\frac{\eps \bar{\Delta}}{\sqrt{q_n}}.
$$
Note that $\Delta \leq 3\bar \Delta$ since $\Delta\leq \bar \Delta+2q_n$ and $\bar \Delta\geq q_n$. With Lemma~\ref{lem:ld_est} and~\eqref{eq:I_subgauss} it follows that
\begin{equation}\label{eq:est_E3_subg}
\P[E_3]
\leq
\P\left[S_{\bar{\Delta}}> \frac{\eps \bar{\Delta}}{\sqrt{q_n}} \right]
\leq
\exp\left\{-\bar{\Delta} I\left(\frac{\eps}{\sqrt {q_n}}\right)\right\}
%\leq
%\exp\left\{- \frac{\eps^2 \bar \Delta}{2q_n} \right\}
\leq
c_1 \exp\left\{- \frac{c_2\Delta}{q_n}\right\}.
\end{equation}
In the case $\bar \Delta < 10 \sqrt A q_n$ we can use the trivial estimate $\P[E_3]\leq 1$ and~\eqref{eq:est_E3_subg} remains valid provided that $c_1$ is sufficiently large.

We estimate $\P[E_1]$. First we have to get rid of the correction terms $D_2^-$ and $D_2^+$. By the inequality stated in Theorem~2.4 on p.~52 in~\cite{petrov_book}, for every $t\in\R$ we have
$$
\P[D_2^{\pm}\geq t]\leq 2\P[S_{q_n}\geq t-\sqrt{2q_n}].
$$
If $X$ is any random variable which is independent of $D_2^+$ and has distribution function $F$, then
\begin{align*}
\P[X+D_2^+\geq t]
&=
\int_{\R} \P[D_2^+\geq t-s] dF(s)\\
&
\leq 2 \int_{\R} \P[S_{q_n}\geq t-s-\sqrt{2q_n}] dF(s)\\
&=
2\P[X+S_{q_n}\geq t-\sqrt{2q_n}].
\end{align*}
Applying this trick twice to $D_2^+$ and $D_2^-$ we obtain
\begin{align*}
\P[E_1]
\leq
4 \P \left [S_{\bar k + 2q_n} \geq  \sqrt{\bar k} u_n+ \eps \frac{\bar \Delta \sqrt{k_2+2q_n}}{q_n u_n} \right]
\leq
4 \P \left [\frac{S_{\bar k+2q_n}}{\sqrt{\bar k+2q_n}} \geq   u_n + \eps \frac{\bar \Delta}{q_n u_n} \right],
\end{align*}
where in the second inequality we used the relations $q_n\sim \log^{p-1} n$, $u_n\sim \sqrt{2\log n}$,  $q_n\leq \bar k\leq k_2$. (The case $\bar k=0$ can be excluded since Lemma~\ref{lem:double_sum_subgauss} holds trivially in this case due to the independence of $F_n(K_1)$ and $F_n(K_2)$). Applying to the right-hand side Theorem~\ref{theo:petrov_est} and then~\eqref{eq:I_subgauss}, we obtain
\begin{align}
\P[E_1]
\leq
C u_n^{-1} \exp\left\{- \frac 12 \left(u_n + \eps \frac{\bar \Delta}{q_n u_n}\right)^2\right\}
\leq
C_1 u_n^{-1} e^{-u_n^2/2} e^{-C_2 \Delta/q_n}. \label{eq:est_E1_subg}
\end{align}
The proof of the lemma is completed by recalling that $F_{n}(K_1)\cap F_{n}(K_2) \subset E_1\cup (E_2\cap E_3)$, where $E_2$ and $E_3$ are independent, and applying~\eqref{eq:est_E2_subg}, \eqref{eq:est_E3_subg}, \eqref{eq:est_E1_subg}.
\end{proof}

Given $B\in\N$ consider a discrete $d$-dimensional cube $Q_B=\{0,\ldots,B-1\}^d$. We can decompose the lattice $\Z^d$ into disjoint discrete cubes of the form $Bu+Q_B$, $u\in\Z^d$. Given $v_1,v_2\in\Z^d$ we write  $v_1 \sim v_2$ if $v_1$ and $v_2$ are in the same cube in this decomposition, that is if there is $u\in\Z^d$ such that $v_1,v_2\in Bu+Q_B$. Even though this is not explicit in our notation, the relation $\sim$ depends on $B$.  Given a set $V\subset \Z^d$ let $\partial_B V$ be the $B$-boundary of $V$ defined as the union of discrete cubes $Bu+Q_B$, $u\in\Z^d$, which have a non-empty intersection with both $V$ and $\Z^d\bsl V$. Let $\|\cdot\|$ be the sup-norm on $\R^d$ and $|V|$ the cardinality of a finite set $V$.
\begin{lemma}\label{lem:double_sum_general}
Let $\VVV_n\subset \Z^d$ be a sequence of finite sets such that for every $B\in\N$, $|\partial_B \VVV_n|=o(|\VVV_n|)$, as $n\to\infty$. Then, for every $\eps>0$,
$$
\lim_{B\to\infty}\limsup_{n\to\infty} \frac{1}{|\VVV_n|} \sum_{\substack{v_1,v_2\in \VVV_n\\ v_1 \nsim  v_2}} e^{-\eps \|v_1-v_2\|}=0.
$$
\end{lemma}
\begin{proof}
Take any sequence $k_B\in\N$ such that $k_B\to\infty$ but $k_B=o(B)$ as $B\to\infty$. We have a disjoint decomposition of the cube $Q_B$ into the ``kernel'' $Q_B'$ and the ``shell'' $Q_B''$ defined by
$$
Q_B'=\{k_B, \ldots, B-k_B\}^d, \;\;\; Q_B''=Q_B\bsl Q_B'.
$$
We have also a disjoint decomposition $\Z^d=Z_B'\cup Z_B''$, where $Z_B'=\cup_{u\in\Z^d} (Bu+Q_B')$ and $Z_B''=\Z^d\bsl Z_B'$.

Consider any $v_1\in Z_B'$. Then, for every $v_2\in\Z^d$ such that $v_1\nsim v_2$ we have $\|v_2-v_1\| > k(B)$. Hence,
\begin{equation}\label{eq:double_gen_pr1}
\sum_{v_2\in \VVV_n: v_1\nsim v_2} e^{-\eps \|v_1-v_2\|}\leq \sum_{v_2\in\Z^d: \|v_2\| > k_B} e^{-\eps \|v_2\|}=:C_B.
\end{equation}
Consider any $v_1\in Z_B''$. There is a constant $C<\infty$ (not depending on $v_1$) such that
\begin{equation}\label{eq:double_gen_pr2}
\sum_{v_2\in \VVV_n: v_1\nsim v_2} e^{-\eps \|v_1-v_2\|}
\leq
\sum_{v_2\in \Z^d} e^{-\eps \|v_1-v_2\|}
=
\sum_{v_2\in \Z^d} e^{-\eps \|v_2\|}
=C.
\end{equation}
 It follows from $|\partial_B \VVV_n|=o(|\VVV_n|)$ that $|Z_B'\cap \VVV_n|/|\VVV_n|\to 1$, $|Z_B''\cap \VVV_n|/|\VVV_n|\to 0$, and $|\VVV_n|\to\infty$, as $n\to\infty$. Using~\eqref{eq:double_gen_pr1} and~\eqref{eq:double_gen_pr2} we obtain
$$
\limsup_{n\to\infty}\frac{1}{|\VVV_n|} \sum_{\substack{v_1,v_2\in \VVV_n\\ v_1\nsim v_2}} e^{-\eps \|v_1-v_2\|}
\leq
\limsup_{n\to\infty}\left( C_B \frac{|Z_B'\cap \VVV_n|}{|\VVV_n|}   + C \frac{|Z_B''\cap \VVV_n|}{|\VVV_n|}\right)
=C_B.
$$
%Using  we obtain
%$$
%\limsup_{n\to\infty} \frac{1}{|\VVV_n|} \sum_{\substack{v_1,v_2\in \VVV_n\\ v_1\nsim v_2}} e^{-\eps %|v_1-v_2|}\leq S(B).
%$$
To complete the proof note that $\lim_{B\to\infty} C_B=0$ by~\eqref{eq:double_gen_pr1}.
\end{proof}
\noindent
Introduce the finite set
\begin{equation}\label{eq:VVV_n}
\VVV_n=\Z^2\cap \left\{(x,y)\in\R^2: x\in \left[0, \frac {w_n}{q_n}\right],\; y-x \in \left[\frac{l_n^-}{q_n}, \frac{l_n^+}{q_n}\right]\right\}.
\end{equation}
\begin{lemma}\label{lem:Delta_sup_norm}
There is $c>0$ such that for all $v_1,v_2\in\VVV_n$ we have $\Delta:=\Delta(v_1,v_2)\geq c\|v_1-v_2\|$, where $\Delta(v_1,v_2)$ was introduced before Lemma~\ref{lem:double_probab_subgauss}.
\end{lemma}
\begin{proof}
Let $v=(x_1,y_1)$, $v_2=(x_2,y_2)$, where $x_1<y_1$ and $x_2<y_2$. Then, $\|v_1-v_2\|=\max (|x_1-x_2|,|y_1-y_2|)$. Without restriction of generality, let $x_1\leq x_2$. If $x_2\leq y_1\leq y_2$, then $\Delta=\|v_1-v_2\|$ by definition. If $x_2 < y_2\leq y_1$, then $\Delta=|x_2-x_1|+|y_2-y_1|\geq \|v_1-v_2\|$.  Finally, if $y_1\leq x_2$, then $\Delta=\max(y_1-x_1,y_2-x_2)\geq l_n^-/q_n\geq c_1\log n$ by definition of $\VVV_n$ and $\|v_1-v_2\|\leq w_n/q_n+l_n^+/q_n\leq c_2\log n$.
\end{proof}
\begin{lemma}\label{lem:double_sum_subgauss}
Let $S_n''(B)$ be defined as in~\eqref{eq:Sn_B_2prime_subg}. We have
$$
\lim_{B\to\infty} \limsup_{n\to\infty} n w_n^{-1} S_n''(B) = 0.
$$
\end{lemma}
\begin{proof}
Take first some fixed $B\in\N$. Given a vector $v=(x,y)\in\R^2$ we write $v'=(x+1,y)$.
%Given two intervals $K_1=(i_1,j_1)\in\III$ and $K_2=(i_2,j_2)\in\III$  we write $K_1\sim K_2$ if there are $x\in\calX_n(B)$ and $l\in\calL_n(B)$ such that $(i_1,j_1)$ and $(i_2,j_2)$ are both elements of $\TTT_{Bq_n}(x,l)$.
Let $(x,y)\in Bq_n\Z^2$, $x<y$, be some interval.  Then, we can represent the discrete square $\TTT_{Bq_n}(x,y)$ as a disjoint union of $B^2$ squares of the form  $\TTT_{q_n}(q_n v')$, where
$$
v\in v_0+ \{0,\ldots,B-1\}^2,
\;\;\;
v_0=\left( \frac{x-Bq_n}{q_n},\frac{y}{q_n}\right) \in B\Z^2.
$$
We can estimate the exceedance probability over $\TTT_{Bq_n}(x,y)$ by the sum of the exceedance probabilities over the $\TTT_{q_n}$'s.
It follows from~\eqref{eq:Sn_B_2prime_subg} that
$$
S_n''(B)\leq \sum_{v_1,v_2} \P\left[\max_{(i,j)\in \TTT_{q_n}(q_nv_1')} \ZZZ_{i,j}>u_n, \max_{(i,j)\in \TTT_{q_n}(q_nv_2')} \ZZZ_{i,j}>u_n\right],
$$
where  the sum is taken over all $v_1\in \VVV_n$ and $v_2\in \VVV_n$ such that $v_1 \nsim v_2$.
Applying Lemma~\ref{lem:double_probab_subgauss} and noting that $\Delta(q_nv_1',q_nv_2')=q_n\Delta(v_1,v_2)$ we get
$$
S_n''(B)\leq c_1 u_n^{-1} e^{-u_n^2/2} \sum_{v_1,v_2} e^{-c_2 \Delta(v_1,v_2)}
\leq \frac{c_3}{n \log^{2-p} n} \sum_{v_1,v_2} e^{-c_4 \|v_1-v_2\|}.
$$
Here, $\|\cdot\|$ denotes  the sup-norm and  the second inequality follows from~\eqref{eq:def_un_subgauss} and Lemma~\ref{lem:Delta_sup_norm}.
Applying to the right-hand side Lemma~\ref{lem:double_sum_general} and noting that $|\VVV_n|\leq C w_n \log^{2-p}n$ we arrive at the required statement.
\end{proof}

\subsection{Global probability} \label{subsec:global_subg}
Recall that for $0<A_1<A_2$ we define  $l_n^-=A_1 \log^p n$ and $l_n^+=A_2 \log^p n$.
Denote by $\III_n(A_1,A_2)$ the set of all intervals $(i,j)\in\III_n$ with length $l:=j-i\in[l_n^-, l_n^+]$. Our aim is to prove Theorem~\ref{theo:main_subgauss_scales} which states that
\begin{equation}\label{eq:global_probab_subg}
\lim_{n\to\infty}\P\left[\max_{(i,j)\in \III_n(A_1,A_2)} \ZZZ_{i,j}\leq u_n\right]= \exp\left\{- e^{-\tau} \int_{A_1}^{A_2} \Lambda(a)da \right\}.
\end{equation}
We will decompose the set $\III_n(A_1,A_2)$ into sets of the form $\JJJ_n(z)$; see Lemma~\ref{lem:exc_JJJ_subgauss}. Let $w_n=[3A_2\log^p n]$. To get rid of the boundary effects choose a sequence $\eta_n>0$ such that $\eta_n=o(n)$ but $\eta_n/w_n\to \infty$. Consider the one-dimensional grids
$$
\calR_n'=[-\eta_n, n+\eta_n]\cap w_n\Z, \;\;\; \calR_n''=[\eta_n, n-\eta_n]\cap w_n\Z.
$$
Note that the sets $\JJJ_n(z)$, where $z\in \calR_n'$, are disjoint and cover $\III_n(A_1,A_2)$. Similarly, the sets $\JJJ_n(z)$, where $z\in\calR_n''$, are disjoint and contained in $\III_n(A_1,A_2)$. The exceedance probability over each $\JJJ_n(z)$ satisfies, by Lemma~\ref{lem:exc_JJJ_subgauss},
\begin{equation}\label{eq:global_proof_JJJ_subg}
\P\left[\max_{(i,j)\in \JJJ_n(z)} \ZZZ_{i,j}>u_n\right]\sim e^{-\tau} \frac{w_n}{n} \int_{A_1}^{A_2} \Lambda(a) da,\;\;\; n\to\infty.
\end{equation}
Also, $|\calR_n'|\sim |\calR_n''|\sim n/w_n$ as $n\to\infty$. If the exceedance events over $\JJJ_n(z)$ were independent, the Poisson limit theorem would immediately yield~\eqref{eq:global_probab_subg}. However, the events are dependent. In fact, the dependence is quite weak:  the exceedance event over $\JJJ_n(z)$ depends only on the two neighboring exceedance events over $\JJJ_n(z\pm w_n)$, if $n$ is large. To justify the use of the Poisson limit theorem for such finite-range dependent events we need to check that (see, e.g.,\ \cite[Thm.~1]{arratia_etal})
\begin{equation}\label{eq:global_proof_2JJJ_subg}
\P\left[\max_{(i,j)\in\JJJ_n(0)} \ZZZ_{i,j}>u_n, \max_{(i,j)\in\JJJ_n(w_n)} \ZZZ_{i,j}>u_n \right]= o\left(\frac{w_n}{n}\right),  \;\;\; n \to\infty.
\end{equation}
Since we can apply Lemma~\ref{lem:exc_JJJ_subgauss} to the set $\JJJ_n(0)\cup\JJJ_n(w_n)$ (replacing $w_n$ by $2w_n$), we have
\begin{equation}\label{eq:global_proof_22JJJ_subg}
\P\left[\max_{(i,j)\in\JJJ_n(0)\cup \JJJ_n(w_n)} \ZZZ_{i,j}>u_n  \right]\sim e^{-\tau} \frac{2w_n}{n} \int_{A_1}^{A_2} \Lambda(a) da ,  \;\;\; n \to\infty.
\end{equation}
Combining~\eqref{eq:global_proof_22JJJ_subg} and~\eqref{eq:global_proof_JJJ_subg} we obtain the required relation~\eqref{eq:global_proof_2JJJ_subg}. Thus, the use of the Poisson limit theorem is justified. The proof of~\eqref{eq:global_probab_subg} is complete.

\subsection{Non-optimal lengths}\label{subsec:non_opt_subgauss}
We will now estimate the exceedance probability over all intervals $(i,j)\in\III_n$ which are non-optimal in the sense that their length $l:=j-i$ is not between $l_n^-=A^{-1} \log^p n$ and $l_n^+=A \log^p n$, where $A>0$ is large.   Denoting by $\BBB_n(A)$ the set of all such intervals, we will show that
$$
\lim_{A\to\infty}\limsup_{n\to\infty} \P\left[\max_{(i,j)\in \BBB_n(A)} \ZZZ_{i,j}>u_n \right]=0.
$$
This means that the contribution of $\BBB_n(A)$ to $\MMM_n$ becomes negligible as $A\to\infty$. Combining this with Theorem~\ref{theo:main_subgauss_scales} proved above, we obtain Theorem~\ref{theo:main_subgauss}. We will decompose the set $\BBB_n(A)$ into three subsets $\BBB_n'$, $\BBB_n''$, $\BBB_n'''$ and estimate the exceedance probabilities over these sets in the next three lemmas.
We start by considering  very small intervals.
\begin{lemma}
Fix an arbitrary $a>0$. Let $\BBB_n'$ be the set of all intervals $(i,j)\in\III_n$ whose length  $l=j-i$ satisfies $l\leq a \log n$. Then,
$$
\lim_{n\to\infty} \P\left[\max_{(i,j)\in \BBB_n'} \ZZZ_{i,j}>u_n \right]=0.
$$
\end{lemma}
\begin{proof}
Let $l\in \N$ be such that $l\leq a\log n$. Recall from~\eqref{eq:def_un_subgauss} that $u_n\sim \sqrt{2\log n}$, as $n\to\infty$. Then, $u_n/\sqrt l > 1/\sqrt {a}$, for all large $n$.  Consequently, by~\eqref{eq:I_subgauss}, there is $\delta>0$ such that for all large $n$ and all $l\leq a\log n$,
$$
I\left(\frac{u_n}{\sqrt l}\right)>(1+2\delta)\frac{u_n^2}{2l} > (1+\delta) \frac{\log n}{l}.
$$
By Lemma~\ref{lem:ld_est} the exceedance probability for every individual interval from $\BBB_n'$ satisfies
$$
\P\left[\frac{S_l}{\sqrt l} >u_n\right]
\leq
\exp\left\{-l I\left(\frac{u_n}{\sqrt l}\right)\right\}
\leq
n^{-(1+\delta)}.
$$
Since the number of intervals in $\BBB_n'$ is at most $a n\log n$, we obtain the statement of the lemma.
\end{proof}

\begin{lemma}
Fix any $a>1$. Let $\BBB_n''=\BBB_n''(A)$ be the set of all intervals $(i,j)\in\III_n$ whose length $l=j-i$ satisfies $a\log n \leq l \leq A^{-1} \log^p n$. Then,
$$
\lim_{A\to\infty} \limsup_{n\to\infty}
\P \left[\max_{(i,j)\in \BBB_n''(A)}\ZZZ_{i,j}>u_n\right]
=
0.
$$
\end{lemma}
\begin{proof}
For $k\in\N_0$ let $\BBB_{n,k}''$ be the set of all intervals $(i,j)\in\III_n$ whose length $l$ satisfies $2^{-(k+1)}\log^p n \leq l \leq 2^{-k} \log^p n$. We can cover the set $\BBB_{n,k}''\subset \Z^2$ by disjoint discrete squares $\TTT_{r_{n,k}}(x,x+l)$, see Section~\ref{subsec:local_subgauss}, with side length $r_{n,k}:=[2^{-k} \log^{p-1} n]$, at least as long as $2^k<\log^{p-1} n$. The number of squares we need is at most $c_1 2^k (\log^{2-p} n)  n$. The exceedance probability over any such square $\TTT_{r_{n,k}}(x,x+l)$ can be estimated by Lemma~\ref{lem:local_est_sub_gauss} with $r=r_{n,k}$, $u=u_n$ and is at most
$$
%\P\left[\max_{(i,j)\in T(x,l, r_{n,k})} \ZZZ_{i,j}>u_n\right]
%\leq
c_2 u_n^{-1}\exp\left\{-\frac{u_n^2}{2}-c_3u_n^{q}l^{-\frac{q-2}{2}}\right\}
\leq
c_2 u_n^{-1} \exp\left\{-\frac{u_n^2}{2}-c_42^{k\frac{q-2}{2}}\right\}.
$$
Here, $c_2,\ldots,c_4$ do not depend on $n,k$. We can cover the set $\BBB_n''(A)$ by the sets $\BBB_{n,k}''$, where $k$ is such that $A\leq 2^k\leq a^{-1} \log^{p-1} n$. For the exceedance probability over the set $\BBB_n''(A)$ we obtain the estimate
\begin{align*}
\P \left[\max_{(i,j)\in \BBB_n''(A)}\ZZZ_{i,j}>u_n\right]
&\leq
c_5 (\log^{2-p} n) n \cdot u_n^{-1} e^{-u_n^2/2} \sum_{k=[\log_2 A]}^{\infty} 2^k e^{-c_4 2^{k\frac{q-2}2}}\\
&\leq
c_6 \sum_{k=[\log_2 A]}^{\infty} 2^k e^{-c_4 2^{k\frac{q-2}2}}.
\end{align*}
In the last inequality we have used that $u_n e^{u_n^2/2}\geq c_7  (\log^{2-p}n) n$ by~\eqref{eq:def_un_subgauss}. To complete the proof note that the right-hand side tends to $0$ as $A\to\infty$.
\end{proof}

\begin{lemma}
Let $\BBB_n'''=\BBB_n'''(A)$ be the set of all intervals $(i,j)\in\III_n$ whose length $l=j-i$ satisfies $l\geq A\log^p n$. Then,
$$
\lim_{A\to\infty}\limsup_{n\to\infty}
\P \left[\max_{(i,j)\in \BBB_n'''(A)}\ZZZ_{i,j}>u_n\right]
=
0.
$$
\end{lemma}
\begin{proof}
For $k\in\N_0$ consider the set $\BBB_{n,k}'''$ of all intervals $(i,j)\in\III_n$ with length $l$ satisfying $2^k \log^p n\leq l\leq 2^{k+1}\log^p n$. We can cover the set $\BBB_{n,k}'''\in\Z^2$ by disjoint discrete squares $\TTT_{r_{n,k}}(x,x+l)$, see Section~\ref{subsec:local_subgauss}, with side length $r_{n,k}:=[2^k \log^{p-1} n]$. We need at most $c_1 2^{-k}(\log^{2-p} n) n$ squares. Exceedance probability over any single square can be estimated by Lemma~\ref{lem:local_est_sub_gauss} by
$
c_2 u_n^{-1} e^{-u_n^2/2}.
$
For the exceedance probability over the set $\BBB_n'''(A)\subset \cup_{k=[\log_2 A]}^{\infty}\BBB_{n,k}'''$ we obtain the estimate
\begin{align*}
\P \left[\max_{(i,j)\in \BBB_n'''(A)}\ZZZ_{i,j}>u_n\right]
\leq
c_3 u_n^{-1} e^{-u_n^2/2} \cdot n \log^{2-p} n  \sum_{k=[\log_2 A]}^{\infty} 2^{-k}
\leq
c_4 \sum_{k=[\log_2 A]}^{\infty} \frac 1 {2^{k}}.
\end{align*}
The right-hand side goes to $0$ as $A\to\infty$. The proof is complete.
\end{proof}
%Take $x\in \Z$, $l>A_1\log n$, constant $B$, $q_n=Bl/\log n$. Let $\calB_n$ be the set of intervals $(i,j)\in \III$ such that $x-q_n< i\leq x$ and $x+l\leq j<x+l+q_n$.

\subsection{Proof of Proposition~\ref{prop:non_symm_bern_subgauss}}
We assume that we are in the setting of Section~\ref{subsubsec:non_symm_bern}. Let $\gamma=\sqrt{(1-p)/p}$. Note that the $X_k$'s take values $\gamma$ and $-\gamma^{-1}$. First we will show that
\begin{equation}\label{eq:I_non_symm_bern}
I(s) =
\sum_{k=1}^{\infty} \frac{\sigma s^{2k}}{4k(2k-1)}\left(\gamma^{2k-1}+\gamma^{1-2k}\right)
+ \sum_{k=1}^{\infty} \frac{\sigma s^{2k+1}}{4k(2k+1)}\left(\gamma^{-2k}-\gamma^{2k}\right).
\end{equation}
Taking into account~\eqref{eq:non_symm_bern_varphi} and solving $\varphi'(\tilde t) = s$, we have
$$
\tilde t  = \frac\sigma2\left[\log\left(1+\gamma s\right) - \log\left(1-\frac {s}{\gamma}\right)\right],
$$
for $s\in (-\gamma^{-1}, \gamma)$.
For $s$ in this range,
\begin{equation}\label{eq:I_non_symm_bern1}
I(s) = s\tilde t - \varphi(\tilde t)
= p\left(1+\gamma s\right) \log\left(1+\gamma s\right) + (1-p)\left(1-\frac s {\gamma}\right)\log\left(1-\frac{s}{\gamma}\right).
\end{equation}
For $s$ outside the interval $(-\gamma^{-1}, \gamma)$, we have $I(s)=+\infty$. By Taylor's expansion,
$$
\log\left(1+\gamma s\right)
=
\sum_{k=1}^{\infty}(-1)^{k+1}\gamma^k \frac{s^k}k,
\;\;\;
\log\left(1-\frac{s}{\gamma}\right)
=
-\sum_{k=1}^{\infty} \gamma^{-k} \frac{s^k}k.
$$
%Then,
%\begin{multline*}
%\frac{s\sigma}2\left[\log\left(1+\frac{s\sigma}{2p}\right) - %\log\left(1-\frac{s\sigma}{2(1-p)}\right)\right] \\
%= \sum_{k=1}^\infty\frac{s^{2k}}{2k-1}\left[\left(\frac\sigma{2p}\right)^{2k-1} + %\left(\frac\sigma{2(1-p)})^{2k-1}\right]\frac\sigma2 \\
%+ \sum_{k=1}^\infty\frac{s^{2k+1}}{2k}\left[\left(\frac\sigma{2(1-p)}\right)^{2k} - %\left(\frac\sigma{2p}\right)^{2k}\right]\frac\sigma2,
%\end{multline*}
%and
%\begin{multline*}
%p\log\left(1+\frac{s\sigma}{2p}\right) + (1-p)\log\left(1-\frac{s\sigma}{2(1-p)}\right) \\
%= \sum_{k=1}^{\infty}  \frac{s^{2k-1}}{2k-1}\left[\left(\frac\sigma{2p}\right)^{2k-2} - %\left(\frac\sigma{2(1-p)}\right)^{2k-2}]\frac\sigma2 \\
%- \sum_{k=1}^{\infty} \frac{s^{2k}}{2k}\left[\left(\frac\sigma{2p}\right)^{2k-1} + %\left(\frac\sigma{2(1-p)}\right)^{2k-1}\right]\frac\sigma2.
%\end{multline*}
Inserting this into~\eqref{eq:I_non_symm_bern1}, we obtain~\eqref{eq:I_non_symm_bern}. If $p\in (1/2,1)$, then $0<\gamma<1$ and hence, it follows from~\eqref{eq:I_non_symm_bern} that all coefficients in the Taylor expansion of $I(s)-(s^2/2)$ are non-negative (and in fact, the coefficient of $s^3$ is strictly positive). It follows that $I(s)>s^2/2$ for all $s>0$. This implies that~\eqref{eq:I_subgauss} holds. By Proposition~\ref{prop:dual_conditions_subg}, this implies~\eqref{eq:varphi_subgauss}.
Together with the Taylor expansion in~\eqref{eq:non_symm_bern_varphi}, this shows that we are in the superlogarithmic case with $q=3$.

%********************************************************%
%************* PROOF LOGARITHMIC   ********************%
%********************************************************%

\section{Proof in the logarithmic case}\label{sec:proof_supergauss}
Our aim in this section is to prove Theorems~\ref{theo:main_supergauss}, \ref{theo:main_supergauss_lattice}, \ref{theo:main_supergauss_scales}.
%Let $X_1,X_2,\ldots$ be i.i.d.\ random variables with $\E X_k=0$, $\Var X_k=1$.
Assume that conditions~\eqref{eq:varphi_def}, \eqref{eq:varphi_supergauss}, \eqref{eq:t*_unique} hold. Fix $\tau\in\R$ and define the normalizing sequence $u_n=u_n(\tau)>0$ by
\begin{equation}\label{eq:def_un_supergauss}
u_n^2=2m_*(\log n+\tau).
\end{equation}
Our aim is to compute the limit of $\P[\MMM_n\leq u_n]$, as $n\to\infty$.

\subsection{Dual conditions}
First of all, we need to replace conditions~\eqref{eq:varphi_supergauss} and~\eqref{eq:t*_unique} by their Legendre--Fenchel conjugates. We will assume that there is $s_*>0$ such that
\begin{equation}\label{eq:I_supergauss}
\frac{I(s_*)}{s_*^2/2}= \frac 1 {m_*}<1
\end{equation}
and, additionally, for every $\eps>0$,
\begin{equation}\label{eq:s*_unique}
\inf_{0<s<s_*-\eps}  \frac{I(s)}{s^2/2} > \frac 1 {m_*}
\;\;\;\text{and}\;\;\;
\inf_{s>s_*+\eps}  \frac{I(s)}{s^2/2} > \frac 1 {m_*}.
\end{equation}
\begin{proposition}\label{prop:dual_conditions_superg}
Suppose that~\eqref{eq:varphi_def} holds. Then, conditions~\eqref{eq:varphi_supergauss} and~\eqref{eq:t*_unique} imply conditions~\eqref{eq:I_supergauss} and~\eqref{eq:s*_unique} and vice versa. Furthermore, if these conditions hold, then we have $\varphi(t_*)=I(s_*)=s_*t_*/2$ and
\begin{equation}\label{eq:legendre_fenchel_duality}
s_*=\varphi'(t_*)=t_*m_*,
\;\;\;
t_*=I'(s_*)=\frac{s_*}{m_*},
\;\;\;
I''(s_*)\varphi''(t_*)=1.
\end{equation}
\end{proposition}
\begin{proof}
Assume that~\eqref{eq:varphi_supergauss} and~\eqref{eq:t*_unique} hold. Define $s_*=\varphi'(t_*)$. We will show that~\eqref{eq:I_supergauss} and~\eqref{eq:s*_unique} hold. By Legendre--Fenchel duality, $I'$ is the inverse function of $\varphi'$ and vice versa. Hence, $t_*=I'(s_*)$. The point $t_*$ is the unique maximum of the function $\frac{\varphi(t)}{t^2/2}$ by~\eqref{eq:varphi_supergauss} and~\eqref{eq:t*_unique}. The derivative of this function vanishes at $t_*$ and hence,  $\varphi'(t_*)t_*=2\varphi(t_*)$. In view of~\eqref{eq:varphi_supergauss} this implies that $s_*=t_* m_*$.
Since the maximum of $s_*t-\varphi(t)$ is attained at $t=t_*$, we have, see~\eqref{eq:def_I},
%for $s=s_*$ the maximum in~\eqref{eq:def_I} is attained at $t=t_*$, we have
$$
I(s_*)=\varphi'(t_*)t_* - \varphi(t_*)=\varphi(t_*)=\frac{m_* t_*^2}2=\frac{s_*t_*}{2}=\frac{s_*^2}{2m_*}.
$$
The inverse function of $\varphi'$ is $I'$.  Taking the derivative we obtain $I''(s_*)\varphi''(t_*)=1$. This proves~\eqref{eq:I_supergauss} and~\eqref{eq:legendre_fenchel_duality}.

We will now show that condition~\eqref{eq:s*_unique} is fulfilled. Fix $\eps>0$. Denote by $S_\eps(u)$ the set $\{s>0: |s-u|>\eps\}$. By~\eqref{eq:t*_unique}, for every $\delta>0$ we there exists $c=c(\delta)<1$ such that $\varphi(t)\leq c m_* \frac{t^2}{2}$  for all $t\in S_{\delta}(t_*)$. Then, for every $s\in S_{\eps}(s_*)$ and every $\delta>0$,
\begin{equation}\label{eq:tech_ineq1}
I(s)
=
\sup_{t\geq 0} (st-\varphi(t))
\geq
\sup_{t\in S_{\delta}(t_*)} (st-\varphi(t))
\geq
\sup_{t\in S_{\delta}(t_*)} \left(st- c(\delta)m_*\frac{t^2}{2}\right).
\end{equation}
The supremum of $st- c(\delta)m_*\frac{t^2}{2}$ is attained at $t=\frac s{c(\delta)m_*}$. However, we have to check that $t\in S_{\delta}(t_*)$. Recall that $|s-s_*|>\eps$. It follows that
$$
|t-t_*|=\left|\frac s{c(\delta)m_*}-t_*\Big|=\Big|\frac{s-c(\delta)s_*}{c(\delta)m_*}\right|>\delta,
$$
where the last inequality holds if $\delta=\delta(\eps)>0$ is sufficiently small. (Note that $\lim_{\delta\downarrow 0}c(\delta)=1$). In this case, $t\in S_{\delta}(t_*)$. It follows from~\eqref{eq:tech_ineq1} that $I(s)\geq \frac{s^2}{2c(\delta) m_*}$ for all $s\in S_{\eps}(s_*)$. This proves~\eqref{eq:s*_unique}. The proof that~\eqref{eq:I_supergauss} and~\eqref{eq:s*_unique} imply~\eqref{eq:varphi_supergauss} and~\eqref{eq:t*_unique} is analogous, by the Legendre--Fenchel duality.
\end{proof}

\subsection{Individual probability}
In the sequel, we assume that conditions~\eqref{eq:varphi_def}, \eqref{eq:I_supergauss}, \eqref{eq:s*_unique} hold.
In this section we compute asymptotically the exceedance probability for the value attained by the random field $\ZZZ_{i,j}$ at some individual point $(i,j)\in \III$. We focus here on intervals whose length is close to the optimal length $d_* \log n$, where
\begin{equation}\label{eq:d_*}
d_* = \frac {1}{\varphi(t_*)} = \frac{1}{I(s_*)} = \frac{2m_*}{s_*^{2}}.
\end{equation}
It turns out that the exceedance probability remains the same, up to a constant factor, if we allow fluctuations of the interval length of order $O(\sqrt {\log n})$ and fluctuations of the threshold of order $O(u_n^{-1})$.
\begin{lemma}\label{lem:one_point_super_gauss}
Assume that $X_1$ is non-lattice. Let $l_n\in\N$ be any sequence such that $l_n=d_* \log n + a\sqrt {\log n}+o(\sqrt {\log n})$, for some $a\in\R$, as $n\to\infty$. Fix $s\in\R$. Then, as $n\to\infty$,
$$
P_{n}(s)
:=
\P\left[\frac{S_{l_n}}{\sqrt {l_n}}>u_n-\frac{s}{u_n}\right]\sim \frac{\sqrt{m_*}}{2\sqrt{\pi}\sigma_*}\cdot e^{\frac s {m_*}-\frac{\beta_*^2 a^2}{2}}
\cdot  \frac{e^{-\tau}}{n\sqrt{\log n}}.
$$
Here, $\sigma_*=\sqrt{\varphi''(t_*)}$ and $\beta_*^2=\frac {s_*^4}{8 m_*} (\frac 1 {\sigma_*^{2}}-\frac 1 {m_*})>0$.
\end{lemma}
\begin{proof}
We are going to apply Theorem~\ref{theo:large_dev} with $\alpha:=\lim_{n\to\infty} \frac{u_n}{\sqrt {l_n}}=s_*$. Note that $I'(\alpha)=t_*=s_*/m_*$ by Proposition~\ref{prop:dual_conditions_superg}, and $\sigma^2(\alpha)=\varphi''(t_*)=\sigma_*^2$. We obtain, by Theorem~\ref{theo:large_dev},
\begin{equation}\label{eq:Pns_proof_ind_supergauss}
P_{n}(s) \sim \frac{\sqrt{m_*}}{2\sigma_* \sqrt{\pi \log n}}  \exp\left\{-l_n I\left(\frac{u_n-\frac{s}{u_n}}{\sqrt {l_n}}\right)\right\}.
\end{equation}
Next we develop the term under the sign of exponential in~\eqref{eq:Pns_proof_ind_supergauss} into a Taylor series. Consider the function $J(v)=vI(1/\sqrt v)$, $v>0$. By assumptions~\eqref{eq:I_supergauss}, \eqref{eq:s*_unique} it has a unique minimum at $v_*:=s_*^{-2}$.
The first two derivatives of $J$ are given by
\begin{equation}\label{eq:J_derivatives0}
J'(v)=I\left(\frac{1}{\sqrt v}\right)- \frac{1}{2\sqrt v} I'\left(\frac{1}{\sqrt v}\right),
\;\;\;
J''(v)=\frac{1}{4v^2}I''\left(\frac{1}{\sqrt v}\right)- \frac{1}{4v^{3/2}} I'\left(\frac{1}{\sqrt v}\right).
\end{equation}
For the values of $J$ and its derivatives at $v=v_*$ we obtain
\begin{equation}\label{eq:J_derivatives}
J(v_*)=\frac 1{2 m_*},
\;\;\;
J'(v_*)=0,
\;\;\;
J''(v_*)=\frac{s_*^4}4 \left(\frac 1 {\sigma_*^{2}}-\frac 1 {m_*}\right)>0.
\end{equation}
Note in passing that since $J$ attains a minimum at $v_*$, we have $J''(v_*)>0$. This proves that $\beta_*^2$ is indeed positive.
Now consider
$$
v_n
:=\frac{l_n}{(u_n-\frac{s}{u_n})^2}=\frac{d_*\log n+a\sqrt{\log n}+o(\sqrt{\log n})}{2m_*\log n+o(\sqrt {\log n})}
=
v_*+\frac{a+o(1)}{2m_*\sqrt{\log n}}. %+ o\left(\frac{1}{\sqrt{\log n}}\right).
$$
Expanding $J$ into a Taylor series at $v=v_*$, we obtain
\begin{align*}
l_n I\left(\frac{u_n-\frac{s}{u_n}}{\sqrt {l_n}}\right)
&=
J(v_n) \left(u_n-\frac{s}{u_n}\right)^2\\
&=
\left(\frac{1}{2m_*}+ \frac 12 \cdot \frac{a^2}{4m_*^2 \log n} \cdot J''(v_*)\right) (2m_*(\log n+\tau)-2s)+o(1)\\
&=
\log n+\tau+\frac {\beta_*^2 a^2}{2}- \frac{s}{m_*}+o(1).
\end{align*}
To complete the proof insert this into~\eqref{eq:Pns_proof_ind_supergauss}.
\end{proof}

\subsection{Local probability}\label{subsec:local_supergauss}
Next we compute the exceedance probability over a discrete square in the space of intervals. Recall from Section~\ref{subsec:local_subgauss} that for an interval $(x,y)\in\III$ of length $l:=y-x$ and $B\in\N$ we define $\TTT_B(x,y)$ to be the set of all intervals $(i,j)\in\III$ such that
$$
x-B < i \leq x \text{ and } y \leq  j < y + B.
$$
The set $\TTT_B(x,y)$ is a discrete square with side length $B$ in $\III\subset \Z^2$. Its right bottom point is the ``base interval'' $(x,y)$ which is contained in all other intervals belonging to $\TTT_B(x,y)$.
\begin{lemma}\label{lem:local_super_gauss}
Assume that $X_1$ is non-lattice.  Fix $B\in\N$ and let $l_n\in\N$ be a sequence such that $l_n=d_*\log n + a \sqrt{\log n}+o(\sqrt {\log n})$, for some $a\in\R$. Write $\TTT_n=\TTT_B(x,x+l_n)$.  Then, as $n\to \infty$,
\begin{equation}
Q_{n}:=\P\left[\max_{(i,j)\in \TTT_n} \ZZZ_{i,j} > u_n \right]
\sim
P_{n}(0) \cdot
\left\{
1+  H_*^2(B)
\right\},
\end{equation}
where $P_{n}(0)$ is as in Lemma~\ref{lem:one_point_super_gauss} and the function $H_*:\N\to (0,\infty)$ is defined by
\begin{equation}
H_*(B)=\E \left[\max _{k=0,\ldots,B-1} e^{t_* S_k-k\varphi(t_*)}\right].
\end{equation}
\end{lemma}
\begin{proof}
The proof follows the same idea as the proof of Lemma~\ref{lem:local_sub_gauss}, but the incremental process will be approximated by a discrete-time random walk rather than by a Brownian motion. Let $S_k^{(1)}$ and $S_k^{(2)}$ be independent random walks defined as in the proof of Lemma~\ref{lem:local_sub_gauss}.  Define a random variable $V_n$ by  $\ZZZ_{x,x+l_n}=u_n-u_n^{-1}V_n$.
Every interval from $\TTT_n$ has the form $(x-k_1,x+l_n+k_2)$ for some integers $0\leq k_1,k_2 < B$, hence %a distributional equality
$$
Q_n=\P\left[\max_{0\leq k_1,k_2<B} \frac{(u_n-\frac{V_n}{u_n})\sqrt{l_n}+S_{k_1}^{(1)} + S_{k_2}^{(2)}} {\sqrt{l_n+k_1+k_2}}>u_n\right].
$$
Conditioning on $V_n=s$  and integrating over  $s$, we obtain
\begin{equation}\label{eq:Q_n_dec_supergauss}
Q_n
=
P_n(0) + \int_{0}^{\infty} G_n(s) d\mu_n(s),
\end{equation}
where $\mu_n$ is the probability distribution of $V_n$ and $G_n$ is a non-increasing function defined by
\begin{equation}\label{eq:G_ns_expr_supergauss}
G_n(s)
=
\P\left[\max_{0\leq k_1,k_2 < B}
\left(S_{k_1}^{(1)} + S_{k_2}^{(2)}-u_n(\sqrt{l_n+k_1+k_2}-\sqrt{l_n}) -\frac{s\sqrt{l_n}}{u_n}\right)> 0 \right].
\end{equation}
By Lemma~\ref{lem:one_point_super_gauss}, for every $s\geq 0$,
\begin{equation}\label{eq:asympt_F_n_superg}
\lim_{n\to\infty} \frac{\mu_n([0,s))}{P_n(0)}= \lim_{n\to\infty} \frac{P_n(s)}{P_n(0)}= e^{\frac s{m_*}}.
\end{equation}
Let $s_n$ be any sequence converging to $s\geq 0$. We compute $\lim_{n\to\infty} G_n(s_n)$.
Let  $f_n(k_1,k_2;s_n)$ be a function given by
$$
f_n(k_1,k_2;s_n)=u_n(\sqrt{l_n+k_1+k_2}-\sqrt{l_n})+ \frac{s_n \sqrt{l_n}}{u_n}.
$$
Recall that $u_n\sim \sqrt{2m_*\log n}$ and $l_n\sim d_* \log n$, as $n\to\infty$.
An elementary calculus shows that
$$
\lim_{n\to\infty} f_n(k_1,k_2;s_n)=\frac{s_*}{2}(k_1+k_2)+\frac{s}{s_*}.
$$
Since the a.s.\ convergence  implies the distributional convergence, we obtain that
$$
\max_{0\leq k_1,k_2<B}(S_{k_1}^{(1)} + S_{k_2}^{(2)}- f_n(k_1,k_2;s_n))
\todistr
\frac{U_1+U_2}{t_*}- \frac{s}{s_*},
$$
where $U_1,U_2$ are two independent copies of the random variable
$$
U:=
\max_{k=0,\ldots, B-1} (t_*S_k-k\varphi(t_*)).
$$
It follows that for all but countably many  $s\geq 0$, and all sequences $s_n\to s$,
\begin{equation}\label{eq:asympt_G_n_superg}
\lim_{n\to\infty} G_n(s_n) = \P\left[U_1+U_2>\frac s{m_*}\right].
\end{equation}
Assuming for a moment that interchanging the limit and the integral is justified, we obtain from~\eqref{eq:asympt_F_n_superg} and~\eqref{eq:asympt_G_n_superg} that
\begin{equation}\label{eq:lim_Gn_dmu_n_supergauss}
\lim_{n\to\infty} \int_{0}^{\infty} G_n(s) \frac{d\mu_n(s)}{P_n(0)}
=
\int_0^{\infty} \P\left[U_1+U_2>\frac s{m_*}\right] e^{\frac s{m_*}} \frac{ds}{m_*}
%=
%m_* \cdot \E e^{U_1+U_2}
=
(\E e^{U})^2.
\end{equation}
Inserting this into~\eqref{eq:Q_n_dec_supergauss} completes the proof of Lemma~\ref{lem:local_super_gauss}.

The first equality in~\eqref{eq:lim_Gn_dmu_n_supergauss} will be justified using Lemma~\ref{lem:dom_conv}. To verify its last condition we have to obtain uniform estimates on $G_n$ and $\mu_n$; see Remark~\ref{rem:dom_conv}.   Continuing~\eqref{eq:G_ns_expr_supergauss} and recalling that $\sqrt{l_n}\sim u_n/s_*$, we obtain that for all large $n$, and all $s\geq 0$,
$$
G_n(s)
\leq
\P\left[\max_{0 \leq k_1,k_2 < B} (S_{k_1}^{(1)}+S_{k_2}^{(2)}) > \frac{s\sqrt{l_n}}{u_n}\right]
%\leq \P\left[\max_{k=0,\ldots,B-1} S_{k}^{(1)} > \frac{s\sqrt{l_n}}{2u_n}\right]
\leq
2 \P\left[\max_{0 \leq k < B} S_{k}^{(1)} > \frac{s}{3s_*}\right].
$$
By Lemma~\ref{lem:ld_est} and~\eqref{eq:I_supergauss} we obtain that for some constants $c_1,c_2>0$, all large $n$ and all $s\geq 0$,
\begin{equation}\label{eq:G_ns_uniform_supergauss}
G_n(s)\leq 2\sum_{k=0}^{B-1} \exp\left\{-k I\left(\frac{s}{3s_* k}\right)\right\}
\leq c_1 e^{-c_2s^2}.
\end{equation}
Now we bound $\mu_n([0,s))$. Let first $s\in [u_n^2/2, u_n^2]$. Using Lemma~\ref{lem:ld_est} and~\eqref{eq:I_supergauss} we obtain that %for some $\delta>0$,
$$
\P\left[\frac{S_{l_n}}{\sqrt {l_n}}>u_n-\frac{s}{u_n}\right]
\leq
\exp\left\{-l_n I\left(\frac{u_n-\frac{s}{u_n}}{\sqrt {l_n}} \right)\right\}\\
\leq
\exp\left\{-\frac{1}{2m_*}\left(u_n-\frac{s}{u_n}\right)^2\right\}.
$$
Let now $s\in [0,u_n^2/2]$. Using Theorem~\ref{theo:petrov_est} and~\eqref{eq:I_supergauss} we obtain
%\begin{align*}
$$
\P\left[\frac{S_{l_n}}{\sqrt {l_n}}>u_n-\frac{s}{u_n}\right]
%\leq
%C u_n^{-1}\exp\left\{-l_n I\left(\frac{u_n-\frac{s}{u_n}}{\sqrt {l_n}} \right)\right\}
\leq
C u_n^{-1}\exp\left\{-\frac{1}{2m_*}\left(u_n-\frac{s}{u_n}\right)^2\right\}.
$$
For $s>u_n^2$ we can estimate the probability by $1$. Combining all cases we obtain
$$
\mu_n([0,s))
=
\P\left[\frac{S_{l_n}}{\sqrt {l_n}}>u_n-\frac{s}{u_n}\right]
\leq
C u_n^{-1} \exp\left\{-\frac{u_n^2}{2m_*}+ \frac{s}{m_*}\right\}.
$$
Together with Lemma~\ref{lem:one_point_super_gauss} this implies that $\mu_n([0,s])\leq C e^{s/m_*} P_n(0)$ for all large $n\in\N$ and $s\geq 0$. Conditions of Remark~\ref{rem:dom_conv} are thus verified.
%&\leq
%C u_n^{-1}\exp\left\{-\frac{u_n^2}{2m_*} + \frac{s}{m_*}\right\}.
%\end{align*}
\end{proof}

\subsection{Estimating the double sum}
Given real numbers $A_1<A_2$ define  $l_n^-=d_*\log n+ A_1\sqrt{\log n}$ and $l_n^+=d_*\log n+ A_2\sqrt{\log n}$. The aim of this section is to prove the following result.
\begin{lemma}\label{lem:exc_JJJ_supergauss}
Assume that $X_1$ is non-lattice. Let $w_n\to\infty$ be any integer sequence such that $w_n=O(\log n)$. For $z\in\Z$ let $\JJJ_n(z)$ be the set of all intervals $(i,j)\in\III$ such that $z\leq i<z+w_n$ and $j-i\in [l_n^-, l_n^+]$.
Then, as $n\to\infty$,
\begin{equation}\label{eq:lem:exc_JJJ_supergauss}
\P\left[\max_{(i,j)\in \JJJ_n(z)} \ZZZ_{i,j}>u_n\right]\sim e^{-\tau} \frac{w_n}{n} \int_{A_1}^{A_2} \Theta(a) da,
\end{equation}
where $\Theta(a)=\frac{\sqrt{m_*}H_*^2}{2\sqrt {\pi} \sigma_*} e^{-\frac{\beta_*^2 a^2}2}$, $a\in\R$, and the constant $H_*$ is given by
\begin{equation}\label{eq:H*}
H_*=\lim_{B\to\infty} \frac{H_*(B)} B = \lim_{B\to\infty} \frac 1B \E \left[ \max _{k=0,\ldots,B-1} e^{t_* S_k-k\varphi(t_*)}\right]\in (0,1).
\end{equation}
\end{lemma}
\begin{proof}
The existence of the limit in~\eqref{eq:H*} follows from by taking $Y_k=t_*X_k-\varphi(t_*)$ in Lemma~\ref{lem:pickands_const_supergauss}, below. In fact, \eqref{eq:H*} can be also obtained as a byproduct of the double sum argument presented below.
We prove~\eqref{eq:lem:exc_JJJ_supergauss}. Without restriction of generality, let $z=0$. To get rid of the boundary effects we introduce two sequences $\eps_n\to \infty$ and $\delta_n\to\infty$ such that $\eps_n=o(w_n)$ and $\delta_n=o(\sqrt{\log n})$ as $n\to\infty$. Take some $B\in\N$ and let $q_n=1$. Then, with the same notation as in~\eqref{eq:calXn}, \eqref{eq:calXn_prime}, \eqref{eq:Sn_B_subg}, \eqref{eq:Sn_B_prime_subg}, \eqref{eq:Sn_B_2prime_subg}, we have the Bonferroni inequality
%Introduce the one-dimensional discrete grids with mesh size $B$:
%\begin{align*}
%\calX_n(B)&= B\Z\cap [-\eps_n, w_n+\eps_n], &  \calL_n(B)&=B\Z\cap [l_n^--\delta_n, l_n^++\delta_n], \\
%\calX_n'(B)&= B\Z\cap [\eps_n, w_n-\eps_n], &  \calL_n'(B)&=B\Z\cap [l_n^-+\delta_n, l_n^+-\delta_n].
%\end{align*}
%The discrete squares $\TTT_B(x,l)$ (which were defined in Section~\ref{subsec:local_supergauss}), where $x\in \calX_n(B)$ and $l\in \calL_n(B)$, are disjoint and cover the set $\JJJ_n(0)$. Similarly, the discrete squares $\TTT_B(x,l)$, where $x\in \calX_n'(B)$ and $l\in \calL_n'(B)$ are disjoint and contained in $\JJJ_n(0)$. By the Bonferroni inequality, we have, for every $B\in\N$,
\begin{equation}\label{eq:bonferroni_supergauss}
S_n'(B)-S_n''(B)\leq \P\left[\max_{(i,j)\in \JJJ_n(0)} \ZZZ_{i,j}>u_n\right]\leq S_n(B).
\end{equation}
%where
%\begin{align}
%S_n(B)&=\sum_{x\in\calX_n(B)}\sum_{l\in\calL_n(B)} \P \left[\max_{(i,j)\in \TTT_B(x,l)} \ZZZ_{i,j}>u_n\right],
%\label{eq:Sn_B_superg}\\
%S_n'(B)&=\sum_{x\in\calX_n'(B)}\sum_{l\in\calL_n'(B)} \P \left[\max_{(i,j)\in \TTT_B(x,l)} \ZZZ_{i,j}>u_n\right],
%\label{eq:Sn_B_prime_superg}\\
%S_n''(B)&=\sum \P \left[\max_{(i,j)\in \TTT_B(x_1,l_1)} \ZZZ_{i,j}>u_n , \max_{(i,j)\in \TTT_B(x_1,l_1)} \ZZZ_{i,j}>u_n \right]
%\label{eq:Sn_B_2prime_superg}
%\end{align}
%and in~\eqref{eq:Sn_B_2prime_superg} the sum is taken over all pairs $(x_1,l_1)\neq (x_2,l_2)$ such that $x_1,x_2\in \calX_n'(B)$ and $l_1,l_2\in \calL_n'(B)$.
The statement of Lemma~\ref{lem:exc_JJJ_supergauss} follows by letting $n\to\infty$ and then $B\to\infty$ in~\eqref{eq:bonferroni_supergauss} and applying  Lemmas~\ref{lem:Sn_B_supergauss}, \ref{lem:Sn_B_prime_supergauss}, \ref{lem:double_sum_supergauss} which we will prove below.
\end{proof}

%The following lemmas complete the proof of Lemma~\ref{lem:exc_JJJ_supergauss}.
%\begin{lemma}\label{lem:H_*_supergauss}
%The following limit exists in $(0,\infty)$:
%$$
%H_*:=\lim_{B\to\infty} \frac{H_*(B)} B = \lim_{B\to\infty} \frac 1B \E \exp \max _{k=0,\ldots,B-1} (t_* S_k-k\varphi(t_*)).
%$$
%\end{lemma}

Recall that $w_n\to\infty$ is an integer sequence such that $w_n=O(\log n)$ as $n\to\infty$.
\begin{lemma}\label{lem:Sn_B_supergauss}
Let $S_n(B)$ be defined as in~\eqref{eq:Sn_B_subg} with $q_n=1$. We have
\begin{equation}\label{eq:lem:Sn_B_supergauss}
\lim_{B\to\infty} \limsup_{n\to\infty} nw_n^{-1} S_n(B) \leq  e^{-\tau} \int_{A_1}^{A_2} \Theta(a)da.
\end{equation}
\end{lemma}
\begin{proof}
Since the probability in the right-hand side of~\eqref{eq:Sn_B_subg} does not depend on $x$, we have
$$
S_n(B)=\frac{w_n+o(w_n)}{B}\sum_{l\in\calL_n(B)} \P \left[\max_{(i,j)\in \TTT_B(0,l)} \ZZZ_{i,j}>u_n\right],
$$
where $\calL_n(B)=B\Z\cap [l_n^--\delta_n, l_n^++\delta_n]$.
The idea is  to apply to each probability Lemma~\ref{lem:local_super_gauss} and replace Riemann sums by Riemann integrals. Introduce the function
$$
\theta_{n,B}(a)=n\sqrt{\log n} \cdot \P \left[\max_{(i,j)\in \TTT_B(0,l_{n,B}(a))} \ZZZ_{i,j}>u_n\right], \;\;\; a\in\R,
$$
where $l_{n,B}(a)=\max\{l\in B\Z: l\leq d_*\log n+a\sqrt{\log n}\}$. The function $\theta_{n,B}(a)$ is locally constant and its constancy intervals have length $B/\sqrt{\log n}$.
It follows that
\begin{equation}\label{eq:tech22}
S_n(B)\leq \frac{w_n+o(w_n)}{B^2 n} \int_{A_1-\frac{2\delta_n}{\sqrt {\log n}}}^{A_2+\frac{2\delta_n}{\sqrt {\log n}}}
\theta_{n,B}(a)da.
\end{equation}
For every fixed $a\in\R$, the sequence $l_n=l_{n,B}(a)$ satisfies the assumption of Lemma~\ref{lem:local_super_gauss}. By Lemmas~\ref{lem:local_super_gauss} and~\ref{lem:one_point_super_gauss}, for every $a\in\R$,
$$
\lim_{n\to\infty} \theta_{n,B}(a) = e^{-\tau} \Theta_B(a), \;\;\; \Theta_B(a)= \frac{\sqrt {m_*}}{2\sqrt{\pi}\sigma_*} e^{-\frac{\beta_*^2 a^2}{2}}(1+ H_*^2(B)).
$$
We also need an estimate for $\theta_{n,B}(a)$ which is uniform in $a$. Assume that $a\in [-c,c]$, for some $c>0$.  For every interval $(i,j) \in \TTT_B(0,l_{n,B}(a))$ of length $l$ we have, by Theorem~\ref{theo:petrov_est} and~\eqref{eq:I_supergauss},
$$
\P[\ZZZ_{i,j}>u_n]\leq C u_n^{-1} \exp\left\{- lI\left(\frac{u_n}{\sqrt l}\right)\right\}
\leq
\frac{C}{\sqrt {\log n}} \exp\left\{-\frac{u_n^2}{2m_*}\right\}
\leq
\frac{C}{n\sqrt {\log n}}.
$$
Since $\TTT_B(0,l_{n,B}(a))$ consists of $B^2$ intervals, we obtain that $\theta_{n,B}(a)\leq C$ for all $a\in [-c,c]$, where $C$ does not depend on $n$ and $a$.
%$$
%\theta_{n,B}(a) \leq C_1 n\sqrt{\log n}\cdot u_n^{-1} \exp\left\{- c \right\} \leq C_2.
%$$
%Also, $|\calX_n(B)|\sim w_n/B$ and hence,
Taking the limit as $n\to\infty$ in~\eqref{eq:tech22} and applying the dominated convergence theorem, we obtain
\begin{equation}\label{eq:tech1}
\limsup_{n\to\infty} nw_n^{-1}S_n(B) \leq e^{-\tau} \int_{A_1}^{A_2}  B^{-2}  \Theta_B(a) da.
\end{equation}
This holds for every $B\in\N$. We let $B\to\infty$. The limit $H_*:=\lim_{B\to\infty} B^{-1}H_*(B)\in (0,\infty)$ exists by~\eqref{eq:H*}. Hence, $\lim_{B\to\infty} B^{-2}\Theta_B(a)=\Theta(a)$ uniformly in $a\in\R$, where $\Theta(a)$ is defined as in Lemma~\ref{lem:exc_JJJ_supergauss}. To complete the proof let $B\to\infty$ in~\eqref{eq:tech1}.
\end{proof}

\begin{lemma}\label{lem:Sn_B_prime_supergauss}
Let $S_n'(B)$ be defined as in~\eqref{eq:Sn_B_prime_subg} with $q_n=1$. We have
\begin{equation}\label{eq:lem:Sn_B_prime_supergauss}
\lim_{B\to\infty} \liminf_{n\to\infty} nw_n^{-1} S_n'(B)\geq e^{-\tau} \int_{A_1}^{A_2} \Theta(a)da.
\end{equation}
\end{lemma}
\begin{proof}
Analogous to the proof of Lemma~\ref{lem:Sn_B_supergauss}.
\end{proof}
\begin{remark}
It follows from~\eqref{eq:lem:Sn_B_supergauss} and~\eqref{eq:lem:Sn_B_prime_supergauss} that in both equations we can replace inequality by equality.
\end{remark}

The next lemma is needed to estimate the ``double sum'' $S_n''(B)$.
It provides an estimate for the correlation between exceedance events over different intervals. %It shows that the exceedance events over different intervals become asymptotically independent with exponential speed as the symmetric difference of the intervals gets larger.
Consider two intervals $K_1=(i_1,j_1)\in\III$ and $K_2=(i_2,j_2)\in\III$ such that $k_1:=j_1-i_1\in[l_n^-, l_n^+]$ and $k_2:=j_2-i_2\in [l_n^-,l_n^+]$. Let $K$ be the intersection of $K_1$ and $K_2$. Denote by $k\in\N_0$ the cardinality of $K$.  Assume that $k_1\leq k_2$ and let $\Delta=\Delta(K_1,K_2)=k_2-k$.
\begin{lemma}\label{lem:double_probab_supergauss}
There exist $C_1,C_2>0$ not depending on $n$, $\tau$ such that for all intervals $K_1$ and $K_2$ as above and all $|\tau|<\sqrt{\log n}$,
\begin{equation}\label{eq:lem_double_probab_main_superg}
\P\left[\ZZZ_{i_1,j_1}>u_n, \ZZZ_{i_2,j_2}>u_n\right] \leq  \frac{C_1 e^{-\tau}}{n\sqrt{\log n}} e^{-C_2 \Delta(K_1,K_2)}.
\end{equation}
\end{lemma}
\begin{proof}
%Write $\Delta=\Delta(K_1,K_2)$.
Fix $\eps>0$. The event $\{\ZZZ_{i_1,j_1}>u_n, \ZZZ_{i_2,j_2}>u_n\}$ is contained in the event $E_1\cup (E_2\cap E_3)$, where
\begin{align*}
E_1&=\left\{S_K>\sqrt k u_n+\eps \frac{\sqrt {k}\Delta} {u_n}\right\},\\
E_2&=\{S_{K_1}>\sqrt{k_1} u_n\},\\
E_3&=\left\{S_{K_2\bsl K}>(\sqrt{k_2}-\sqrt{k})u_n-\eps \frac{\sqrt {k}\Delta} {u_n}\right\}.
\end{align*}
We estimate $\P[E_1]$. Let first $ u_n\geq 2s_*k$. Using  Lemma~\ref{lem:ld_est}  and~\eqref{eq:s*_unique} we obtain that there is $\delta>0$ such that
$$
\P[E_1]
\leq
\exp\left\{- k I\left(\frac{u_n+\eps u_n^{-1}\Delta}{\sqrt k}\right)\right\}
\leq
\exp\left\{-\frac{1+\delta}{2m_*}(u_n^2+2\eps \Delta)\right\}.
%\leq
%C n^{-(1+\delta)} \exp\{-\eps \Delta/m_*\}.
$$
Now let $u_n\leq 2s_*k$. We have, by Theorem~\ref{theo:petrov_est} and~\eqref{eq:I_supergauss},
$$
\P[E_1]
%&=\P[\ZZZ_K>u_n+\eps u_n^{-1} \Delta]\\
\leq C u_n^{-1} \exp\left\{-k I\left(\frac{u_n+\eps u_n^{-1}\Delta}{\sqrt k}\right)\right\}
\leq \frac{C}{\sqrt{\log n}} \exp\left\{- \frac{(u_n^2+2\eps\Delta)}{2m_*} \right\}.
$$
Combining both cases we obtain that
%Using~\eqref{eq:I_supergauss} we obtain
\begin{equation}\label{eq:est_E1_superg}
\P[E_1]
%\frac{C}{\sqrt{\log n}} \exp\left\{- (u_n^2+2\eps\Delta)/(2m_*) \right\}
\leq \frac{Ce^{-\tau}}{n\sqrt{\log n}} e^{-\eps \Delta/m_*}.
\end{equation}
We estimate $\P[E_2]$. By Theorem~\ref{theo:petrov_est} and~\eqref{eq:I_supergauss},
\begin{equation}\label{eq:est_E2_superg}
\P[E_2]
\leq
 C u_n^{-1} \exp\left\{-k I\left(\frac{u_n}{\sqrt k}\right)\right\}
\leq
\frac{C}{\sqrt{\log n}} \exp\left\{- \frac{u_n^2}{2m_*} \right\}
\leq
\frac{Ce^{-\tau}}{n\sqrt{\log n}}.
\end{equation}
We estimate $\P[E_3]$. We have, using that $k_2,k\leq l_n^+$,
$$
\P[E_3]
=
\P\left[\frac{S_{\Delta}}{\Delta} > \frac{u_n}{\sqrt {k_2} + \sqrt k}- \eps \frac{\sqrt k} {u_n}\right]
\leq
\P\left[\frac{S_{\Delta}}{\Delta} > \frac{u_n}{2\sqrt{l_n^+}} - \eps \frac{\sqrt {l_n^+}} {u_n} \right].
$$
We can choose $\eps>0$ so small that $\frac{u_n}{2\sqrt{l_n^+}} - \eps \frac{\sqrt {l_n^+}} {u_n}>\eps$. It follows by Lemma~\ref{lem:ld_est} that
\begin{equation}\label{eq:est_E3_superg}
\P[E_3]\leq \P[S_{\Delta}>\eps \Delta]\leq e^{-\Delta I(\eps)}.
\end{equation}
Here, $I(\eps)>0$.
The probability on the left-hand side of~\eqref{eq:lem_double_probab_main_superg} is not larger than $\P[E_1]+ \P[E_2] \P[E_3]$ since  the events $E_2$ and $E_3$ are independent.  Combining~\eqref{eq:est_E1_superg}, \eqref{eq:est_E2_superg}, \eqref{eq:est_E3_superg} we obtain the required estimate.
\end{proof}

%\noindent
%Introduce the finite set
%\begin{equation}\label{eq:VVV_n}
%\VVV_n=\Z^2\cap ([0, w_n]\times [l_n^-, l_n^+]).
%%\{(x,y)\in\R^2: x\in [0, w_n],\; y-x \in [l_n^-, l_n^+]\}.
%\end{equation}
\begin{lemma}\label{lem:double_sum_supergauss}
Let $S_n''(B)$ be defined as in~\eqref{eq:Sn_B_2prime_subg} with $q_n=1$. We have
$$
\lim_{B\to\infty} \limsup_{n\to\infty} n w_n^{-1} S_n''(B) = 0.
$$
\end{lemma}
\begin{proof}
Introduce the finite set $\VVV_n=\Z^2\cap ([0, w_n]\times [l_n^-, l_n^+])$. Take some fixed $B\in\N$.
%For a vector $v=(x,y)\in\R^2$ we write $v'=(x+1,y)$.
Let $(x,y)\in B\Z^2$, $x<y$, be some interval.  The discrete square $\TTT_{B}(x,y)$ consists of $B^2$ intervals. We estimate the exceedance probability over $\TTT_{B}(x,y)$ by the sum of the exceedance probabilities over these intervals.
It follows from~\eqref{eq:Sn_B_2prime_subg} that
$$
S_n''(B)\leq \sum_{v_1,v_2} \P\left[\ZZZ_{v_1}>u_n, \ZZZ_{v_2}>u_n\right],
$$
where  the sum is taken over all $v_1\in \VVV_n$ and $v_2\in \VVV_n$ such that $v_1 \nsim v_2$.
Applying Lemmas~\ref{lem:double_probab_supergauss} and~\ref{lem:Delta_sup_norm} we get
$$
S_n''(B)\leq \frac{C}{n\sqrt{\log n}} \sum_{v_1,v_2} e^{-c_2 \Delta(v_1,v_2)}
\leq \frac{C}{n\sqrt{\log n}} \sum_{v_1,v_2} e^{-c_4 \|v_1-v_2\|}.
$$
Here, $\|\cdot\|$ is the sup-norm.
Applying to the right-hand side Lemma~\ref{lem:double_sum_general} and noting that $|\VVV_n|\leq C_1 w_n \sqrt{\log n}$ we arrive at the required statement.
\end{proof}

\subsection{Global probability}
Given $A_1<A_2$ recall that  $l_n^-=d_*\log n+ A_1\sqrt{\log n}$ and $l_n^+=d_*\log n+ A_2\sqrt{\log n}$.
Denote by $\III_n(A_1,A_2)$ the set of all intervals $(i,j)\in\III_n$ with length $l:=j-i\in[l_n^-, l_n^+]$. Theorem~\ref{theo:main_supergauss_scales} states that
\begin{equation}\label{eq:global_probab_superg}
\lim_{n\to\infty}\P\left[\max_{(i,j)\in \III_n(A_1,A_2)} \ZZZ_{i,j}\leq u_n\right]= \exp\left\{- e^{-\tau} \int_{A_1}^{A_2} \Theta(a)da \right\}.
\end{equation}
The proof of~\eqref{eq:global_probab_superg} goes as follows. Let $w_n=[3d_*\log n]$. We decompose the set $\III_n(A_1,A_2)$ into $\sim n/w_n$ sets of the form $\JJJ_n(z)$, $z\in w_n\Z$.  The exceedance  probability over any of these sets is asymptotically equivalent to $e^{-\tau} (\int_{A_1}^{A_2} \Theta(a)da)w_n/n$ by Lemma~\ref{lem:exc_JJJ_supergauss}. Also, the exceedance event over $\JJJ_n(z)$ is independent of all other exceedance events except for $\JJJ_n(z\pm w_n)$. Justifying the use of the Poisson limit theorem we obtain~\eqref{eq:global_probab_superg}. The proof, up to trivial changes, is the same as in Section~\ref{subsec:global_subg}.

\subsection{Non-optimal lengths}\label{subsec:non_optimal_superg}
In this section we complete the proof of Theorem~\ref{theo:main_supergauss}. For $A>0$ write $l_n^-=d_*\log n -A\sqrt {\log n}$ and $l_n^+=d_*\log n + A\sqrt {\log n}$. Denote by $\BBB_n(A)=\III_n\bsl \III_n(-A,A)$ the set of all intervals $(i,j)\in\III_n$ whose length $l:=j-i$ satisfies $l\notin[l_n^-, l_n^+]$. The aim of this section is to show that the contribution of these non-optimal lengths to $\MMM_n$ is negligible, if $A$ is large. More precisely, we will show that
\begin{equation}\label{eq:nonoptimal_supergauss}
\lim_{A\to\infty} \limsup_{n\to\infty}
\P\left[\max_{(i,j) \in \BBB_n(A)} \ZZZ_{i,j} >u_n\right]=0.
\end{equation}
Combined with Theorem~\ref{theo:main_supergauss_scales} proved above this yields Theorem~\ref{theo:main_supergauss}. We will cover the set $\BBB_n(A)$ by three sets $\BBB_n', \BBB_n'', \BBB_n'''$ (depending on some further parameters) which will be considered separately in the next three lemmas. In this section we don't need the non-lattice assumption.
First we consider intervals which are sufficiently small but not close to the optimal length $d_*\log n$.
\begin{lemma}
Let $0<\eps<d_*$ be fixed.  For $\delta>0$ denote by $\BBB'_n=\BBB_n'(\eps,\delta)$ the set of all intervals $(i,j)\in \III_n$ whose length $l=j-i$ satisfies $l\leq n^{\delta}$ and $|l-d_*\log n|>\eps \log n$. Then, we can choose  $\delta>0$ so small that
$$
\lim_{n\to\infty} \P\left[\max_{(i,j) \in \BBB_n'} \ZZZ_{i,j} >u_n\right]=0.
$$
\end{lemma}
\begin{proof}
Let $l\in\N$ be such that $l\leq n^{\delta}$ and $|l-d_*\log n|>\eps \log n$.
%Write $s=u_n/\sqrt l$.
Then, we can find an $\eps_1>0$ (depending on $\eps$, but not on $\delta, l,n$) such that  $|\frac{u_n}{\sqrt l}-s_*|>\eps_1$. By~\eqref{eq:s*_unique} there are $\delta_1,\delta_2>0$ such that for all large $n$ and all $l$ as above,
$$
l I\left(\frac{u_n}{\sqrt l}\right) = \frac{I(u_n/\sqrt l)}{(u_n/\sqrt l)^2/2} \cdot \frac{u_n^2}{2}>\left(\frac 1 {m_*}+\delta_1\right) \cdot \frac{u_n^2}{2}>(1+\delta_2)\log n.
$$
Using Lemma~\ref{lem:ld_est} we obtain that for all $(i,j)\in \BBB_n'$, the individual exceedance probability can be estimated as follows:
$$
\P[\ZZZ_{i,j}>u_n] \leq e^{-l I(u_n/\sqrt l)} <  n^{-(1+\delta_2)}.
$$
Choose any $\delta\in (0, \delta_2)$. Since the number of intervals in $\BBB_n'$ is at most $n^{1+\delta}$,  the statement of the lemma follows.
\end{proof}
Next we consider intervals whose length is close to being optimal.
\begin{lemma}
For $\eps>0$ and $A>0$ let $\BBB_n''=\BBB_n''(A)$ be the set of all intervals $(i,j)\in\III_n$ whose length $l=j-i$ satisfies
$(d_*-\eps)\log n \leq l \leq l_n^-$ or  $l_n^+ \leq l \leq (d_*+\eps)\log n$.
%(Recall that $l_n^-$ and $l_n^+$ depend on $A$).
Then, we can choose $\eps>0$ so small that
$$
\lim_{A\to\infty} \limsup_{n\to\infty}
\P\left[\max_{(i,j) \in \BBB_n''(A)} \ZZZ_{i,j} >u_n\right]=0.
$$
\end{lemma}
\begin{proof}
It follows from~\eqref{eq:J_derivatives} that we can choose $\eps,\delta>0$ so small that for all $v\in (v_*-\eps, v_*+\eps)$,
\begin{equation}\label{eq:est_J_quadr}
J(v) > \frac{1}{2m_*} + \delta (v-v_*)^2.
\end{equation}
Here, $v_*=s_*^{-2}$.  For $k\in\N_0$ let $\BBB_{n,k}''$ be the set of all intervals $(i,j)\in\III_n$ with length $l$ satisfying
$$
d_*\log n+2^{k}\sqrt{\log n}\leq l \leq  d_*\log n+2^{k+1}\sqrt{\log n}.
$$
It follows from~\eqref{eq:est_J_quadr} that for every $(i,j)\in \BBB_{n,k}''$,
$$
l I\left(\frac{u_n}{\sqrt l}\right)
=
u_n^2 J\left(\frac{l}{u_n^2}\right)
>
\frac{u_n^2}{2m_*}+ \frac {\delta} {u_n^{2}}\left(d_*\log n +2^{k}\sqrt {\log n}-\frac{u_n^2}{s_*^2}\right)^2
>
\log n + c_1 2^{2k}.
$$
By Theorem~\ref{theo:petrov_est} we obtain
$$
\P\left[\frac{S_l}{\sqrt l}>u_n\right]
\leq
c_2 u_n^{-1} \exp \left\{- l I\left(\frac{u_n}{\sqrt l}\right)\right\}
\leq
\frac{c_3}{n\sqrt {\log n}} e^{-c_4 2^{2k}}.
$$
Let $\BBB''_{n,+}(A)$ be the set of all intervals $(i,j)\in\III_n$ such that $l_n^+ \leq l \leq (d_*+\eps)\log n$. The number of intervals in $\BBB_{n,k}''$ is at most $c_5 2^k n\sqrt {\log n}$.  For the exceedance probability over the set $\BBB''_{n,+}(A)$ we obtain the estimate
$$
\P\left[\max_{(i,j) \in \BBB_{n,+}''(A)} \ZZZ_{i,j}>u_n\right] \leq c_6 \sum_{k=[\log_2 A]}^{\infty} 2^k e^{-c_4 2^{2k}}.
$$
The right-hand side goes to $0$ as $A\to\infty$. Exceedance probability over the set $\BBB_{n,-}''(A)$ consisting of all intervals $(i,j)\in\III_n$ with length satisfying $(d_*-\eps)\log n \leq l \leq l_n^-$ can be estimated analogously.
\end{proof}
\begin{lemma}
Let $\delta>0$ be arbitrary. Let $\BBB_n'''$ be the set of all intervals $(i,j)\in\III_n$ whose length $l=j-i$ satisfies $l\geq \log^{1+\delta} n$. Then,
$$
\lim_{n\to\infty} \P\left[\max_{(i,j) \in \BBB_n'''} \ZZZ_{i,j} >u_n\right]=0.
$$
\end{lemma}
\begin{proof}
For $k\in\N$ such that $2^k>\log n$ let $\BBB_{n,k}'''$ be the set of all intervals $(i,j)\in\III_n$ with length $l$ satisfying $2^k\leq l\leq 2^{k+1}$. We can cover this set by discrete squares of the form $\TTT_{r_k}(x, x+l)$, see Section~\ref{subsec:local_supergauss}, where $r_k=[2^k/\log n]$ and $x,l\in r_k\Z$. We need at most $c2^{-k} n \log^2 n$ such squares. The exceedance probability over each such square can be estimated using the same method as in Lemma~\ref{lem:local_est_sub_gauss}.
For every $\eps>0$ there is $\eta>0$ such that $I(y)\geq (1-\eps) y^2/2$ for all $y\in [0,\eta]$. Recall that $l\geq \log^{1+\delta} n$ and hence,  $u_n/\sqrt l \leq \eta$ for all sufficiently large $n$. By Lemma~\ref{lem:ld_est}, for every $s\in [0,u_n^2]$,
$$
F_{l,u_n}(s):=\P\left[\frac{S_l}{\sqrt l}>u_n-\frac{s}{u_n}\right]\leq \exp\left\{- l I\left(\frac{u_n-\frac{s}{u_n}}{\sqrt l}\right) \right\}
\leq
e^s \exp\left\{-\frac{(1-\eps)u_n^2}2\right\}.
$$
This inequality continues to hold for $s\geq u_n^2$, since in this case the right-hand side is greater than $1$. Arguing in the same way as in the proof of Lemma~\ref{lem:local_est_sub_gauss}, but replacing~\eqref{eq:asympt_F_lu} by the above inequality, we obtain
$$
\P\left[\max_{(i,j)\in \TTT_{r_k}(x,x+l)} \ZZZ_{i,j}>u_n\right] \leq  C\exp\left\{-\frac{(1-\eps)u_n^2}2\right\}
\leq
Cn^{-(1+\eps)}.
$$
For the exceedance probability over the set $\BBB_{n,k}'''$ we obtain
$$
\P\left[\max_{(i,j) \in \BBB_{n,k}'''} \ZZZ_{i,j} >u_n\right]\leq   C 2^{-k} n^{-\frac {\eps} 2}.
$$
To complete the proof, take the sum over all $k\in\N$.
% By the strong approximation theorem of Komlos--Major--Tusnady, we can construct, on some probability space,  the random variables $X_1,X_2,\ldots$  %together with standard Gaussian random variables $Y_1,Y_2,\ldots$ with partial sums $W_k=Y_1+\ldots+Y_k$, $k\in\N$, such that
%\begin{equation}\label{eq:kmt}
%\lim_{n\to\infty} \P\left[\max_{k=0,\ldots,n} |S_k-W_k| > C \log n\right]=0.
%\end{equation}
%For $0\leq A_1 < A_2\leq n$ recall the definition of $\MMM_n(A_1,A_2)$ in~\eqref{eq:def_Ln_A1A2} and write
%$$
%R_n(A_1,A_2)=\max_{\substack{0\leq i < j\leq n\\ A_1\leq j-i\leq A_2}} \frac{W_j-W_i}{\sqrt {j-i}}.
%$$
%Let $l_n=\log^{1+\delta} n$. It follows from~\eqref{eq:kmt} that
%$$
%\lim_{n\to\infty} \P\left[|\MMM_n(l_n,n) - R_n(l_n,n)| > 2C \log^{\frac {1-\delta}2} n\right] = 0.
%$$
%On the other hand, by~\cite{shao} (one could also use~\cite{siegmund_venkatraman} or~\cite{kabluchko_unpub_07}), for every $\delta>0$,
%$$
%\P[R_n(l_n,n)> \sqrt{(2+\delta)\log n}]\leq \P[R_n(1,n)> \sqrt{(2+\delta)\log n}]\to 0, \;\;\, n\to\infty.
%$$
%It follows that
%$$
%\lim_{n\to\infty} \P[\MMM_n(l_n,n)> \sqrt{(2+2\delta)\log n}]=0.
%$$
%Recall that $u_n\sim \sqrt{2m_*\log n}$, where $m_*>1$. Choosing $\delta>0$ so small that $2+2\delta<2m_*$, we obtain the statement of the lemma.
\end{proof}

\subsection{Tightness  in the lattice case}
Now we allow the distribution of $X_1$ to be lattice and prove Theorem~\ref{theo:main_supergauss_lattice}.  The tightness of the sequence $\MMM^2_n-2m_*\log n$ follows from Lemmas~\ref{lem:tight_lattice_1} and~\ref{lem:tight_lattice_2} below.
We use the same notation as in Section~\ref{subsec:non_optimal_superg}. Namely, for $A>0$ we write $l_n^-=d_*\log n-A\sqrt{\log n}$ and $l_n^+=d_*\log n+A\sqrt{\log n}$.   Let $\III_n(A)$ be the set of all intervals $(i,j)\in\III_n$ with length $l=j-i$ satisfying $l\in [l_n^-, l_n^+]$. Write $\BBB_n(A)=\III_n\bsl \III_n(A)$.
Let $u_n=u_n(\tau)>0$ be defined by
\begin{equation}
u_n^2(\tau)=2m_*(\log n+\tau),\;\;\; \tau\in\R.
\end{equation}
\begin{lemma}\label{lem:tight_lattice_1}
For every $\eps>0$ we can find $\tau=\tau(\eps)$ such that $\P[\MMM_n>u_n(\tau)]<\eps$ for all large $n$.
\end{lemma}
\begin{proof}
Fix $\eps>0$. By~\eqref{eq:nonoptimal_supergauss} we can find $A=A(\eps)$ such that  for large $n$,
\begin{equation}\label{eq:tight_latt_tech1}
\P\left[\max_{(i,j)\in\BBB_n(A)}\ZZZ_{i,j}>u_n(0)\right]<\frac{\eps} 2.
\end{equation}
We estimate the exceedance probability over $\III_n(A)$. Let $\tau>0$. For every interval $(i,j)\in\III_n(A)$ with length $l\in [l_n^-, l_n^+]$ we have, by Theorem~\ref{theo:petrov_est} and~\eqref{eq:s*_unique},
$$
\P[\ZZZ_{i,j}>u_n(\tau)]
\leq
C u_n^{-1}\exp\left\{-l I\left(\frac{u_n}{\sqrt l}\right)\right\}
\leq
\frac C{\sqrt{\log n}} \exp\left\{-\frac{u_n^2}{2m_*}\right\}
=
\frac {Ce^{-\tau}}{n\sqrt{\log n}}.
$$
Since the number of elements in $\III_n(A)$ is at most $4A n \sqrt{\log n}$, we obtain that there is $C_1$ depending only on $A$ such that
\begin{equation}\label{eq:tight_latt_tech2}
\P\left[\max_{(i,j)\in\III_n(A)}\ZZZ_{i,j} > u_n(\tau)\right]\leq C_1 e^{-\tau},
\end{equation}
We can choose $\tau>0$ so large that $C_1 e^{-\tau}< \eps/2$. To complete the proof combine~\eqref{eq:tight_latt_tech1} and~\eqref{eq:tight_latt_tech2}.
\end{proof}

We now give a lower estimate for the exceedance probability. Let $w_n=[3d_*\log n]$. Take $A=1$. Define $\JJJ_n(z)$ as in Lemma~\ref{lem:exc_JJJ_supergauss}.
\begin{lemma}\label{lem:exc_JJJ_lattice_supergauss}
There is  a constant $C$ such that for all $z\in\Z$, large $n\in\N$, and all $|\tau|<\sqrt{\log n}$,
$$
\P\left[\max_{(i,j)\in \JJJ_n(z)} \ZZZ_{i,j}>u_n(\tau)\right] \geq \frac{C e^{-\tau}}{n/w_n}.
$$
\end{lemma}
\begin{proof}
Without restriction of generality let $z=0$.
%To get rid of the boundary effects we introduce two sequences $\eps_n\to \infty$ and $\delta_n\to\infty$ such that $\eps_n=o(w_n)$ and $\delta_n=o(\sqrt{\log n})$ as $n\to\infty$.
Take some $B\in\N$. Let $\calJ_n(B)$ be a two-dimensional discrete grid with mesh size $B$ defined by
$$
\calJ_n(B)=B\Z^2\cap \JJJ_n(0)=\{(i,j)\in B\Z^2:  i\in [0, w_n], j-i\in [l_n^-, l_n^+]\}.
$$
Then, by the Bonferroni inequality,
\begin{equation}\label{eq:tech_bonferr}
\P\left[\max_{(i,j)\in \JJJ_n(0)} \ZZZ_{i,j}>u_n(\tau)\right] \geq S_n'(B)-S_n''(B),
\end{equation}
where $S_n'(B)$ and $S_n''(B)$ are defined by
\begin{align*}
S_n'(B)=\sum_{(i,j)\in \calJ_n(B)} \P[\ZZZ_{i,j}>u_n(\tau)],\;\;\;
S_n''(B)=\sum_{} \P[\ZZZ_{v_1}>u_n(\tau), \ZZZ_{v_2}>u_n(\tau)],
\end{align*}
and the second sum is taken over all $v_1\neq v_2$ with $v_1=(i_1,j_1)\in \calJ_n(B)$ and $v_2=(i_2,j_2)\in \calJ_n(B)$.
We estimate $S_n''(B)$ first. By Lemma~\ref{lem:double_probab_supergauss} and Lemma~\ref{lem:Delta_sup_norm},
$$
\P[\ZZZ_{v_1}>u_n(\tau), \ZZZ_{v_2}>u_n(\tau)]\leq \frac{Ce^{-\tau}}{n\sqrt{\log n}} e^{-c\|v_1-v_2\|}.
$$
%We have $c_1w_n\sqrt {\log n}\leq |\calJ_n(B)|\leq c_2w_n\sqrt {\log n}$.
It follows that
\begin{equation}\label{eq:tech_Sn''B}
S_n''(B)
\leq
\frac{C_1e^{-\tau}}{n\sqrt{\log n}} \sum e^{-c\|v_1-v_2\|}
\leq
\frac{C_1e^{-\tau}}{n\sqrt{\log n}} |\calJ_n(B)| \sum_{v\in B\Z^2} e^{-c\|v\|}.
\end{equation}
%where $C_B=\sum_{v\in \Z^2} e^{-c\|v\|}$.
Now we estimate $S_n'(B)$. By Theorem~6 of~\cite{petrov_ld} (which is a converse inequality to Theorem~\ref{theo:petrov_est}),
$$
\P[\ZZZ_{i,j}>u_n(\tau)]
\geq
C_2 u_n^{-1}\exp\left\{-l I\left(\frac{u_n}{\sqrt l}\right)\right\}
\geq
%\frac C{\sqrt{\log n}} \exp\left\{-\frac{u_n^2}{2m_*}\right\}
%=
\frac {C_3e^{-\tau}}{n\sqrt{\log n}}.
$$
Hence,
\begin{equation}\label{eq:tech_Sn'B}
S_n'(B)\geq |\calJ_n(B)| \frac {C_3e^{-\tau}}{n\sqrt{\log n}}.
\end{equation}
We can choose $B$ so large that $C_1\sum_{v\in B\Z^2} e^{-c\|v\|} < C_3/2$. Taking~\eqref{eq:tech_bonferr}, \eqref{eq:tech_Sn''B}, \eqref{eq:tech_Sn'B} together and noting that $|\calJ_n(B)|>c w_n\sqrt{\log n}$  yields the statement of the lemma.
\end{proof}

\begin{lemma}\label{lem:tight_lattice_2}
For every $\eps>0$ we can find $\tau=\tau(\eps)$ (sufficiently close to $-\infty$) such that $\P[\MMM_n < u_n(\tau)]<\eps$ for all large $n$.
\end{lemma}
\begin{proof}
 Consider the sets $\JJJ_n(z)$, where $z\in 2w_n\Z$. There are at least $n/(3w_n)$ such sets contained in $\III_n$.  The exceedance events over these sets are independent, hence,
\begin{align*}
%\P\left[\max_{(i,j)\in \III_n(1)} \ZZZ_{i,j}< u_n(\tau)\right]
\P[\MMM_n < u_n(\tau)]
\leq
\left(1-\P\left[\max_{(i,j)\in\JJJ_n(0)}\ZZZ_{i,j}>u_n(\tau)\right]\right)^{\frac n{3w_n}}
\leq
\left(1-\frac{C e^{-\tau}}{n/w_n}\right)^{\frac n{3w_n}}.
\end{align*}
The right hand-side converges to $\exp\{-\frac 13 Ce^{-\tau}\}$, as $n\to\infty$.  It follows that we can choose $\tau$ so close to $-\infty$ that for all large $n$, the right-hand side  is smaller than $\eps$.
%Then, we have
%$$
%\P[\MMM_n < u_n(\tau)]\leq \P\left[\max_{(i,j)\in \III_n(1)}\ZZZ_{i,j} < u_n(\tau)\right] < \eps.
%$$
The proof is complete.
\end{proof}

\subsection{Pickands-type constant}\label{subsec:pickands_type_const}
In this section we provide two alternative expressions for the Pickands-type constant $H_*$; see~\eqref{eq:beta*_H*2}.  Let $Y_1,Y_2,\ldots$ be non-degenerate i.i.d.\ random variables such that $\E e^{Y_k}=1$, $k\in\N$.  Independently, let also $Y_{-1},Y_{-2},\ldots$ be i.i.d.\ random variables such that
\begin{equation}
\P[Y_{-k}\in dy]= e^y \P[Y_k\in dy],\;\;\; k\in\N.
\end{equation}
Note that $\E e^{-Y_{-k}}=1$, $k\in\N$. Define a stochastic process $\{W_k, k\in\Z\}$ by $W_0=0$ and
\begin{equation}
W_{k}=Y_1+\ldots+Y_k,
\;\;\;
W_{-k}=Y_{-1}+\ldots+Y_{-k},
\;\;\;
k\in \N.
\end{equation}
\begin{lemma}\label{lem:pickands_const_supergauss}
Let $L_N=\max_{k=0,\ldots,N} W_k$, $k\in\N$. Then,
\begin{equation}\label{eq:pickands_const_supergauss}
\lim_{N\to\infty} \frac 1 N \E e^{L_N} = \P\left[\forall k\in\N: W_k<0, W_{-k}\leq 0\right].
%R_+ R_-\in (0,\infty),
\end{equation}
%where $R_+=\P[\max_{k\in\N} W_k<0]$ and $R_-=\P[\max_{k\in\N} W_{-k}\leq 0]$.
%\begin{align*}
%R_+&=\P[\max_{k\in\N} W_k<0]=\exp\left\{-\sum_{k=1}^{\infty} \frac 1k \P[W_k>0]\right\},\\
%R_-&=\P[\max_{k\in\N} W_{-k}\leq 0]=\exp\left\{-\sum_{k=1}^{\infty} \frac 1k \P[W_{-k}\geq 0]\right\}.
%\end{align*}
\end{lemma}
\begin{remark}\label{rem:spitzer}
%By~\cite{spitzer_book} it follows that $\lim_{N\to\infty} \frac 1 N \E e^{L_N}=R_-R_+$, where
By~\cite{spitzer_book} the probability on the right-hand side of~\eqref{eq:pickands_const_supergauss} is equal to $R_-R_+$, where
\begin{equation}
R_+=\exp\left\{-\sum_{k=1}^{\infty} \frac 1k \P[W_k>0]\right\},\;\;\;
R_-=\exp\left\{-\sum_{k=1}^{\infty} \frac 1k \P[W_{-k}\geq 0]\right\}.
\end{equation}
\end{remark}
\begin{remark}
By taking $Y_k=t_*X_k-\varphi(t_*)$ we obtain alternative expressions for the constant $H_*$ defined in~\eqref{eq:beta*_H*2}. It follows from~\eqref{eq:pickands_const_supergauss} that $0\leq H_*\leq 1$. Since $\E Y_k<0$, $k\in\Z$, we have $\lim_{k\to\pm \infty} W_k=-\infty$ a.s.\ by the law of large numbers. Therefore, we even have strict inequalities $0<H_*<1$.
\end{remark}
\begin{proof}[Proof of Lemma~\ref{lem:pickands_const_supergauss}]
Let $g(w)=\sum_{N=0}^{\infty} w^N \E e^{L_N}$, $|w|<1$. Then, by p.\ 207 of~\cite{spitzer_book},
\begin{align*}
g(w)
&=
\frac{1}{1-w} \exp\left\{- \sum_{k=1}^{\infty}\frac{w^k}{k} \E [(1-e^{W_k})\ind_{W_k>0}]\right\}\\
&=
\frac{1}{(1-w)^2} \exp\left\{- \sum_{k=1}^{\infty}\frac{w^k}{k} (\P[W_k>0]+\P[W_{-k}\leq 0])\right\}.
\end{align*}
It follows that $g(w)\sim R_+R_-/(1-w)^2$, as $w\uparrow 1$. The sequence $\E e^{L_N}$ is non-decreasing.  By the Hardy--Littlewood Tauberian theorem, see Corollary~1.7.3 on p.~40 in~\cite{bingham_book}, it follows that
$
\lim_{N\to\infty} \frac 1N \E e^{L_N} = R_+R_-. %\exp\left\{- \sum_{k=1}^{\infty}\frac{1}{k} (\P[W_k>0]+\P[W_{-k}\leq 0])\right\}.
$
In view of  Remark~\ref{rem:spitzer} this completes the proof.
\end{proof}
%\subsection*{Acknowledgment}

%********************************************************%
%************* PROOF SUBLOGARITHMIC *********************%
%********************************************************%

\section{Proof in the sublogarithmic case}\label{sec:proof_exp}

\subsection{Proof of Theorem~\ref{theo:main_exp}}
Let $\UUU_n=\max\{X_1,\ldots,X_n\}$. Take any $\frac 12 < \beta < \min (\frac{1}{\alpha},1)$. By the assumption of the theorem, we have
$$
n\P[X_1>\log^{\beta} n] > ne^{-\log^{\alpha \beta} n}\to\infty, \;\;\; n\to\infty,
$$
since $\alpha\beta < 1$. It follows that
$$
\lim_{n\to\infty} \P[\UUU_n\leq \log^{\beta} n]=0.
$$
The proof will be complete after we have shown that
\begin{equation}\label{eq:proof_sublog_tech1}
\lim_{n\to\infty} \P[\MMM_n(a\log n, n) \leq \log^{\beta} n]=1.
\end{equation}
Recall the definition of $I$ in~\eqref{eq:def_I}. Since $I$ is a convex function we can find $s_0>0$  such that $I(s)>3/a$ for all $s>s_0$.  For every interval $(i,j)\in\III_n$ with length $l=j-i$ such that $a\log n\leq l\leq s_0^{-2} \log^{2\beta}n$ we have, by Lemma~\ref{lem:ld_est},
$$
\P[\ZZZ_{i,j}>\log^{\beta} n] \leq \exp\left\{- l I\left(\frac{\log^{\beta} n}{\sqrt l}\right) \right\}
\leq \exp\left\{-\frac{3l} a\right\}\leq \frac 1 {n^3}.
$$
Since $I(s)\sim \frac {s^2}{2}$ as $s\downarrow 0$, we can find $c>0$ such that $I(s)>cs^2$ for all $s\in [0, s_0]$. It follows that for every interval $(i,j)\in \III_n$ of length $l\geq s_0^{-2} \log^{2\beta} n$,
$$
\P[\ZZZ_{i,j}>\log^{\beta} n]
\leq
\exp\left\{- l I\left(\frac{\log^{\beta} n}{\sqrt l}\right) \right\}
\leq
\exp\{-c \log^{2\beta} n \}
\leq
\frac 1 {n^3},
$$
where we have used that $\beta>\frac 12$. Since the number of intervals in $\III_n$ is at most $n^2$ it follows that~\eqref{eq:proof_sublog_tech1} holds.

\subsection{Proof of Theorem~\ref{theo:main_exp_regular}}
Choose $v\in (0,1)$ such that $2^{1-\frac {\alpha}2} v^{1+\alpha}>1$ (recall that $\alpha<2$). By assumption~\eqref{eq:exponenntial_tail} we have
$$
n\P\left[X_1> v\left(\frac{\log n}{D}\right)^{1/\alpha} \right] = n \exp\{-(v^{\alpha}+o(1)) \log n\} \to\infty, \;\;\; n\to\infty.
$$
It follows that the maximum $\UUU_n=\max\{X_1,\ldots,X_n\}$ satisfies
$$
\lim_{n\to\infty}\P\left[\UUU_n \leq  v \left(\frac{\log n}{D}\right)^{1/\alpha} \right]=0.
$$
In view of Theorem~\ref{theo:main_exp} the proof of Theorem~\ref{theo:main_exp_regular} will be complete after we have shown that
\begin{equation}\label{eq:exp_proof1}
\lim_{n\to\infty}\P\left[\MMM_n (2, \log n) \geq  v \left(\frac{\log n}{D}\right)^{1/\alpha} \right]=0.
\end{equation}
Assume first that $\alpha=1$. Then, condition~\eqref{eq:exponenntial_tail} implies that $\varphi(t)=\log \E e^{tX_1}$ is finite for $t\in [0,D)$, and equal to $+\infty$ for $t>D$. This implies that
$$
I(s)\sim Ds \text{ as } s\to\infty.
$$
Consider now the case $\alpha\in (1,2)$. By Kasahara's theorem~\cite[p.~253]{bingham_book}, condition~\eqref{eq:exponenntial_tail} is equivalent to  $\varphi(t)\sim G t^{\beta}$ as  $t\to\infty$
where $\frac 1 {\alpha}+\frac 1 {\beta}=1$ and $G^{1-\alpha}\beta^{-\alpha} (1-\beta)=D$. For the Legendre--Fenchel conjugate, one obtains~\cite[p.~48]{bingham_book}
$$
I(s)\sim Ds^{\alpha}  \text{ as } s\to\infty.
$$
Hence, both for $\alpha=1$ and for $\alpha\in (1,2)$ we have $I(s)> vDs^{\alpha}$ for large $s$. By Lemma~\ref{lem:ld_est}, for every interval $(i,j)\in \III_n$ of length $2\leq l\leq \log n$, we have, for large $n$,
$$
\P\left[\ZZZ_{i,j}>v \left(\frac{\log n}{D}\right)^{1/\alpha}\right]
\leq
\exp\left\{-l I\left(\frac {v}{\sqrt l} \left(\frac{\log n}{D}\right)^{1/\alpha}\right)\right\}
\leq
\exp\left\{- 2^{1-\frac {\alpha}2} v^{1+\alpha}  \log n\right\}.
$$
%$$
%\P\left[\ZZZ_{i,j}>v \frac{\log n}{D}\right]
%\leq
%\exp\left\{-l I\left(v\frac{\log n}{\sqrt l D}\right)\right\}
%\leq
%\exp\left\{-v^2 \sqrt 2 \log n\right\}.
%$$
Recall that $2^{1-\frac {\alpha}2} v^{1+\alpha}>1$. Since the number of intervals in $\III_n$ with length not exceeding $\log n$ is at most $n\log n$, we obtain~\eqref{eq:exp_proof1}.

\section*{Acknowledgement} Zakhar Kabluchko is grateful to Axel Munk from whom he learned about multiscale scan statistics.

\bibliographystyle{elsarticle-harv}
\bibliography{stand_incr_bib}

\end{document}